\documentclass[reqno, 11pt]{amsart}
 \usepackage{color,mathrsfs}
 \usepackage{appendix}
\usepackage[utf8]{inputenc}
\usepackage[colorlinks=true, pdfstartview=FitV, linkcolor=blue, citecolor=blue, urlcolor=blue,pagebackref=false,hyperindex,breaklinks]{hyperref}
\usepackage{calc,amsfonts,amsthm,amscd,epsfig,psfrag,amsmath,amssymb,enumerate,graphicx}
\usepackage[showlabels,sections,floats,textmath,displaymath]{}

 \renewcommand{\thesubsection}{\arabic{section}.\arabic{subsection}} 

\reversemarginpar
\newlength\fullwidth
\numberwithin{equation}{section}

\DeclareMathSymbol{\leqslant}{\mathalpha}{AMSa}{"36} 
\DeclareMathSymbol{\geqslant}{\mathalpha}{AMSa}{"3E} 
\DeclareMathSymbol{\eset}{\mathalpha}{AMSb}{"3F}     
\renewcommand{\leq}{\;\leqslant\;}                   
\renewcommand{\geq}{\;\geqslant\;}                   
 \def\1{\ifmmode {1\hskip -3pt
    \rm{I}} \else {\hbox {$1\hskip -3pt \rm{I}$}}\fi}

\newcommand{\var}{\operatorname{Var}} 
 
\newcommand{\argmax}{\operatorname{argmax}} 
 
\newcommand{\sign}{{\rm sg}}

\newcommand{\R}{\mathbb{R}}
\newcommand{\Z}{\mathbb{Z}}

\renewcommand{\O}{\Omega}

\newtheorem{theorem}{Theorem}[section] 
\newtheorem{assumption}{Assumption}[section] 
\newtheorem{lemma}[theorem]{Lemma} 
\newtheorem{proposition}[theorem]{Proposition} 
\newtheorem{rem}[theorem]{Remark} 
\newtheorem{corollary}[theorem]{Corollary}

\newtheorem{definition}[theorem]{Definition}
 
\newcommand{\N}{\mathbb N}
\newcommand{\bP}{{\bf P}} 

\newcommand{\bff}{{\bf f}}

\newcommand{\bbE}{{\ensuremath{\mathbb E}} }

\newcommand{\bbL}{{\ensuremath{\mathbb L}} } 
 
\newcommand{\bbN}{{\ensuremath{\mathbb N}} } 
 
\newcommand{\bbP}{{\ensuremath{\mathbb P}} } 
 
\newcommand{\bbR}{{\ensuremath{\mathbb R}} }

\newcommand{\bbZ}{{\ensuremath{\mathbb Z}} }

\newcommand{\gep}{\varepsilon}

\newcommand{\gO}{\Omega}
\newcommand{\ind}{{\bf 1}}
\newcommand{\gd}{\delta}
\newcommand{\dd}{\mathrm{d}}

\newcommand{\gl}{\lambda}

\newcommand{\DD}{\mathcal{D}}
\newcommand{\DDD}{\mathscr{D}}

\newcommand{\8}{\infty}

\newcommand{\gL}{\Lambda}
\newcommand{\gD}{\Delta}

\newcommand{\be}{\mathbf{e}}
\newcommand{\bof}{\mathbf{f}}

\newcommand{\mal}{\mathcal A_L}

\begin{document}

\title[Curve-shortening evolution for the 2D Ising Model]{Zero-temperature 2D stochastic Ising model  and anisotropic 
curve-shortening flow}

\author[H. Lacoin]{Hubert Lacoin}
\address{H. Lacoin, 
CEREMADE - UMR CNRS 7534 - Universit\'e Paris Dauphine,
Place du Mar\'echal de Lattre de Tassigny, 75775 CEDEX-16 Paris, France. \newline
e--mail: {\tt lacoin@ceremade.dauphine.fr}}

\author[F. Simenhaus]{Fran\c{c}ois Simenhaus}
\address{F. Simenhaus, 
CEREMADE - Universit\'e Paris Dauphine - UMR  CNRS 7534,
Place du Mar\'echal de Lattre de Tassigny, 75775 CEDEX-16 Paris France. \newline
e--mail: {\tt simenhaus@ceremade.dauphine.fr}}
\author[F. L. Toninelli]{Fabio Lucio Toninelli} 
\address{F. L. Toninelli, CNRS and Universit\'e Lyon 1, 
Institut Camille Jordan,\newline
  43 bd du 11 novembre 1918, 69622 Villeurbanne, France.  \newline e--mail: {\tt
   toninelli@math.univ-lyon1.fr}}

\begin{abstract}
  Let $\DD$ be a simply connected, smooth enough domain of $\bbR^2$.
  For $L>0$ consider the continuous time, zero-temperature heat bath
  dynamics for the nearest-neighbor Ising model on $\mathbb Z^2$ with
  initial condition such that $\sigma_x=-1$ if $x\in L\DD$ and
  $\sigma_x=+1$ otherwise. It is conjectured \cite{cf:Spohn} that, in
  the diffusive limit where space is rescaled by $L$, time by $L^2$
  and $L\to\infty$, the boundary of the droplet of ``$-$'' spins
  follows a \emph{deterministic} anisotropic curve-shortening flow,
  where the normal velocity at a point of its boundary is given by the
  local curvature times an explicit function of the local slope.  The
  behavior should be similar at finite temperature $T<T_c$, with a
  different temperature-dependent anisotropy function.

  We prove this conjecture (at zero temperature) when $\DD$ is convex.
  Existence and regularity of the solution of the deterministic
  curve-shortening flow is not obvious \textit{a priori} and is part of our
  result.  To our knowledge, this is the first proof of mean
  curvature-type droplet shrinking for a model with genuine
  microscopic dynamics.

\end{abstract}

\thanks{F. T. was partially supported by European Research Council through the
“Advanced Grant” PTRELSS 228032 and by ANR project SHEPI}

\maketitle
\tableofcontents
\thispagestyle{empty}

\section{Introduction}

Consider a thermodynamic system with two coexisting phases and imagine
to prepare it in an initial condition where a droplet of one phase is
immersed in the other phase. If the system undergoes a dynamics that
does not conserve the order parameter, it is well understood
phenomenologically \cite{cf:Lifshitz} that the droplet will shrink in
order to decrease its surface tension until it eventually disappears, and that (roughly speaking) the
normal speed at a point of its boundary will be proportional to the
local mean curvature.  Deriving such behavior from first principles,
i.e.\ from a microscopic model undergoing a local (stochastic)
dynamics, is a much harder task and this program was started only
rather recently \cite{cf:Spohn}.  More precisely, what one expects is
that, if the initial droplet is of diameter $L$, it will ``disappear''
within a time of order $L^2$ (this behavior is sometimes referred to
as ``Lifshitz law''). Moreover, in the
``diffusive limit'' where $L\to\infty$ and at the same time space is
rescaled by $L$ (so that the initial droplet is of size $O(1)$) and
time is accelerated by $L^2$, the droplet evolution should become
deterministic and follow some anisotropic version of a mean curvature
flow. Anisotropy (i.e.\ the fact that the normal velocity will also
depend on the local orientation of the droplet boundary) is expected when
the underlying model is defined on a lattice, as will be the case for
us.

Up to now, mathematical progress on this issue has been rather
modest, the main difficulty being that it is not clear how to
implement the idea that the fast modes related to relaxation inside
the two pure phases should decouple from slow modes related to the
interface motion, which are responsible for the diffusive time scaling $L^2$.

A fairly well understood situation is that where the interface can be
described by a height function and the bulk structure of the two
phases is disregarded. This is possible (by definition) for the
so-called ``effective interface models'' or Ginzburg-Landau
$\nabla\phi$ interface models: for models with continuous heights and
strictly convex potential undergoing a Langevin-type dynamics, Funaki
and Spohn \cite{FS} derived the full mean-curvature motion in the
diffusive scaling.  Another well-studied case is that of models with
Kac-type potentials: in this case, mean-curvature motion can be proven
to emerge \cite{K1,K2,K3} in a limit where interaction range is
taken to infinity at some stage, but in this limit there is no sharp
interface separating the phases and the system becomes very close to
mean-field.  

As for true lattice models, results are much more scarce. For
instance, for the two-dimensional nearest-neighbor Ising model below
the critical temperature, the best known upper bound on the ``disappearance
time'' for a droplet of ``$-$ phase'' immersed in the ``$+$ phase'' is
of the order $L^{c(T)\log L}$ \cite{cf:LSMT}, very far from the
expected $L^2$ scaling. Recently, a weak version of the Lifshitz law
was proven for the three-dimensional Ising model at zero temperature:
the disappearance time of a ``$-$'' droplet is of order $L^2$ (upper and lower bounds), up to
multiplicative logarithmic (in $L$) corrections \cite{cf:CMST}.
When the dimension is higher than three (always at zero temperature), 
an \emph{upper bound} for the disappearance time of order $L^2(\log L)^c$, for some constant $c$, was  proven in \cite{cf:L}.

In this work, we concentrate on the two-dimensional nearest-neighbor
Ising model on the infinite square lattice. The dynamics takes a very
simple form: each spin is updated with rate one and after the update
it takes the same value as the majority of its neighbors, or the value
$\pm1 $ with equal probabilities if exactly two neighbors are ``$+$''
and two are ``$-$''.  In this case, the disappearance time of a large
``$-$'' droplet should be asymptotically given by one half its volume
(number of ``$-$'' spins).  Moreover, in the diffusive scaling limit
the droplet boundary should be given by a deterministic curve
$\gamma(t)$ whose normal speed is given by the local (signed) curvature,
times a function $a(\theta)$ where $\theta$ is the angle of the local
normal vector. The function $a(\cdot)$ is explicitly known, see \eqref{eq:a}.
In this two-dimensional setting, it is more natural to refer to such flow as 
``(anisotropic) curve-shortening flow'' (rather than ``mean curvature flow'').

Our main result (Theorem \ref{th:convex}) is a proof of the
curve-shortening conjecture (and, as a byproduct, of the Lifshitz law)
when the initial droplet is convex.

There are some previous partial results available on this problem.  The scaling limit of the evolution
when initially spins are ``$-$'' in the first quadrant of $\mathbb
Z^2$ (infinite corner) and ``$+$'' elsewhere is described  in 
\cite[Section 4.2]{cf:KL}  (with the language of exclusion
processes rather than spin systems). This is a simple
situation because the interface motion is mapped to symmetric simple
exclusion and is described by the associated height function at all times.  In
\cite{cf:Spohn}, Spohn described the scaling limit of the interface
motion in a situation that more or less corresponds to the
zero-temperature Ising model in an infinite vertical cylinder, with an
initial condition such that the interface separating ``$+$'' from
``$-$'' spins can be globally described by a height function at all
times (in particular, this cannot describe a droplet, and implicitly
he has to modify the dynamics to guarantee that droplets do not appear
later in the evolution).  In \cite{cf:CSS}, Chayes {\sl et al.} proved
the Lifshitz law (but not the curve-shortening conjecture) for a
modified dynamics where updates which break the droplet into several
droplets are forbidden. In \cite{cf:CL}, Cerf and Louhichi computed the
``drift at time $0$'' of the droplet (for the non-modified dynamics),
but their result does not allow to get information on the evolution
for finite time $t>0$.

An important building block of our proof of the anisotropic
curve-shortening conjecture is that, as was well understood by Spohn
in \cite{cf:Spohn}, locally the interface can be (roughly speaking)
described by the hydrodynamic limit of a certain zero-range process at
the points where the tangent to the boundary is horizontal or
vertical, and by the hydrodynamic limit of the symmetric simple
exclusion process elsewhere.  However, such correspondence is not exact
due to updates that split the droplet into more than one connected components
(see for instance Figure \ref{fig:interditmove}). In other words, the
interface is not (even locally) the graph of a function. Also, it is a non-trivial
task to patch together the various pieces of ``local analysis'' to control globally
the evolution of the droplet. Both problems will be tamed by a sequence of
monotonicity arguments, which are allowed because the dynamics conserves the
stochastic ordering among configurations.

In order to prove Theorem \ref{th:convex}, we also need to know that a
classical solution to the anisotropic curve-shortening flow exists up to the
time where the droplet disappears, and  (crucially) that such
solution is sufficiently regular in space and time (i.e.\ that the curvature is a Lipschitz
function of the angle and a continuous function of time). To our surprise, we found that the existing
literature on curve-shortening flows does not provide global (in time)
results for the flow associated to the zero-temperature 2d Ising
model. The reason is that the anisotropy function $a(\cdot)$ is not
smooth (its derivative has jumps, reflecting the singularities of the
surface tension at zero temperature), while the existing results
assume that $a(\cdot)$ is at least $C^2$, cf.\ \cite{GL1,GL2}. To prove
existence, uniqueness and regularity of the solutions (cf.\ Theorem
\ref{th:deterministico}), we will regularize the function $a(\cdot)$
and then analyze the regularized flow following the ideas of
\cite{GL1,GL2}.  Of course, it will be crucial to guarantee that all
the estimates we need are uniform in the regularization parameter,
which tends to zero in the end.

\medskip

The case where the initial ``$-$'' droplet is non-convex will be
considered in future work. The additional difficulties are
two-fold.  First of all, from the analytic point of view, available global
existence and regularity results for the solution of curve-shortening
flows with non-convex initial condition seem to be limited to the isotropic case
where $a(\cdot)\equiv 1$ 
\cite{Grayson}. Secondly, due to the fact that the droplet will move at
the same time outwards and inwards at different locations according to
the sign of the curvature, various monotonicity arguments we use in
the rest of the paper will not work.

\section{Model and results}
\label{sec:modelresult}

\subsection{Glauber dynamics and expected limiting evolution}

Set $\bbZ^*:=\bbZ+\frac{1}{2}:=\{x+(1/2)\ | x\in \bbZ\}$.  We
consider the zero-temperature stochastic Ising model on $(\bbZ^*)^2$
with its usual lattice structure ($x$ and $y$ are linked if $|x-y|=1$
for the $l_1$ norm). This is a continuous time Markov chain
$(\sigma(t))_{t\ge 0}$ on the space of spin configurations on $(\bbZ^*) ^2$, $\gO:=\{-1,1\}^{(\bbZ^*) ^2}$. We write
$\sigma(t)=(\sigma_x(t))_{x\in (\bbZ^*)^2}$ and for simplicity
we  write $\sigma_x=-$ (resp. $\sigma_x=+$)
instead of $\sigma_x=-1$ (resp.  $\sigma_x=+1$).

The transition rules are the following : for each site $x\in
(\bbZ^*)^2$, the value $\sigma_x$ of the spin at $x$ is updated
independently with rate one.  When the spin at a site is updated, it
takes the same value as the spin of the majority of its neighbors, or
the values $\pm1$ with equal probabilities $1/2$ if two neighbors have
"$+$" spins and the other two "$-$" spins.  
That these rules yield a well-defined Markov chain even in infinite
volume is a standard fact (cf.\ \cite{cf:Lig}). In what follows
(cf.\ \eqref{bien}), we will consider only initial conditions where
the number of ``$-$'' spins is finite.  It is easy to realize that the
spins outside the smallest square containing all the initial ``$-$''
spins stay ``$+$'' forever, so that in reality we have a dynamics on a
finite volume and the question of existence of the process is trivial.

We are interested in the evolution of the set of "$-$" spins for this
Markov chain when the initial condition $\sigma(0)$ is a large droplet, i.e.\
a finite connected set of "$-$" spins surrounded by "$+$" spins. In that
case, almost surely, after a finite time $\tau_+$, all the "$-$" spins
have turned to "$+$" and the dynamics will stay forever in the all "$+$"
configuration (which is an absorbing state).  Our aim is to describe
the evolution of the shape of the rescaled ``$-$'' droplet on a proper
(diffusive)
time-scale. In the next paragraph we make that aim more precise.

\medskip

We consider a compact, simply connected subset $\mathcal D\subset
[-1,1]^2$ whose boundary is a closed smooth curve.  Given $L\in \N$ we
consider the Markov chain described above with  initial 
condition
\begin{equation}\label{bien}
\sigma_x(0)=\begin{cases} - 1 \quad &\text{ if } x\in (\bbZ^*)^2\cap L\mathcal D,
                          \\ +1 \quad &\text{ otherwise}.
            \end{cases}
 \end{equation}
In order to see a set of "$-$" spins as a subset of $\bbR^2$, each
vertex $x\in (\bbZ^*)^2$ may be identified with the closed square of
side-length one centered at $x$,
\begin{equation}
\mathcal C_x:= x+[-1/2,1/2]^2.
\end{equation}
One defines 
\begin{equation}\label{mal}
\mal(t):=\bigcup_{\{x: \ \sigma_x(t)=-1\}} \mathcal C_x,
\end{equation}
which is the ``$-$ droplet'' at time $t$ for the dynamics.  The
boundary of $\mal(t)$ is a union of edges of $\Z^2$ (this is the only
reason why we defined the Ising model on $(\Z^*)^2$).
\medskip
                                                   
What was conjectured by Lifshitz \cite{cf:Lifshitz} on heuristic
grounds for the low temperature Ising model is that $\mal(t)$ should
follow an anisotropic curve shortening  motion: after rescaling space
by $L$ and time by $L^2$ and letting $L$ tend to infinity, the motion
of the interface between $\mal(t)$ and its complement should be
deterministic and the local drift of the interface should be
proportional to the curvature, with an anisotropic correction to
reflect anisotropy of the underlying lattice.  More precisely, one can
formulate this conjecture as follows \cite{cf:Spohn}: Let
$\gamma(t,L)$ denote the boundary of the (random) set $(1/L)\mal(L^2
t)$. Then, for $L\to\infty$, $\gamma(t,L)$ should converge to a
deterministic curve $\gamma(t)$ and the evolution of
$(\gamma(t))_{t\ge0}$ should be such that the normal velocity at a
point $x\in\gamma(t)$ is given by the curvature at $x$, times an anisotropic
factor $a(\theta_x)$, where $\theta_x$ is the slope of the outwards
directed normal
to $\gamma(t)$ at $x$. The velocity is directed inwards at points
where $\gamma(t)$ is convex and outwards at points where it is 
concave.
The function $a(\cdot)$ should have the explicit expression
\begin{eqnarray}
  \label{eq:a}
  a(\theta):=\frac1{2(|\cos(\theta)|+|\sin(\theta)|)^2}.
\end{eqnarray}
In particular, the curve $\gamma(t)$ should shrink to a point in a finite
time 
\[
t_0=\frac{Area(\mathcal D)}{\int_0^{2\pi} a(\theta) \dd\theta}=
\frac{Area(\mathcal D)}2.
\]

Note that the function $a(\cdot)$ is symmetric around $0$ and is periodic
with period $\pi/2$, which reflects the discrete symmetries of the lattice
$(\Z^*)^2$. It is important to note for the following that $a(\cdot)$
is $C^\infty$ except at $\theta=j\pi/2, j=0,\dots,3$ where it is only 
continuous and its first derivative has a jump: indeed,
$a(\theta)\sim 1/2-|\theta-i\pi/2|$ for $\theta$ close to $i\pi/2,
i=0,\dots,3$. 

\subsection{Results}
\subsubsection{Convex initial droplet}
\label{sec:cib}

We prove the anisotropic curve shortening conjecture in the case where
the initial droplet is convex and suitably smooth.  
Given a strictly convex smooth domain  $\DD$ in $\bbR^2$ and letting $\gamma=\partial\DD$ be its boundary, we parameterize it
following a standard convention of convex geometry
(cf.\ e.g.\ \cite{GL2} and Figure \ref{fig:suppfunc}). 
For $\theta\in[0,2\pi]$ let ${ v}(\theta)$ be the unit vector forming an anticlockwise angle
$\theta$ with the horizontal axis
and let
\begin{eqnarray}
h(\theta)=\sup\{x\cdot v(\theta), x\in\gamma\}
\end{eqnarray}
with $\cdot$ the usual scalar product in $\bbR^2$.

Sometimes, we abusively say that $\gamma$ is a convex curve if 
the domain $\DD$ it encloses is convex, and we identify $\gamma$ with $\DD$.

 \begin{figure}[hlt]
\leavevmode
\epsfysize =5.8 cm
\psfragscanon 
\psfrag{theta}{\small $\theta$}
\psfrag{O}{\small $\bf{0}$}
\psfrag{x}{\small $x(\theta)$}
\psfrag{N}{}
\psfrag{k}{\small $k(\theta)$}
\psfrag{V}{\small $v(\theta)$}
\psfrag{h}{\small $h(\theta)$}
\epsfbox{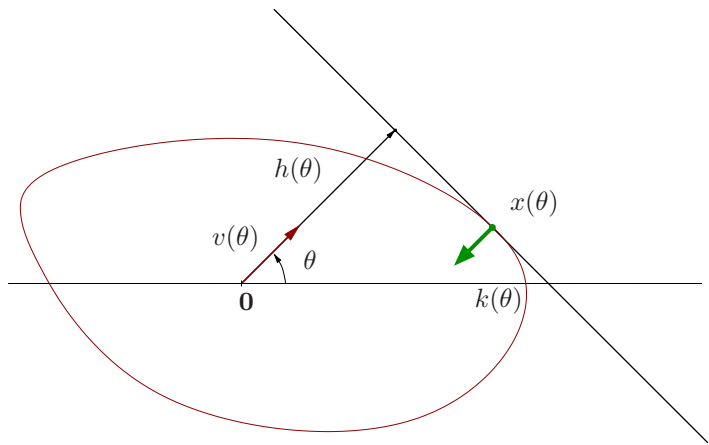}
\begin{center}
\caption{
 A graphic description of the support function $h$.
Given $\theta$, consider the point $x(\theta)$ of $\gamma$ that maximizes 
$x\cdot v(\theta)$ (it is unique if the curve is strictly convex).
Then $h(\theta)=x(\theta)\cdot v(\theta)$, and $k(\theta)$ is the norm of the curvature vector of $\gamma$ (bold vector)
 at  $x(\theta)$. If the tangent to $\gamma$ at $x$ exists it is normal to $v(\theta)$ and $|h(\theta)|$ 
is the distance between the tangent and the origin.
We emphasize that this construction works equally well when the origin is not 
inside $\gamma$. } 
\label{fig:suppfunc}
\end{center}
\end{figure}

The function $\theta\mapsto h(\theta)$ (called ``the support function'') 
uniquely determines $\gamma$:
\begin{eqnarray}
  \DD=\cap_{0\le \theta\le 2\pi}\{x\in\bbR^2:x\cdot v(\theta)\le h(\theta)\}.
\end{eqnarray}

With this parameterization, the anisotropic curve shortening evolution reads
\begin{eqnarray}
  \label{eq:meanc}
  \begin{cases}
\partial_t h(\theta,t)=- a(\theta)k(\theta,t)      \\
h(\theta,0)=h(\theta)
  \end{cases}
\end{eqnarray}
where, for a convex curve $\gamma$,
$k(\theta)\ge0$ is the curvature at the point $x(\theta)\in\gamma$ 
where the outward normal forms  an anticlockwise angle
$\theta$ with the horizontal axis and
the $t$-derivative is taken at constant $\theta$ (see \cite[Lemma 2.1]{GL2} for a proof of \eqref{eq:meanc}). 
Of course $h(\cdot)$ is the support function of $\partial \mathcal D$.

In general,
even proving the existence of a solution of \eqref{eq:meanc} with 
$a(\cdot)$ given in \eqref{eq:a}
is non-trivial,
since $a(\cdot)$ has points of non-differentiability and the existing
literature (e.g.\ \cite{GL1, GL2}) usually assumes that $a(\cdot)$ is at least $C^2$. 

\medskip

Our first result is
\begin{theorem}
\label{th:deterministico}
Let $\mathcal D\subset[-1,1]^2$ be strictly convex and assume that its boundary
$\gamma=\partial \mathcal D$ is a curve whose curvature $[0,2\pi]\ni
\theta\mapsto k(\theta)$ defines a positive, $2\pi$-periodic,
Lipschitz function. Then there exists a unique flow of convex curves
$(\gamma(t))_t$ with curvature defined everywhere, such that $\gamma(0)=\gamma$ and that the corresponding
support function $h(\theta,t)$ solves \eqref{eq:meanc} for $t\ge0$ and
satisfies the correct initial condition $h(\theta,0)=h(\theta)$. The
curve $\gamma(t)$ shrinks to a point ${\bf x}_f\in\bbR^2$ at time
$t_f=Area(\mathcal D)/2$. For $t<t_f$, $\gamma(t)$ is a smooth curve
in the following sense: its curvature function $k(\cdot,t)$ is
Lipschitz and bounded away from $0$ and infinity on any compact subset
of $[0,t_f)$.  
\end{theorem}

We let $\mathcal D(t)$ denote the convex closed set enclosed by $\gamma(t)$  (of
course, $\mathcal D(t=0)=\mathcal D$). Also, we use the convention
that $\mathcal D(t)=\{{\bf x}_f\}$ if
$t\ge t_f$.

\medskip
For $\delta>0$ let $B(x,\delta)$ denote the ball of radius $\delta$ centered at $x$
and 
for any compact set $\mathcal C\subset \bbR^2$ define
\begin{equation}\begin{split}
\label{eq:cdelta}
 \mathcal C^{(\delta)}:=\bigcup_{x\in \mathcal C} B(x,\gd),\quad\quad
 \mathcal C^{(-\delta)}:=\left(\bigcup_{x\notin \mathcal C} B(x,\gd)\right)^c.
\end{split}\end{equation}
Note that $\mathcal D(t)^{(\delta)}=B({\bf x}_f,\delta)$ and $\mathcal
D(t)^{(-\delta)}=\emptyset$ if $t\ge t_f$.

An event $B_L$ is said to occur \textit{with
  high probability} (w.h.p.)  if $\lim_{L\to \infty} P(B_L)=1$.

\medskip

\begin{theorem}
\label{th:convex}
Under the same assumptions on $\mathcal D$ as in Theorem \ref{th:deterministico}, for any $\delta>0$ one has w.h.p.
\begin{gather}\label{eq:scaling}
  \mathcal D^{(-\delta)}(t)\subset \frac1L \mal( L^2 t)\subset  \mathcal D^{(\delta)}(t)
\quad\quad \text{for every}\quad 0\le t\le t_f+\delta\\
\mal(L^2 t)=\emptyset \quad\quad \text{for every}\quad t>t_f+\delta.
\end{gather}
In particular,
 one has the following convergence in probability:
\begin{equation}
\label{eq:drif}
  \lim_{L\to\infty} \frac{\tau_+}{L^2 Area( \mathcal D)}=\frac{1}{2}.
\end{equation}
\end{theorem}

  The reason why in Theorems \ref{th:deterministico} and \ref{th:convex}
we do not content ourselves with, say,  initial
$C^\infty$ curves 
 is that, as we see in next section, there is a very natural initial condition
whose curvature function is only Lipschitz  and not $C^1$ (and stays
so at later times).

Theorem \ref{th:convex} does not apply directly if one considers $\mathcal
D=[0,1]^2$ or any other non-smooth or non-strictly convex convex set.  However, approximating
$\mathcal D$ from above and below by smooth compact sets and using
monotonicity (cf.\ Section \ref{sub:gcm}), one sees easily that \eqref{eq:drif} holds in any
case. In particular, the disappearance time of an $L\times L$ square
droplet is with high probability $L^2/2(1+o(1))$.

Theorems \ref{th:convex} and \ref{mainres} tell us that for our choices of initial configuration, 
the disappearance time of the minus droplet is non-random at first
order. This implies that the variation distance of our Markov Chain from equilibrium
(which is concentrated on the all-plus configuration) drops abruptly from $1$
to $0$ around time $L^2\,t_f$ within a time-window of width 
$o(L^2)\ll L^2\,t_f $  (we conjecture that the correct order of the window should be $O(L^{3/2})$). 
This is a particular instance of a phenomenon called cut-off
(cf.\ \cite{DS81} and \cite{LPW}).

\subsubsection{Scale-invariant droplet}

A particular case of Theorem \ref{th:convex} is that where the initial
condition is scale invariant, \textit{i.e.}\ when the limiting
evolution $(\gamma(t))_t$ is a homothetic contraction.  Consider the
function
\begin{equation}
\label{eq:soleqdif1}
f_0: \left[-\frac{1}{\sqrt{2}} , +\frac{1}{\sqrt{2}} \right]\ni x\mapsto
f_0(x)=\beta\left\{4\alpha 
x\int_0^xe^{2\alpha t^2}dt-e^{2\alpha x^2}\right\},
\end{equation}
where $\alpha$ is the unique positive solution of
\begin{equation}
\label{eq:alpha}
4\sqrt{2}\alpha e^{-\alpha}\int_0^{1/\sqrt{2}}e^{2\alpha t^2}dt=1
\end{equation}
and
\begin{eqnarray}
  \label{eq:beta}
  \beta=-\sqrt{2}e^{-\alpha}<0.
\end{eqnarray}
Note that $f_0$ is $\mathcal{C}^\infty$, positive, concave, symmetric
around $0$ and increasing on $[-\frac{1}{\sqrt{2}},0]$.  We denote by
$(\be_1,\be_2)$ the canonical basis of $\R^2$ and
$(\bff_1,\bff_2)=(\frac{\be_1-\be_2}{\sqrt{2}},\frac{\be_1+\be_2}{\sqrt{2}})$
the image of $(\be_1,\be_2)$ by the rotation of angle $-\pi/4$. We
also define the curve $\gamma_1$ to be the graph of $f_0$ in the
coordinate system $(\bof_1,\bof_2)$, i.e.
\begin{equation}
  \gamma_1:=\left\{ x\bof_1+ f_0(x)\bof_2 \ \big| \ x\in \left[-\frac{1}{\sqrt{2}},\frac{1}{\sqrt{2}}\right] \right\} .
\end{equation}
 If $S_1$ (resp. $S_2$) denotes the symmetry with respect to the axis 
$\be_1$ (resp. $\be_2$)
one defines the closed curve $\gamma$ by 
\begin{equation}
\label{eq:C}
\gamma=\gamma_1\cup (S_1\gamma_1)\cup (S_2 \gamma_1)\cup ((S_1\circ S_2) \gamma_1).
\end{equation}
In the sequel $ \mathscr D$ denotes the compact, convex set enclosed
in $\gamma$, see Figure \ref{fig:Dessin2}.  

\begin{figure}[hlt]

\leavevmode
\epsfxsize =8 cm
\psfragscanon 

\psfrag{-a}[c]{$-\frac{1}{\sqrt{2}}$}
\psfrag{+a}[c]{$\frac{1}{\sqrt{2}}$}
\psfrag{b}[c]{$\frac{1}{\sqrt{2}}$}
\psfrag{x}[c]{$x$}
\psfrag{f(x)}[c]{$f_0(x)$}
\psfrag{e1}{$\be_1$}
\psfrag{e2}{$\be_2$}
\psfrag{f1}{$\bof_1$}
\psfrag{f2}{$\bof_2$}
\psfrag{gamma}{$\gamma$}
\psfrag{gamma1}{$\gamma_1$}

\psfrag{D}{$\mathscr{D}$}

\epsfbox{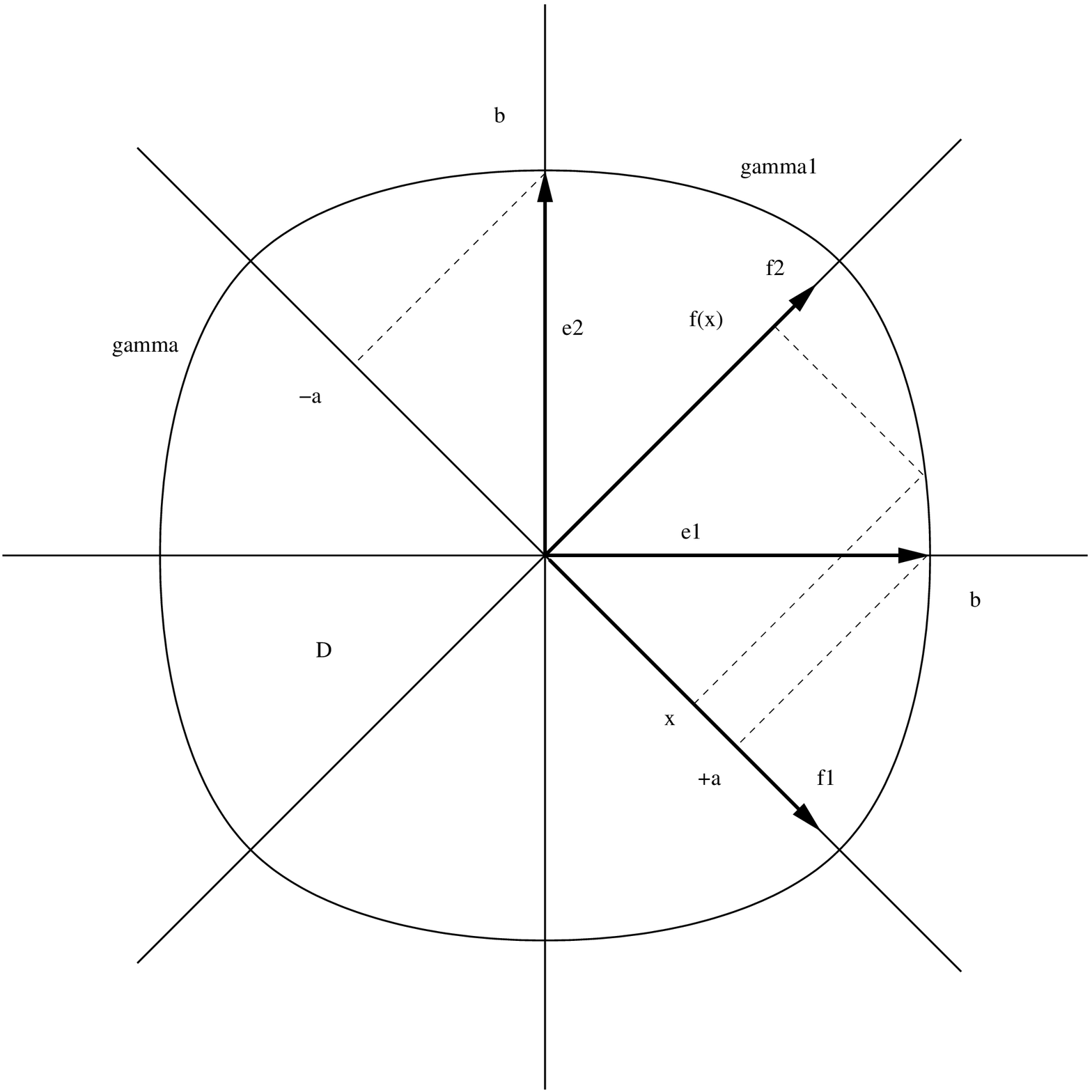}
\begin{center}
\caption{
The curve $\gamma=\partial\DDD$ and the coordinate systems $(\be_1,\be_2)$ and $(\bff_1,\bff_2)$.
}
\label{fig:Dessin2}
\end{center}
\end{figure}

One can check  that the  curvature function $\theta\mapsto
k(\theta)$ of $\partial\mathscr D$ is Lipschitz and bounded away from zero, 
but not differentiable at $\theta=i\pi/2, i=0,1,2,3$.
In this case, Theorem \ref{th:convex} can be formulated as follows.
\begin{theorem}\label{mainres}
Assume that $\mathcal D=\mathscr D$.
For any $\eta>0$, w.h.p,
\begin{equation}
\label{eq:theorem}
  (\sqrt{1-2\alpha t}-\eta)\mathscr D\subset \frac1L \mal(tL^2)\subset (\sqrt{1-2\alpha t}+\eta)\mathscr D \quad \mbox{for every}\quad t\ge0
\end{equation}
where we work with the convention that $\sqrt{x}=0$ for $x\le 0$ and that $x\mathscr D=\emptyset$ for $x<0$.
Moreover, one has the following convergence in probability:
\begin{equation}
\label{eq:driftbis}
  \lim_{L\to\infty} \frac{\tau_+}{Area( L \mathscr D)}=\alpha\lim_{L\to\infty} \frac{\tau_+}{L^2}=\frac{1}{2}.
\end{equation}

\end{theorem}
It is easy to check, using Lemma \ref{lemma:soleqdif1} below
and a couple of integrations by parts, that
$Area(\mathscr D)=1/\alpha$, yielding the first equality in
\eqref{eq:driftbis}.
The expression \eqref{eq:soleqdif1} for the invariant shape appears also, although with different notations, in the recent work \cite{Krap}.

\subsection{Graphical construction of the dynamics and monotonicity}
\label{sub:gcm}
Before starting the proofs, we wish to give a construction of the
 Markov process  (called sometimes the \emph{graphical construction}) that
yields nice monotonicity properties.  We consider a family of
independent Poisson clock processes $(\tau^{x})_{x\in(\Z^*)^2}$. More
precisely, to each site $x\in (\bbZ^*)^2$ one associates a random
sequence (independently from other sites) of times
$(\tau^x_{n})_{n\ge0}$, that are such that $\tau^x_{0}=0$ and
$(\tau^x_{n+1}-\tau^x_{n})_{n\ge 0}$ are IID exponential variables
with mean one.  One also defines random variables $(U_{n,x})_{n\ge 0},
x\in (\bbZ^*)^2$ that are IID Bernoulli variables of parameter $1/2$,
with values $\pm 1$.

\medskip

Then given an initial configuration $\xi\in \{-1,1\}^{(\bbZ^*)^2}$
one constructs the dynamics $\sigma^{\xi}(t)$ starting from
$\sigma^\xi(0)=\xi$ as follows
\begin{itemize}
 \item $(\sigma_x(t))_{t\ge 0}$ is constant on the intervals of the type $[\tau^x_{n},\tau^x_{n+1})$. 
 \item $\sigma_x(\tau^x_{n})$ is chosen to be equal to $\pm1$ if a
   strict majority of the neighbors of $x$ satisfies
   $\sigma_y(\tau^x_{n})=\pm1$, and $U_{n,x}$ otherwise (this definition makes sense as, almost surely, two neighbors will not update at the same time.)
\end{itemize}

This construction gives a simple way to define simultaneously the
dynamics for all initial conditions (we denote by $P$ the associated
probability). Moreover this construction preserves the natural order
on $\{-1,+1\}^{(\bbZ^*)^2}$, given by
\begin{equation}
 \xi\ge \xi' \Leftrightarrow \xi_x\ge \xi'_x \;\;\text{for every\;\;} x \in (\bbZ^*)^2
\end{equation}
(this order is just the opposite of the inclusion order for the set of "$-$" 
spins, which is therefore also preserved).
Indeed, if $\xi\ge \xi'$, with the above construction, one has $P$-a.s.
\begin{equation}
\forall t>0\;\;\;\;\;\;\; \sigma^{\xi}(t)\ge \sigma^{\xi'}(t).
\end{equation}

\section{Local interface dynamics}
\label{sub:dlm}
One problem one has to deal with when proving mean curvature motion
for the whole droplet is that even though initially the interface
between "$+$" and "$-$" (i.e.\ the geometric boundary of the set $\mal(0)$)
is a simple curve, it can later split to form several loops.
In fact, as a byproduct of our results, we will obtain that, with
large probability, only very small extra loops can be created. We will
tackle this problem by introducing some auxiliary dynamics that do not
allow creation of new loops and stochastically compare to the original one.

A second problem is that the interface that one has to control is not
exactly the graph of function, for which it would be easier
to describe the macroscopic motion using partial differential
equations.  We begin by studying two dynamics for which the interface
is indeed a graph, and which have locally the same large-scale behavior as the
true evolution.  It is more natural to
introduce these dynamics as dynamics on interfaces rather than
dynamics on spins. Our task then will consist of glueing together the
``local results'' of Theorems \ref{mickey} and \ref{spo} to get 
 Theorems \ref{th:convex} and \ref{mainres}. 

\subsection{Local dynamics away from the poles and simple exclusion process}
\label{sub:ldap}
The first auxiliary dynamics is used to control the evolution of the
boundary of $(1/L)\mal(tL^2)$
away from the points  (the poles) where the tangent to the deterministic curve
$\gamma(t)$ is either horizontal or vertical. The evolution
near the poles will be analyzed via a second auxiliary dynamics, see
Section \ref{sub:dap}.

Given two positive natural numbers $M, N$ consider the state-space
$\gO_{M,N}$ of nearest-neighbor directed paths of length $L:=M+N$ with
$M$ steps up and $N$ steps down:
\begin{equation}
\gO_{M,N}=\big\{(h_x)_{x\in \{0,\ldots,M+N\}}\in \bbZ^{M+N+1} \ \big|  \  |h_{x+1}-h_x|=1, h_0=0; h_{M+N}=M-N \big\}.
\end{equation}

Given $h\in\Omega_{M,N}$ and $x\in\{1,\dots,L-1\}$, we denote by $h^{(x)}$ the path with a corner ``flipped'' at $x$ defined by 
$h^{(x)}_y=h_y$ for all $y\ne x$ and
\begin{equation}
\label{eq:defho}
h^{(x)}_x:=\begin{cases} h_x-2 \text{ if } h_{x\pm 1}=h_x-1,\\
h_x+2 \text{ if } h_{x\pm 1}=h_x+1,\\
h_x \text{ if } |h_{x+1}-h_{x-1}|=2.\\
\end{cases}
\end{equation}

The dynamics on $\Omega_{M,N}$ we consider is the one that flips every
corner with rate $1/2$.  More precisely it is the Markov chain whose
generator $\mathcal L$ is defined as
\begin{equation}\label{defL1}
  \mathcal L f (h):= \frac12\sum_{x= 1}^{L-1} (f(h^{(x)})-f(h)), \quad \forall f: \Omega_{M,N} \mapsto\bbR.
\end{equation}

We denote by $(h(t))_{t\ge 0}$ the trajectory of  the Markov chain started from
initial condition $h(0):=h^0\in \gO_{M,N}$.

\begin{rem}\rm\label{121} Note that this dynamics is in one-to-one correspondence with
  the Ising dynamics on a rectangle $N\times M$ with ``$+$'' boundary
  conditions on two adjacent sides and ``$-$'' boundary conditions on
  the two opposite sides, provided that the initial configuration is
  such that the length of the $-/+$ boundary is $M+N$ (i.e.\ the
  minimal possible length).  More precisely (see Figure
  \ref{fig:coresis2}) the correspondence is obtained by taking the
  graph of $h$, rotating it by $\pi/4$ and rescaling space by a factor
  $\sqrt{2}$ (so that squares have side-length one on the left-hand side
  picture). 
Note that we are implicitly identifying an
  element $h\in\Omega_{M,N}$ with a continuous function
  $F:[0,M+N]\mapsto \bbR$ such that $F(x)=h_x$ for $x=0,1,\dots,M+N$
  and $F(\cdot)$ is affine on intervals $(n,n+1)$ with integer $n$.
\end{rem}

 \begin{figure}[hlt]

\leavevmode
\epsfxsize =10 cm
\psfragscanon 

\psfrag{O}{\small $\bf{0}$}
\psfrag{M}{\small $M$}
\psfrag{N}{\small $N$}
\psfrag{M+N}{\small $M+N$}
\psfrag{M-N}{\small $M-N$}

\epsfbox{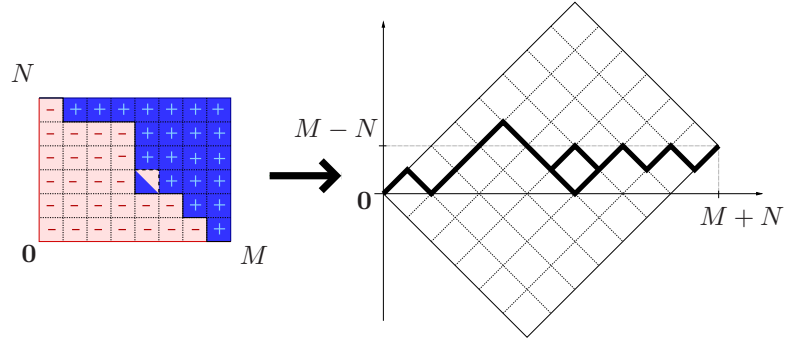}
\begin{center}
\caption{\label{fig:coresis2}
One-to-one correspondence between the dynamics in a rectangle with mixed 
boundary conditions and the corner-flip dynamics on paths. A possible spin update together with the 
equivalent corner-flip are represented}
\end{center}
\end{figure}

This corner-flip dynamics has been widely studied (see
e.g. \cite{Wilson}) and can be mapped to  the symmetric simple exclusion
process (SSEP) on a finite interval 
(just say that there is a particle at $x=0,\dots,
M+N-1$ if and only if $h_{x+1}-h_x=+1$, and check that dynamics in terms of particles coincides with that of SSEP). From hydrodynamic-limit results, it is
quite clear that the rescaled version of $h$ when $M,N$ tends to
infinity should satisfy the heat equation (see \cite[Section 4.2.]{cf:KL} for an account on hydrodynamic equations for the exclusion process).  However, we have
not found  in the literature a proof of the following precise
statement we need (we give a concise proof of it in Section \ref{sec:potm}):
\begin{theorem}\label{mickey}
Given a  $1$-Lipschitz function $\phi^0: [0,1]\mapsto \bbR$ with $\phi^0(0)=0$,
let $(h(t))_{t\ge0}$ the dynamics starting from initial condition
$h^{0}\in \gO_{M_L,N_L}$ given by
\begin{equation}\begin{split}
h^{0}_x&:= 2\lfloor L\phi^0(x/L)/2\rfloor \text{ for even } x \\
h^{0}_x&:= 2\lfloor (L\phi^0(x/L)-1)/2\rfloor+1 \text{ for odd } x
\end{split}\end{equation}
($M_L$ and $N_L$ are implicitly fixed by $L$ and $\phi^0(1)$).
 For all $T\ge 0$ and $\gep>0$, w.h.p.
\begin{equation}
\sup_{t\in [0,T],x\in [0,1]} \frac{1}{L} \left |  h_{\lfloor x L\rfloor} (L^2 t)-L\phi(x,t) \right|\le \gep
\end{equation}
where $\phi: [0,1]\times \bbR_+\to \bbR $ is the solution of the  Cauchy problem
\begin{equation}
\label{eq:edpcont}
\begin{cases}
\partial_t \phi( x, t)= \frac12\partial^2_x \phi (x, t) \quad \forall t>0, \quad \forall x\in (0,1), \\
\phi( 0, t)=0, \quad \phi(1, t)=\phi^0(1) \quad \forall t>0,\\
\phi(x,0)=\phi^0(x) \quad \forall x\in (0,1).
\end{cases}
\end{equation}

\end{theorem}
Here, $\lfloor x\rfloor$ denotes the integer part of $x$, and 
the fact that $h^0$ does belong to $\O_{M_L,N_L}$ is an easy consequence of 
 $\phi^0$ being  $1$-Lipschitz.

\subsection{Local dynamics around the poles and a zero-range process}
\label{sub:dap}

For the definition of the second auxiliary dynamics, we use the same
notation as in the previous section, but no confusion should arise as
the proofs will be given in two independent sections.  The state space
is 
\begin{equation}
\label{thesta}
\gO_{L}:=\{ h: \{-L,\dots, L+1 \}\mapsto \bbZ\  \}.
\end{equation}
For $h\in\gO_{L}$ and $x\in \{-L+1,\dots, L \}$ define $h^{+,x}$
(resp. $h^{-,x}$) as the configuration such that
%
$h^{+,x}_y=h_y$ if $y\ne x$ and 
$h^{+,x}_x=h_x+1 $ (resp. $h^{-,x}_x=h_x-1 $).
We consider the Markov chain $(h(t))_{t\ge 0}$ started from some $h^0\in 
\Omega_L$ and with generator $\mathcal L$ defined by
\begin{equation}\label{dnaspo}
\mathcal L f (h)=\frac12\sum_{x=-L+1}^L c^{+,x}(h)(f(h^{+,x})-f(h))+c^{-,x}(h)(f(h^{-,x})-f(h))
\end{equation}
where
\begin{equation}\begin{split}
 c^{+,x}(h)&=\ind_{\{h_{x+1}>h_x\}}+\ind_{\{h_{x-1}>h_x\}},  \\
 c^{-,x}(h)&=\ind_{\{h_{x+1}<h_x\}}+\ind_{\{h_{x-1}<h_x\}}.
\end{split}\end{equation}
Note that the values $h_{-L}$ and $h_{L+1}$ are fixed in time and
should be considered as boundary conditions.  
\begin{rem}\rm
  \label{rem:131}

  This dynamics corresponds to the motion of the interface for a
  \emph{modified Ising dynamics} in a vertical strip of width $2L$
  with the following boundary condition: spins on the left
  (resp. right) boundary of the system are "$+$" if and only if their
  vertical coordinate is larger than $h_{-L}$ (resp. $h_{L+1}$). The
  dynamics is modified in the sense that 
updates are discarded if after the update the
  boundary between the "$-$" and "$+$" domain is not a simple (open) curve
  (see Figure \ref{fig:interditmove}). 
It is at times more convenient to identify  $h\in \Omega_L$ with a càdlàg function
$H:[-L-1/2,L+3/2]\mapsto \bbZ$ which equals identically $h_n$ on intervals $[n-1/2,n+1/2)$ 
for integer $n$.

 \begin{figure}[hlt]

\leavevmode
\epsfxsize =10 cm
\psfragscanon 
\psfrag{pluspins}{$+$ spins}
\psfrag{minusspins}{$-$ spins}
\epsfbox{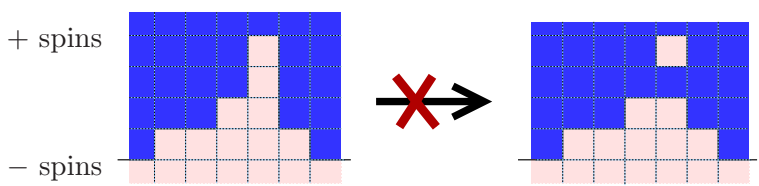}
\begin{center}
\caption{\label{fig:interditmove} An example of spin update that splits the interface into two disconnected components.
The interface dynamics presented in this section does not allow this kind of move.}
\end{center}
\end{figure}

Another way to interpret this dynamics \cite{cf:Spohn} is to look at
the gradients $\eta_x=h_{x+1}-h_x$: one recognizes then a zero-range
process with two type of particles (if $\eta_x=n>0$ we say there are
$n$ particles of type A at $x$, if $\eta_x=-n<0$ we say there are $n$
type-B particles).  Each particle performs a symmetric simple random
walk with jump rate   $1/(2n)$ (with $n$ the occupation number of the site where the particle sits) 
to either left or right and particles of
different type annihilate instantaneously when they are at the same
site. See Figure \ref{fig:partisys}.

 \begin{figure}[hlt]

\leavevmode
\epsfxsize =10 cm
\psfragscanon 
\epsfbox{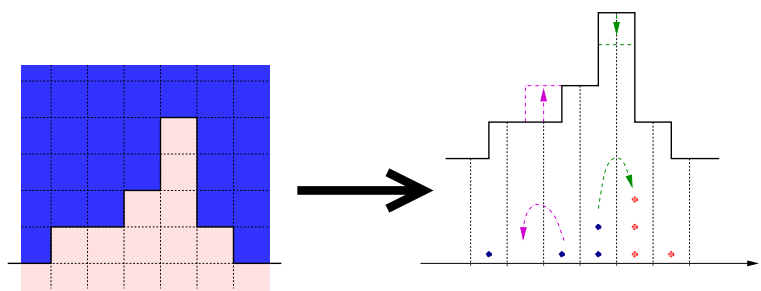}
\begin{center}
  \caption{\label{fig:partisys} Correspondence between interface
    dynamics and zero-range process.  Arrows represent possible
    motions for the interface and their representation in terms of
    particle moves.  When an A particle jumps on a B particle
    (green arrow) both annihilate.}
\end{center}
\end{figure}

\end{rem}

In \cite[Appendix A]{cf:Spohn}, this dynamics was considered but in a
periodized setup. A scaling limit result was given but the proof there
is somewhat sketchy.  Here we adapt the proof to the non-periodic case
and write it in full details.

\medskip

Consider $\phi^0: [-1,1]\mapsto \bbR$ a  $C^2$
 function with $\phi^0(1)=\phi^0(-1)=0$. 
We further assume that $\phi^0$ has a finite number of changes
of monotonicity.
Define  $\Phi_0:\{-L,\dots,L+1\}\mapsto \bbR$ as
\begin{equation}
\label{puntatore}
\Phi_0(x):=L\phi^0( x/L )
\end{equation}
and $h^{0}:\{-L,\dots,L+1\}\mapsto \bbZ$ by 
\begin{equation}\label{defho}
h^{0}_x:=\lfloor \Phi_0(x) \rfloor.
\end{equation}
We define $\Phi:\{-L,\dots, L+1 \}\times \bbR_+\to \bbR$ as the solution of the following Cauchy problem:
\begin{equation}
\label{eq:edpdiscret}
\begin{cases}
 \partial_t \Phi(x,t)&= \frac12\left[\sigma(q_{x}(t))-\sigma(q_{x-1}(t))\right],
\\
 \Phi( L+1, t)&=\Phi(-L,t)=0, 
\\ 
  \Phi(x, 0)&=\Phi_0(x)
\end{cases}
\end{equation}
for every $t\ge0$ and $x\in \{- L,\dots, L +1\}$, where
$\sigma(u)={u}/({1+|u|})$ and
\begin{equation}\begin{split}
\label{qs}
q_{x}(t):&=\Phi(x+1,t)-\Phi(x,t).
\end{split}\end{equation}

The result we state now is slightly weaker than Theorem \ref{mickey} as it allows to control the profile $h$
only at a fixed time and not on a whole time interval.

\begin{theorem}\label{spo}
  Given $\phi^0$ as above, consider $(h(t))_{t\ge0}$ the dynamics
  described by \eqref{dnaspo} with initial condition $h^{0}$ as in
  \eqref{defho}. Then for any $t$, the following convergence holds in
  probability
\begin{equation}\label{atrusk}
\lim_{L\to \infty} \max_{x\in \{-L,\dots,L+1\}}\frac{1}{L}\left|h_x(L^2t)-\Phi(x,L^2t)\right|=0.
\end{equation}

\end{theorem}
It is quite intuitive that one should have that $\frac{1}{L}\Phi( \lfloor Lx \rfloor ,L^2t)\to \phi(x,t)$ for any $x\in[-1,1]$,
where  $ \phi:[-1,1]\times \bbR_+\to \bbR$ is the solution of
\begin{equation}
\begin{cases}
\partial_t \phi(x,t)&=\frac12\frac{\partial^2_x \phi(x,t)}{(1+|\partial_{x} \phi(x,t)|)^2}
\\
 \phi( 1, t)&=\phi(-1,t)=0\\
  \phi(x, 0)&=\phi^0(x)
\end{cases}
\end{equation}
for $t\ge0$ and $x\in (-1,1)$.  The particular form of the
non-linearity of this PDE makes the convergence question non-trivial,
but fortunately Theorem \ref{spo} together with a comparison with the
heat equation (cf.\ Section \ref{sub:lb}) turns out to be sufficient
for our purposes.  Indeed, define
$\bar \phi: [-1,1]\times \bbR_+\to \bbR$ to be the solution of
\begin{equation}
\begin{cases}
\partial_t \bar\phi(x,t)&=\frac12\partial^2_x \bar\phi(x,t)
\\
 \bar\phi( 1, t)&=\bar\phi(-1,t)=0\\
 \bar \phi(x, 0)&=\phi^0(x).
\end{cases}
\end{equation}
Then
\begin{corollary}\label{spo2}
  Let $\phi^0$ be as above, and assume further that it is 
  concave with $\| \partial_x \phi^0\|_\infty \le \eta$. For every
  $t\ge0$ and every $\gep>0$ the following inequality holds w.h.p.
\begin{equation}
\bar\phi(x/L,t)-\gep \le \frac{1}{L}h_x(L^2t)\le \bar\phi(x/L,(1+\eta)^{-2} t)+\gep
\quad\text{for every\;} x\in\{-L,\dots,L+1\}.
\end{equation} 
\end{corollary}

\begin{proof}
The result follows by combining Theorem \ref{spo}, Proposition \ref{proplaplb}, and by taking limits of rescaled versions of $\Phi_1$ and $\Phi_2$ in \eqref{trare} when $L$ tends to infinity (cf.\ Lemma \ref{lem:comparaison}).
\end{proof}

\subsection{About the scale-invariant shape}
\label{subsec:asis}
 Now that we know how
the interface should evolve locally (from Theorems \ref{mickey} and \ref{spo}) it is
possible to explain why $\mathscr D$ should be scale invariant.  By
symmetries of the problem and the fact that motion is driven by curvature,
the scale-invariant shape should be convex
symmetric around the axes $\bbR\bf e_1$, $\bbR\bf
e_2$.
Therefore it is enough to consider 
the boundary of the intersection of $\mathscr D$ with the first quadrant.

From Theorem \ref{mickey}, if $f$ is a Lipschitz function and  
$\partial \mathscr D$ is the graph of $f$ in the  coordinate system
 $(\bff_1, \bff_2)$,  the initial drift 
in the $\bf f_2$ direction is
$(1/4)\partial^2_x f,$
where the factor $1/4$ (instead of $1/2$) is due to the fact that in the correspondence
between Ising dynamics and dynamics of nearest-neighboring paths, space
has to be rescaled by $\sqrt{2}$, cf.\  Remark
\ref{121}).  One the other hand, the homothetic contraction of a
shape $\mathcal D$ of initial velocity $\alpha$
gives an initial drift of the interface in the $\bf f_2$ direction
 \begin{equation}
\alpha(-f+x\partial_x f).
 \end{equation}
That leads to the partial differential equation
 \begin{equation}
\partial^2_x f= 4\alpha(-f+x\partial_x f).
\label{eq:eqdif1}
 \end{equation}
 Next we impose the correct boundary conditions on $f$:
 \begin{itemize}
 \item We fix the scaling  by imposing that the point $(1,0)$
(and therefore also $(0,1),(-1,0),(0,-1)$) belongs to $\partial \mathscr D$.
 This gives 
 \begin{eqnarray}
   \label{eq:bc1}
f\left(\pm1/{\sqrt{2}}\right)=1/{\sqrt 2}.   
 \end{eqnarray}
 \item To guarantee that the curvature of $\partial \mathscr D$ is
well defined at the point $(0,1)$ we have to impose
\begin{eqnarray}
  \label{eq:bc2}
  \partial_x f\left(-1/{\sqrt{2}}\right)=-\partial_x f\left(1/{\sqrt{2}}\right)=1.
\end{eqnarray}
 \end{itemize}
  
We finally notice that
\begin{lemma}
\label{lemma:soleqdif1}
The function $f_0$ defined in \eqref{eq:soleqdif1} is the unique
solution of the Cauchy problem \eqref{eq:eqdif1}-\eqref{eq:bc1}-\eqref{eq:bc2}
for  $x\in (-1/\sqrt{2},
+1/\sqrt{2})$.
For other values of $\alpha$ the above problem has no solution.
\end{lemma}

\begin{proof}
Uniqueness of the solution is standard from theory of ordinary differential equation. The rest is just a matter of checking.
\end{proof}

\subsection{Organization of the paper}

Instead of proving directly Theorem \ref{th:convex} and then deducing
Theorem \ref{mainres} as a corollary,
we decided for pedagogical reasons to give first the proof in the case of the scale-invariant
droplet and then to point out what needs to be modified in the more
general case of a convex droplet. The reason is that, this way, we 
can easily separate the question of comparing the stochastic evolution
with the deterministic one (which works more or less the same in the two cases
but is simpler for the invariant droplet, due to its symmetries)
from the analytic, PDE-type issues which appear only in the general case.

The paper is therefore organized as follows:
\begin{itemize}
\item in Section \ref{infinitez}, we show that to prove  Theorem \ref{mainres}
it is sufficient to have a good control on the continuity of the
interface motion (Proposition \ref{trucrelou}) and a result on the
evolution after an ``infinitesimal time'' $\gep L^2$ (Proposition \ref{mainprop}).
Such crucial results are proven in Sections \ref{upb} and \ref{lwb};
\item in Section \ref{sec:convexbu} we first prove Theorem \ref{th:deterministico}
on the existence of a solution to \eqref{eq:meanc}, and then we prove
Theorem \ref{th:convex} via a suitable generalization of Propositions
\ref{trucrelou} and \ref{mainprop};
\item finally, the hydrodynamic limit results of Theorems 
\ref{mickey} and \ref{spo} are proven in detail  in Section \ref{sec:potm}
 and the Appendix \ref{sec:pots} respectively.
\end{itemize}

\section{Proof of Theorem \ref{mainres}: evolution of the scale-invariant 
droplet}\label{infinitez}

\subsection{Reducing to an ``infinitesimal'' time interval}

We decompose the proof of Theorem \ref{mainres} into two propositions.  The
first (and the main one) says that after a time $\gep L^2$ the droplet
looks very much the same but contracted by a factor
$(1-\alpha\gep+o(\gep))$.  
\begin{proposition}\label{mainprop}
  For all $\delta>0$ there exists $\gep_0(\delta)>0$ such that for all
  $0<\gep<\gep_0(\delta)$, w.h.p.,
\begin{equation}\label{hok}
\mathcal{A}_L(L^2 \gep) \subset (1-\gep(\alpha-\gd))L\mathscr{D},
\end{equation}
and 
\begin{equation}\label{hic}
\mathcal{A}_L(L^2 \gep) \supset (1-\gep(\alpha+\gd))L\mathscr{D}.
\end{equation}
\end{proposition}
The second proposition  controls continuity in time of the
rescaled motion:
\begin{proposition}\label{trucrelou}
For every $\gd>0$, w.h.p., 
\begin{equation}\label{tip}
\mathcal{A}_L(L^2 t) \subset (1+\gd)L\mathscr D
\text{\;for every\;} t\ge0.
\end{equation}
Moreover, for every $\gd>0$ there exists $\gep>0$ such that  w.h.p 
\begin{equation}\label{top}
  \mathcal{A}_L(L^2 t) \supset (1-\gd)L\mathscr D
\text{\;for every\;} t\in[0,\gep].
\end{equation}
\end{proposition}

\begin{proof}[Proof of Theorem \ref{mainres} assuming Propositions \ref{mainprop} and \ref{trucrelou}]
Given $\eta$ fix $\gd$ small enough and $\gep<\gep_0(\delta)$.
 Then using \eqref{hok} 
 one gets that w.h.p.
 \begin{equation}\label{touspass}
  \mathcal{A}_L(L^2 \gep) \subset (1-(\alpha-\gd)\gep)L\mathscr D.
 \end{equation}
 Let $(\mathcal{A}^{(1)}_L(L^2t))_{t\ge 0}$ denote the evolution of
 the set of "$-$" spins for the dynamics started from initial condition
 "$-$" on $(1-\gep(\alpha-\gd))L\mathscr D$
 and "$+$" elsewhere.  Then using the Markov property and monotonicity
 of the dynamics, one can couple the dynamics
 $(\mathcal{A}_L(L^2(\gep+t)))_{t\ge 0}$ and
 $(\mathcal{A}^{(1)}_L(L^2t))_{t\ge 0}$ such that on the event
 \eqref{touspass}
 \begin{equation}
 \mathcal{A}_L(L^2(\gep+t))\subset \mathcal{A}^{(1)}_L(L^2t), \quad \text{for
every}\;\; t\ge 0.
 \end{equation}
 Therefore, after conditioning to the event in \eqref{touspass} and using \eqref{hok} for $(1-(\alpha-\gd)\gep) L$ instead of $L$, one gets that w.h.p.:
 \begin{equation}
   \mathcal{A}_L(L^2\gep(1+(1-(\alpha-\gd)\gep)^2))\subset \mathcal{A}^{(1)}_L(L^2(1-(\alpha-\gd)\gep)^2\gep)\subset (1-(\alpha-\gd)\gep)^2 L\mathscr D.
 \end{equation}
Here we used the fact that $\mathcal A^{(1)}_L(t)$ has the same law as
$\mathcal A_{L(1-\alpha\varepsilon)}(t)$.
Using this argument repeatedly one gets that,
w.h.p.,
for all $k\in [1,\gep^{-3/2}]$
\begin{equation}
\mathcal{A}_L(L^2 t_k)\subset   (1-(\alpha-\gd)\gep)^k L\mathscr D
\end{equation}
where $t_k$ is defined by 
\begin{equation}
t_k:=\gep \sum_{i=0}^{k-1} (1-(\alpha-\gd)\gep)^{2i}=\gep \frac{1-(1-(\alpha-\gd)\gep)^{2k}}{1-(1-(\alpha-\gd)\gep)^2}.
\end{equation}
Here and in the sequel we assume
$\gep^{-3/2}$ to be in $\N$. 
The value $\varepsilon^{-3/2}$ could 
equally well be replaced by any number $k_f$ much larger than
$1/\varepsilon$: the only thing that matters is that 
$t_{k_f}$ is close to $1/(2(\alpha-\delta))$.
One remarks that for all values of $k$
\begin{equation}\begin{split}
(1-(\alpha-\delta)\gep)^k&=\sqrt{1-\frac{t_k(1-(1-(\alpha-\delta)\gep)^2)}{\gep}}\\
					&=\sqrt{1-2 (\alpha-\delta)t_k +t_kO(\gep)}.
\end{split}\end{equation}
As $(t_k)_{k\geq 0}$ is bounded above, there exists $C>0$ such that  for every $k\in [1,\gep^{-3/2}]$
\begin{equation}\label{corco}
\mathcal{A}_L(L^2t_k)\subset \left(\sqrt{1-2(\alpha-\gd)t_k +C\gep}\right)L\mathscr D\subset (\sqrt{1-2\alpha t_k}+\eta/2)L\mathscr D
\end{equation}
w.h.p. where the second inclusion holds  provided that $\gep$ and $\gd$ are small enough. 
 Combining \eqref{corco}, 
Proposition
 \ref{trucrelou} and stochastic coupling, one gets w.h.p. that, for every $k\in[0,\gep^{-3/2}]$ and $t\in (t_k,t_{k+1})$, 
\begin{equation}
  \mathcal{A}_L(L^2 t)\subset (\sqrt{1-2\alpha t_k}+(3\eta/4))L\mathscr D\subset (\sqrt{1-2\alpha t}+\eta)L\mathscr D
\end{equation}
and that, w.h.p., for every $ t\ge t_{\gep^{-3/2}}$
\begin{equation} \label{biendf} 
  \mathcal{A}_L(L^2 t)\subset (\sqrt{1-2\alpha
    t_{\gep^{-3/2}}}+(3\eta/4))L\mathscr D\subset \eta L \mathscr D.
\end{equation}
This  ends the proof of the upper inclusion in \eqref{eq:theorem}
(note that $t_{\gep^{-3/2}}$ approaches $1/(2\alpha)$ for $\gep,\delta$ small).
Moreover, 
 \eqref{biendf} and stochastic domination implies that, for some constant $C$, w.h.p.
\begin{equation}
\tau_+\le \frac{L^2}{2\alpha}(1+C\eta^2).
\end{equation} 
Indeed, it is known from \cite{FSS} that 
a droplet of minus spins of linear size $\eta L$
disappears within a time $\tau_+$  which w.h.p. is upper bounded by
$C\eta^2L^2$. 

The lower inclusion in \eqref{eq:theorem} and the lower bound on
$\tau_+$ are proved in an analogous way using \eqref{hic} instead of
\eqref{hok} and \eqref{top} instead of \eqref{tip}. Note that using \eqref{top} we have to take care to choose $\varepsilon$ small enough
but it is possible as $t_k-t_{k-1}$ is a non-increasing function of $k$ and $1/\varepsilon$.

\end{proof}

\subsection{Strategy of the proof of Proposition \ref{mainprop}}
Our aim is to use Theorem \ref{mickey} to control the motion of the interface away from the ``poles''
and Theorem \ref{spo} (or more precisely Corollary \ref{spo2})
to control the  motion of the interface close to the ``poles''. 
It is therefore crucial to compare the local SSEP 
or the zero-range dynamics introduced in Sections \ref{sub:ldap} and \ref{sub:dap} 
to the true evolution of the boundary between "$+$" and "$-$" spins. 

As we have already discussed at the beginning of  Section \ref{sub:dlm}, however, there exists no
exact mapping between the evolution of the height function associated
to the two particle processes and the evolution of the $+/-$ boundary,
since the original ``$-$'' droplet can break
into more droplets and, strictly speaking, the interface cannot be
described, even locally, as a height function. The way out is that,
thanks to monotonicity arguments and to the \textit{a priori} ``continuity'' information provided by
Proposition \ref{trucrelou}, we can remove certain updates of the
Markov Chain,
e.g. freeze certain spins to their initial value. This way, we can
show that locally the interface can be stochastically compared to the
height function associated to the SSEP (or to the zero-range process close to the poles). Of course, the detail
of the ``update removal procedure'' is quite different according
to whether we want to prove an upper or a lower bound on the ``$-$
domain''. For instance, if we want an upper bound we are allowed to
freeze ``$-$'' spins or to change some ``$+$'' into ``$-$'' spins in the initial
condition  (this is fine thanks to monotonicity) and at the same time
we can freeze the spins
outside $(1+\delta)L\mathscr D$ to  ``$+$'' (this is not allowed
directly by monotonicity, but \eqref{tip} guarantees that such spins
stay ``$+$'' for all times anyway, w.h.p.). If the ``update removal procedure'' is
performed suitably, the effect is that the various portions of the
$+/-$ interface (away from and close to the poles) then become
independent and evolve \emph{exactly} like the height functions of the
SSEP/zero-range process.

The approach outlined here will be also used in Section \ref{sec:thconv} in the
case with general convex initial condition (the generalization of
Proposition \ref{mainprop} is 
 Proposition \ref{mainpropc}).

\subsection{Upper Bound: proof of \eqref{hok} and \eqref{tip} }
\label{upb}

 The inclusion \eqref{hok} can be rewritten in the following manner,
 which is more convenient for the proof:
 for any positive $\gd$, for all $\gep$ small enough, w.h.p.
 \begin{equation}\label{manierdevoir}
 \sigma_x(\gep L^2)=+
\text{\;for every\;} x\in \left[(1-\gep(\alpha-\gd))L\mathscr D\right]^c.
 \end{equation}

 Given $\gd$, we fix a value of $\xi$ which is small enough (depending
 on $\gd$ in a way that is specified in Section \ref{sec:45}) and set
(cf.\ Figure \ref{label}) 
\begin{equation}\begin{split}
\label{MN}
M(\gep,\xi)&:= \{(x,y) \in \bbR^2,\  x\geq \xi \textrm{ and } y\geq \xi  \} 
\setminus \left[(1-\gep(\alpha-\gd))\mathscr D\right]\\
N(\gep,\xi)&:=\{(x,y)\in \bbR^2,\  y\geq 0 \textrm{ and } -\xi \leq x\leq \xi \}
\setminus \left[(1-\gep(\alpha-\gd))\mathscr D\right].
\end{split}\end{equation}
Remark that for any $\gep>0$, $M,N$ and their successive images by
rotation of angle $\pi/2$, $\pi$, $3 \pi/2$ form an $8$-piece cover of
the complementary set $[(1-\gep(\alpha-\gd))\mathscr D]^c$.

 \begin{figure}[hlt]

\leavevmode
\epsfxsize =8 cm
\psfragscanon 

\psfrag{Nepsilonxsi}{$N(\gep,\xi)$}
\psfrag{Mepsilonxsi}[c]{$M(\gep,\xi)$}
\psfrag{xsi}{$\xi $}
\psfrag{-xsi}{$-\xi $}
\psfrag{epsilon}{\small{$(1-\gep(\alpha-\delta))\mathscr{D}$}}
\psfrag{D}{$\mathscr{D}$}

\epsfbox{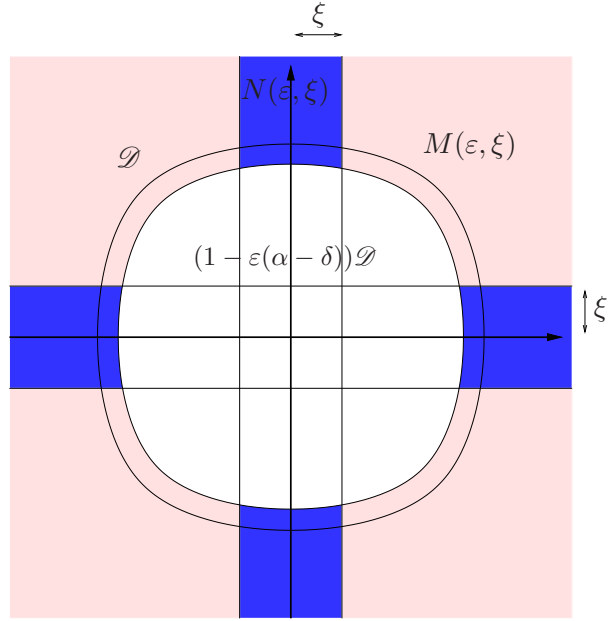}
\begin{center}
\caption{
The light-colored (resp. dark-colored) zones correspond  $M(\gep,\xi)$
(resp. $N(\gep,\xi)$) and its rotations. Together, they form a
partition of the complement of $(1-\gep(\alpha-\delta))\DDD$ (white
central region).}
\label{label}
\end{center}
\end{figure}

 As the dynamics and the initial shape are invariant under these same
 rotations, \eqref{manierdevoir} is proved if we can show that for
 $\gep$ small enough, w.h.p.
\begin{align}
\sigma_x(\gep L^2)&=+\quad \text{for every\;} x\in L M(\gep,\xi)\label{al:2steps1}\\
\sigma_x(\gep L^2)&=+\quad \text{for every\;} x\in L N(\gep,\xi).\label{al:2steps2}
\end{align}
The above new formulation of \eqref{hok} is very convenient as it
allows to consider separately the 
dynamics close to the poles and away from them.

\subsubsection{Proof of \eqref{al:2steps1}} 
\label{nunne}
For any $L>0$, we consider the dynamics which has initial condition
with "$-$" spins in $L\mathscr D$ and "$+$" otherwise, and the
same generator as the original dynamics except that spins on the sites in 
$V_1:=\{\pm \frac{1}{2}\}\times \{-L+\frac{1}{2},\cdots,L-\frac{1}{2}\}$
and on  $V_2:= \{-L+\frac{1}{2},\cdots,L-\frac{1}{2}\}\times \{\pm
\frac{1}{2}\}$ are ``frozen to $-$''. (The construction of the
dynamics is the same as in Section \ref{sub:gcm}, except that there is
no update for these sites).  We denote $(\sigma^{(1)}_L(t))_{t\ge 0}$
the evolution of this dynamics and define 
\begin{equation}
\mal^{(1)}(t):=\bigcup_{\{x:\ \sigma^{(1)}_x(t)=-1\}}\mathcal{C}_x.
 \end{equation}
The graphical construction of Section \ref{sub:gcm} gives a natural coupling of $\sigma$ and $\sigma^{(1)}$:
\begin{equation}
\label{eq:monotonie1}
\mal(t) \subset \mal^{(1)}(t),\quad \text{for every\;\;} t\geq 0.
\end{equation}
The advantage of the ``freezing procedure'' is that then the evolution in 
the four quadrants of $(\bbZ^*)^2$ becomes independent. The reason is that the spins on sites $(\{\pm \frac{1}{2}\}\times\bbZ^*)\setminus V_1$ and  $ (\bbZ^* \times \{\pm
\frac{1}{2}\})\setminus V_2$ are ``$+$'' for all times
(recall that the spins outside the smallest square containing the
initial ``$-$'' droplet stay ``$+$'' forever) so the boundary spins of all four quadrants
are frozen. 

The set $\mal^{(1)}(t)\cap \R_+^2$ is a Young diagram (i.e.\ a
collection of vertical columns of width $1$ and non-negative integer
heights, with heights non-increasing from left to right) for all $t\ge
0$ and we can thus consider $\partial \mal^{(1)}(t)\cap \R_+^2$ as the
graph of a (random) piecewise affine function in the coordinate system
$(\bff_1,\bff_2)$, that we denote by $F_L(\cdot,t)$.  Equation
\eqref{al:2steps1} is thus proved (for any choice of $\xi$) if one
proves that for any $\nu<2^{-1/2}$, for any $\gep$ small enough,
w.h.p.  
\begin{equation}
\label{eq:inefonctions}
F_L(x,\gep L^2)\le L f(x/L, (\alpha-\gd)\gep) \text{\; for every\;} x\in (-\nu L,\nu L),
\end{equation}
where $f(\cdot,t)$ is the function whose graph in the coordinate
system $(\bff_1,\bff_2)$ is given by the intersection of the
boundary of $({1-t})\mathscr D$ 
with the half-plane $\{(x,y)\in\R^2,\ (x,y)\cdot \bff_2 \geq 0\}$ (the
domain of definition of $f(\cdot,t)$ depends on $t$ but includes
$[-2^{-1/2},2^{-1/2}]$ for $t$ small enough). By definition of $\mathscr D$, one has
$f(\cdot,0)=f_0(\cdot)$ 
(recall the definition of $f_0$ in \eqref{eq:soleqdif1}). 

In practice, to prove
\eqref{al:2steps1} one has to prove \eqref{eq:inefonctions} with $\nu$
such that $1/\sqrt2-\nu=\xi/\sqrt2+o(\xi)$ for $\xi$ small (with $\xi$
as in \eqref{MN}). The reason
is that the point of $\partial\mathscr D$ with horizontal coordinate
$\xi$ and positive vertical coordinate (in the coordinate system $(\bf e_1,
\bf e_2)$) has horizontal coordinate
$-(1-\xi)/\sqrt 2+o(\xi)$ in the $(\bff_1,\bff_2)$ coordinate system.


As explained in  Remark \ref{121} and on Figure \ref{fig:coresis2}, the function $F_L(\cdot,t)$, up to  space rescaling (by a factor $\sqrt{2}$) undergoes the corner flip-dynamics of Theorem \ref{mickey}. 
Thus the scaling limit of $F_L$ satisfies the heat-equation or more precisely we have the following convergence in probability for every fixed $T>0$:
\begin{equation}
\label{eq:limithydro1}
\lim_{L\to \8}\sup_{x\in[-\frac{1}{\sqrt{2}},\frac{1}{\sqrt{2}}]}\sup_{t\leq T}\left|\frac{1}{L}F_L(xL,tL^2)-g(x,t)\right|=0
\end{equation}
where $g$ is the solution for $t\ge0$ and $x\in(-1/\sqrt 2,1/\sqrt 2)$ of
\begin{equation}
\begin{cases}
\partial_t g(x,t)=\frac{1}{4}\partial^2_x g(x,t) 
\\
 \label{eq:edp1} g(\cdot,0)=f_0(\cdot)
\\
 g(-\frac{1}{\sqrt{2}},t)=g(\frac{1}{\sqrt{2}},t)=\frac{1}{\sqrt{2}}
.
\end{cases}
\end{equation}

Note that the above result plus equation \eqref{eq:monotonie1}, plus
the fact that $g$ is decreasing in $t$ (since it stays concave through
time) gives \eqref{tip} of Proposition \ref{trucrelou} for every $
t\le T<\infty$.  Moreover, according to \cite[Theorem 1.3]{FSS}, 
the disappearance time $\tau_+$ of the minus-droplet is $O(L^2)$ with
high probability, so that \eqref{tip} also holds for $ t> T$ provided
that $T$ was chosen large enough. As a byproduct, we have proven \eqref{tip}.

Concerning \eqref{al:2steps1}, in order
to prove \eqref{eq:inefonctions} we are reduced to show that for
every $\nu\in(0,1/\sqrt2)$ and every $x\in (-\nu,\nu),$ 
\begin{equation}\label{cczzcc}
g(x,\gep)<f(x,(\alpha-\gd)\gep).
\end{equation}

This is a consequence of the way $f_0$ was determined (see
\eqref{eq:eqdif1} and discussion in Section
\ref{subsec:asis}). First we notice that the time
derivative of $g$ is uniformly continuous away from the boundary points $\pm 1/\sqrt2$:
\begin{lemma}
\label{lem:contilaplacien}
For any $0<\nu<\frac{1}{\sqrt{2}}$,
\begin{equation}
\label{eq:contilaplacien}
\lim_{t\to 0} \sup\{| \partial_t g(x,s)-\partial_t g(x,0)|, s\in[0,t]  \textrm{ and } x\in[-\nu,\nu]  \}=0.
\end{equation}
\end{lemma}
\begin{proof}[Proof of Lemma \ref{lem:contilaplacien}]
  This is well known but we sketch a probabilistic proof for the sake of
  completeness. Let $(B_t)_{t\geq 0}$ denote a standard Brownian
  motion starting at $x\in [-2^{-1/2},2^{-1/2}]$ (with the
  associated expectation denoted by $E_x$) and let $T$ denote the hitting time of $\{\pm1/{\sqrt{2}}\}$.  One has
\begin{equation}
 \partial^2_x g(x,t)=E_x\left[\partial^2_x f_0(B_t){\bf 1}_{\{t<T\}}\right].
 \end{equation}  
We can thus rewrite \eqref{eq:contilaplacien} as 
\begin{equation}
\label{eq:contibrownien}
\lim_{t\to 0} \sup\{|E_x(\partial^2_xf_0(B_s)-\partial^2_xf_0(x))|,s\in[0,t]  \textrm{ and }  x\in[-\nu,\nu]  \}=0
\end{equation}
and we can conclude using the uniform continuity of $\partial^2_xf_0$ on $[-\frac{1}{\sqrt{2}},\frac{1}{\sqrt{2}}]$ and well-known continuity properties
of the Brownian motion. Remark that 
\eqref{eq:contilaplacien} would not hold with $\nu=1/\sqrt2$ because of boundary effects: for $t>0$ one has that $\partial^2_x g(x,t)$ approaches zero 
as $x$ approaches $\pm1/\sqrt2$, since $P_x(T>t)\to 0$ when $x\to \pm1/\sqrt2$.
\end{proof}

Therefore, for every   $\eta>0$ arbitrarily small we have for all $x$ in $(-\nu,\nu)$, if  $\varepsilon$ is small enough,
\begin{equation}
g(x,\gep)< f_0(x)+\gep\left(\frac{1}{4}\partial^2_x f_0 (x)+\eta\right).
\end{equation}
We are left with proving that for $\gep$ small enough and $x$ in $(-\nu,\nu)$
\begin{equation}
 f_0(x)+\gep\left(\frac{1}{4}\partial^2_x f_0 (x)+\eta\right)<f(x,(\alpha-\gd)\gep).
\end{equation}
From the definition of $f(\cdot,t)$ as the graph in $({\bf f}_1,{\bf f}_2)$ of the boundary of
$({1-t})\mathscr D$, we get that if $\gep$ is small enough, uniformly for all $x\in (-\nu,\nu)$,
\begin{gather}
f(x,(\alpha-\gd)\gep)={[1-(\alpha-\gd)\gep]} f\left(\frac{x}{{1-(\alpha-\gd)\gep}},0\right)\\=f_0(x)+(\alpha-\gd)\gep \left(x\partial_x f_0(x) -f_0(x)\right)+O(\gep^2). 
\end{gather}
Now recall that  $f_0$ satisfies equation \eqref{eq:eqdif1}, so we are reduced
to check that for all $x\in(-\nu,\nu)$,
\begin{equation}
\eta+\delta(x\partial_x f_0(x) -f_0(x))=
\eta+\frac\delta{4\alpha}\partial^2_xf_0
<0
\end{equation}
which holds provided $\eta=\eta(\delta)$ is small,
since $\partial^2_{x} f_0(\cdot)$ is negative and uniformly bounded away from zero.
Equation \eqref{al:2steps1} is proven. \qed

\subsubsection{Proof of \eqref{al:2steps2}} 
\label{sec:45}
The method is similar to the one we used for \eqref{al:2steps1}, the
main difference being that, via a chain of monotonicity arguments, we
analyze the evolution of the portion of interface near the ``poles''
by comparing it to the interface dynamics of Section \ref{sub:dap}
(which coincides with the height function of the zero-range process with
two types of particles) instead of the ``corner-flip dynamics''.

Denote by $h(\cdot,t)$ the function whose graph in the coordinates
system $(\be_1,\be_2)$ is given by the intersection of 
 $({1-t})\partial \mathscr D$ with the upper half-plane $\bbR\times \bbR_+$, and
$h_0(\cdot)=h(\cdot,0)$.  Note that $h_0$ is $C^{\infty}$ on $(-1,0)$
and on $(0,+1)$ by the definition of $\mathscr D$. 
 The boundary
condition \eqref{eq:bc2} ensures continuity of the first derivative of $h_0$
at zero ($\partial_x h_0(0)=0$); the reader can check that $h_0$
has also continuous second derivative and that
 \begin{equation}
\label{tgv}
 \partial^2_x h_0
(0)=\frac{1}{2\sqrt{2}}\partial^2_x f_0(-1/\sqrt{2})=-2\alpha,
 \end{equation}
but that the third derivative exhibits a discontinuity in $0$.

Recall that $\xi$ is the positive constant appearing in
\eqref{MN}.
 We set $\bar h: [-4\xi,4\xi]\mapsto \bbR$ to be the function defined
by the following conditions: $\bar
 h\equiv h_0$ on $[-2\xi,2\xi]$,  $\bar h$ is affine on
 $[-4\xi,-2\xi]$ and on $[2\xi,4\xi]$ and  the
derivative $\partial_x\bar h(\cdot)$ is continuous on $(-4\xi,4\xi)$.
Since $h_0(\cdot)$ is strictly convex, we have $h_0(x)\le \bar h(x)$ with
strict inequality outside $[-2\xi,2\xi]$.
 Define also the following subsets of $\bbR^2$ (cf.\ Figure \ref{fig:dessin6}):
\begin{equation}
J^1:=[4\xi,\infty)\times [\bar h(4\xi),\infty)\quad\quad J^2:=(-\8,-4\xi]\times[\bar h(4\xi),\8).
\end{equation}
To avoid notational complications with integer parts, we assume that $L\bar h(4\xi)$ and
$4L\xi$ belong to $\bbZ^*$.

 \begin{figure}[hlt]

\leavevmode
\epsfxsize =10 cm
\psfragscanon 

\psfrag{j_1}{$L\,J^1$}
\psfrag{aa}{$L \bar h(4\xi)$}
\psfrag{j_2}{$L\,J^2$}
\psfrag{xi}{$L\xi$}
\psfrag{2xi}{$2L\xi$}
\psfrag{4xi}{$4L\xi$}
\psfrag{hbar}[c]{$L\bar{h}(\cdot/L)$}
\psfrag{Nepsilon}[c]{$N(\gep,\xi)$}

\epsfbox{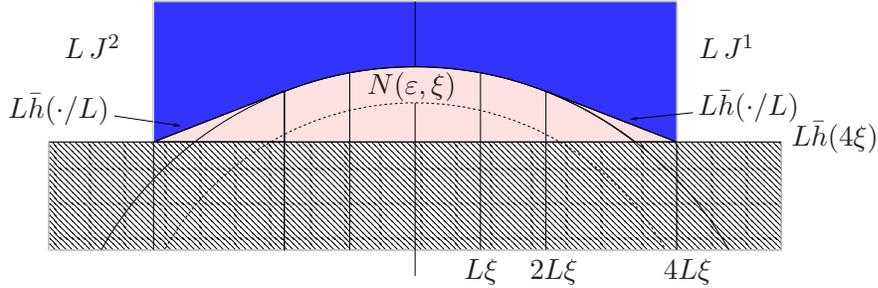}
\begin{center}
\caption{\label{fig:dessin6}
In the set $L(\,J^1\cup J^2)$ spins are frozen to "$+$" while in the
dashed region they are  frozen to ``$-$''. The initial condition is
``$+$'' in the dark-colored region and ``$-$'' in the light-colored
one. The boundary separating dark/light region is determined by the
function $\bar h(\cdot)$.}
\end{center}
\end{figure}

  First of
all, observe that, thanks to \eqref{tip}, we can freeze the spins in $ L(J^1\cup J^2) $ to their initial value "$+$" and, w.h.p., the
dynamics will be identical for all times to the original one.

Next, we employ a chain of monotonicities, based on the graphical
construction of Section \ref{sub:gcm}. Since we are
after an \emph{upper bound} on the set of minus spins, we can freeze
to "$-$" all spins whose vertical coordinate is below $L\bar
h(4\xi)$. Therefore, we have just a dynamics in the set
\[Y:= [-4L\xi+1,4L\xi-1]\times [L\bar h(4\xi),\infty) .\] In
principle, its initial condition is such that the spin at site
$(x_1,x_2)\in Y$ is "$-$" if and only if $x_2\in [L \bar h(4\xi),L
h_0(x_1/L)]$.  
The problem is however that the function $x\mapsto
\max(\bar h(4\xi),h_0(x))$ is not concave, which prevents to apply
directly Corollary \ref{spo2} later.  By monotonicity, we can modify
such initial condition by adding extra "$-$" spins: we therefore
stipulate that at time $t=0$ the spin at site $(x_1,x_2)$ is "$-$" if
and only if $x_2\in [L \bar h(4\xi),L \bar h(x_1/L)]$. Recall that
$\bar h(x)\ge h_0(x)$, so monotonicity goes in the correct direction.
With some abuse of notation, we still call $(\sigma(t))_{t\ge0}$ the
dynamics thus modified and $\mal(t)$ the set of minus spins. We need a final step in order to map the
evolution into the zero-range process. Note that, at time $t=0$, the
boundary of $\mal(t=0)$, intersected with the strip
$[-4L\xi+1/2,4L\xi-1/2]\times \bbR$, can be identified with the graph
of a c\`adl\`ag function \[ H_L(\cdot,0):[-4L\xi+1/2,4L\xi-1/2] \mapsto [L\bar
h(4\xi)-1/2,\infty)\cap \bbZ,\] which is constant on intervals
$[n,n+1)$ with $n\in \bbZ$ and takes boundary values $L\bar
h(4\xi)-1/2$ at the two endpoints ($H_L(x,0) $ is just a discretized
version of $L \bar h(x/L)$). However, for time $t>0$ it is not true in
general that the boundary of $\mal(t)$ is still the graph of a
function, simply because the set $\mal(t)$ can be non-connected (see
Figure \ref{fig:interditmove}).  Let $(\sigma^{(2)}(t))_{t\ge0}$ be
the dynamics obtained by erasing all the updates that would make
$\mal(t)$ non-connected. It is easy to realize that, since
$H_L(\cdot,0)$ has a single change of monotonicity (from
non-decreasing to non-increasing, recall that $\bar h(\cdot)$ is
concave) such erased updates can only correspond to a "$-$" spin
turning into a "$+$" spin (see again Figure \ref{fig:interditmove}). Therefore, the set of
minus spins of the dynamics $(\sigma^{(2)}(t))_{t\ge0}$ dominates
stochastically $\mal(t)$: more precisely,
we have shown that the coupling given by the graphical construction
implies that, w.h.p. and for all $t\ge0$, 
\begin{equation}
\label{eq:fields}
\mal(t)\subset\mathcal A_L^{(2)}(t):=\bigcup_{\{ x : \sigma^{(2)}_x(t)=-\}} \mathcal C_x.
\end{equation}
We let 
 \[
 H_L(\cdot,t):[-4L\xi+1/2,4L\xi-1/2] \mapsto [L\bar
 h(4\xi)-1/2,\infty)\cap \bbZ,\] denote the piecewise constant
 (random) function whose graph in the usual coordinates system
 $(\be_1,\be_2)$ is the intersection between $\partial
 \mal^{(2)}(t)$ and the strip $[-4L\xi+1/2,4L\xi-1/2]\times
 \bbR$. Note that $H_L(-4L\xi+1/2,t)=H_L(4L\xi-1/2,t)=L\bar
 h(4\xi)-1/2$.

\medskip

Equation \eqref{al:2steps2} is proved if one shows that for any $\gep$
small enough, w.h.p.
\begin{equation}\label{cocomero}
\frac1L H_L(x,\gep L^2)\le  h(x/L,(\alpha-\gd)\gep)\text{\;for every\;} x\in (-\xi L,\xi L).
\end{equation}
It is clear from  Remark \ref{rem:131} 
that the function $H_L(\cdot,\cdot)$ follows the
dynamics described in Section \ref{sub:dap}, with generator
\eqref{dnaspo}
(here we identify the function $H_L(\cdot,t)$ with
an element of $\Omega_{4\xi L-1/2}$, see  \eqref{thesta}).
According to Corollary \ref{spo2} one has
for arbitrarily small $\eta>0$, w.h.p, for all $x\in (-\xi L,\xi L)$
\begin{equation}\label{atrusk1}
  \frac1LH_L(x,\gep L^2)\le \bar\phi(x/L, (1+\| \partial_x \bar h\|_\infty)^{-2}
\gep)+\eta
\end{equation}
where $\| \partial_x \bar h\|_\infty=\sup_{[-4\xi,4\xi]}|\partial_x\bar h(x)|$ and $\bar\phi(x,L^2t)$ is the solution of 
\begin{equation}
\begin{cases}
\partial_t \bar\phi(x,t)&=\frac{1}{2}\partial^2_x \bar\phi(x,t)
\\
 \bar\phi( -4\xi, t)&=\bar\phi(4\xi,t)=\bar h(4\xi)\\
 \bar \phi(x, 0)&=\bar h(x) \quad \text{for every\;} x\in [-4\xi,4\xi].
\end{cases}
\end{equation}
The equation \eqref{cocomero} is thus proved if one has
\begin{equation}\label{letrucpareildelautresection}
 \bar\phi(x, (1+\| \partial_x \bar h\|_\infty)^{-2}\gep)< h(x,(\alpha-\gd)\gep)\text{\;for every\;} x\in [-\xi,\xi].
\end{equation}

Note that by Lemma \ref{lem:contilaplacien} (which is applicable
because the second derivative of $\bar h(\cdot)=h_0(\cdot)$ is uniformly
continuous in $(-2\xi,2\xi)$)
one has, uniformly on $[-\xi,\xi]$
\begin{multline}\label{patatra}
 \bar\phi(x, (1+\| \partial_x \bar h\|_\infty)^{-2}\gep)= \bar\phi(x, 0)+\frac{\gep}{2}(1+\| \partial_x \bar h\|_\infty)^{-2} \partial^2_x\bar\phi(x, 0)+o(\gep)\\
=h_0(x)+\frac{1}{2} (1+\| \partial_x\bar h\|_\infty)^{-2} (\partial^2_x h_0(0)+r(x))\gep+o(\gep)
\end{multline}
where $r(x)$ tends to $0$ for $x\to0$.
Finally, using \eqref{tgv}, if $\xi$ is chosen small enough
so that both $r(x)$ and $\|\partial_x\bar h\|_\infty$ are sufficiently
smaller than $\delta$,
\begin{equation}
 \bar\phi(x, (1+\| \partial_x \bar h\|_\infty)^{-2}\gep)\le h_0(x)-(\alpha-\gd/4)\gep.
\end{equation}

On the other hand one has
\begin{equation}
h(x,(\alpha-\gd)\gep)\ge h_0(x)-(\alpha-\gd/2)\gep,
\end{equation}
which ends the proof of \eqref{al:2steps2}. \qed

\subsection{Lower bound: proof of \eqref{top} and \eqref{hic}}\label{lwb}

The proofs follow the same ideas as those of Section \ref{upb}: we
need to control the dynamics for different portions of the interface
separately (around the poles and away from them) using the 
scaling
limit results provided by Theorems \ref{spo} and \ref{mickey}.


\subsubsection{Proof of \eqref{top}}

\label{sec:top}
Equation \eqref{top} is absolutely crucial to start the proof of
\eqref{hic} and quite independent of the rest.  The proof is very similar
to that of \cite[Theorem 2]{cf:CMST}, so we only sketch the main steps.
Set
\begin{equation*}
   D:=\{x\in (\bbZ^*)^2:d(x,(1-\delta)L\mathscr D)\le 1\}\,,\quad
   D':=(\bbZ^*)^2\cap\left((1+\delta^3)L\mathscr D\right)^c           
\end{equation*}
and  consider a modified dynamics $(\tilde \sigma(t))_{t\ge0}$
(whose law is denoted $\tilde\bP$), with the same initial condition as
$( \sigma(t))_{t\ge0}$ and the rules that: (i) after each update, any
"$-$" spin which has more than two "$+$" neighbors is turned to "$+$",
and the operation is repeated as long as such spins exist; (ii) the
dynamics stops at the time $\tilde \tau_{D,D'}$, the first time when
there is either a "$+$" spin in $D$ or a ``$-$'' spin in $D'$. We
define also $\tilde \tau_D$ the first time when there is a ``$+$''
spin in $D$ and $\tau_{D,D'},\tau_D$ the analogous random times for
the original dynamics.

Note that, by \eqref{tip}, w.h.p.  $\tau_{D,D'}=\tau_D$.  Note also
that the two dynamics can be coupled in a way that
$\tau_{D,D'}=\tau_D$ implies $\tilde\tau_{D,D'}=\tilde\tau_D\le
\tau_D$ (thanks to point (i) above, since before $\tilde\tau_{D,D'}$
the modified dynamics has less ``$-$'' spins that the original one).
Therefore,
\[
\bP(\tau_D\le \gep L^2)=\bP(\tau_D\le \gep L^2;\tau_{D,D'}=\tau_D)+o(1)\le
\tilde\bP(\tilde\tau_D\le \gep L^2;\tilde\tau_{D,D'}=\tilde\tau_D)+o(1)
\]
and it suffices to prove for instance that \[ \tilde\bP(\tilde
\tau_{D,D'}\le \gep L^2; \tilde\tau_D=\tilde\tau_{D,D'})\le
\exp(-\gamma L) \] for some $\gep=\gep(\delta)>0,\gamma>0$.
For this, one first observes (as in \cite[Eq. (8.6)]{cf:CMST}) that when $\tilde\tau_{D,D'}=\tilde\tau_D$
the difference between the number of ``$+$'' spins
at time $\tilde\tau_D$ and the number of ``$+$'' spins
at time $0$ is at least $c\gd^2 L^2$ deterministically, for some $c>0$.

Finally, (as in \cite[Eq. (8.10)]{cf:CMST}) one proves  that 
\[
\tilde\bP(|\{x:\tilde\sigma_x(\gep L^2)=+\}|-|\{x:\tilde\sigma_x(0)=+\}|
\ge c\delta^2L^2)\le \exp(-\gamma L)
\]
if $\gep=\gep(\delta)$ is small enough.  This is based on the fact
(cf.\  \cite[Lemma 8.5]{cf:CMST}) that, for times smaller than
$\tilde\tau_{D,D'}$, the rate of increase of the number of "$+$" spins
is uniformly bounded by a constant.

\subsubsection{Scheme of the proof of \eqref{hic}}
\label{sec:pohic}

Given some fixed $\delta>0$, we want to prove that for $\gep>0$ small enough, w.h.p.
\begin{equation}\label{lequipefr}
(1-(\alpha+\delta) \gep)L\DDD\subset \mal(\gep L^2),
\end{equation}
or equivalently
\begin{equation}
\sigma_x(\gep L^2)=-\;\text{for every}\;  x\in (1-(\alpha+\delta) \gep)L\DDD.
\end{equation}
Given $\xi$ small enough (depending on $\gd$) and $\nu$ small enough
(depending on $\xi$), we define (cf.\ Figure \ref{fig:dessin8})
\begin{equation}\begin{split}
\label{UAB}
U&:=(1-\nu)\mathscr D,\\
A_1(\gep)&:=\left[((1-(\alpha+\delta) \gep)\mathscr D)\setminus U\right]\cap [\xi,+\infty)^2,\\
B_1(\gep)&:=\left[((1-(\alpha+\delta) \gep)\mathscr D)\setminus U \right]\cap ([-\xi,\xi]\times\bbR^+).
\end{split}
\end{equation}
and $A_i,B_i$, $i=2,3,4$
as the  images of $A_1(\gep),B_1(\gep)$ by the
rotation of angle $(i-1)\frac{\pi}{2}$.

 \begin{figure}[hlt]

\leavevmode
\epsfxsize =8 cm
\psfragscanon 

\psfrag{A1}{$A_1$}
\psfrag{B1}{$B_1$}
\psfrag{U}{$U$}
\psfrag{xsi}{$\xi$}
\psfrag{4xi}{$4\xi$}
\psfrag{D}{$\mathscr{D}$}
\epsfbox{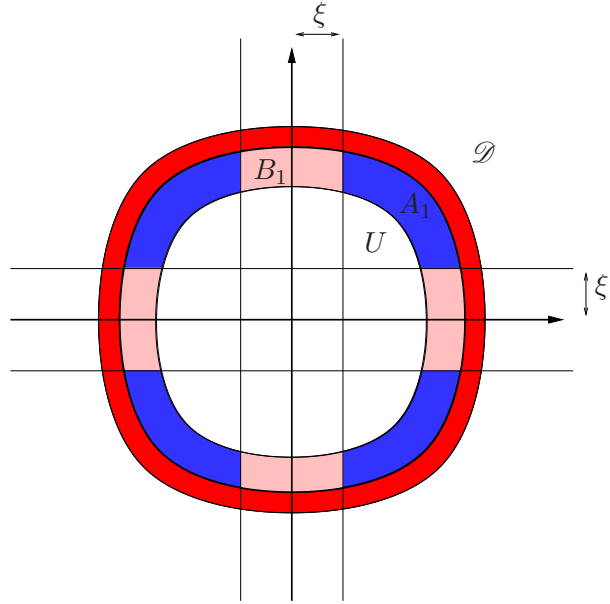}
\begin{center}
\caption{\label{fig:dessin8} The large droplet is $\DDD$ and
  $(1-\gep(\alpha+\delta))\mathscr{D}$ is obtained by removing the
  external dark layer. The white central region $U$, together
  with $A_1,B_1$ and its rotations (deformed rectangular regions) form
  a partition of $(1-\gep(\alpha+\delta))\mathscr{D}$.}
\end{center}
\end{figure}

One has
\begin{equation}
(1-(\alpha+\delta) \gep)\mathscr D=U\cup\left(\bigcup_{i=1}^4 A_i\right)\cup\left(\bigcup_{i=1}^4 B_i\right),
\end{equation}
and hence (using rotational symmetries), to prove \eqref{lequipefr}, it is sufficient to prove that for $\gep$ small enough, w.h.p.
\begin{align}
L\,U&\subset  \mal(\gep L^2), \label{eq:3steps1}\\
LA_1(\gep)&\subset \mal(\gep L^2), \label{eq:3steps2}\\
LB_1(\gep)&\subset \mal(\gep L^2).\label{eq:3steps3}
\end{align}
The first line, i.e.\ Equation \eqref{eq:3steps1}, is a direct
consequence of \eqref{top} provided that $\gep$ is chosen small enough
(how small depending on
$\nu$). 
Actually, one has the following
stronger statement that will be useful for what follows: if $\gep$ is small  then w.h.p.
\begin{equation}\label{trucbien}
L\,U\subset  \mal(t L^2)\;\text{for every\;} t\le \gep. 
\end{equation}
The main work is thus to prove  \eqref{eq:3steps2} and \eqref{eq:3steps3}.

\subsubsection{Proof of \eqref{eq:3steps3}}
This is similar to the proof of \eqref{al:2steps2}, except that
monotonicities will be needed in the opposite direction.

Let $\bar h:[-2\xi,2\xi]\mapsto\R$ be a concave, twice differentiable,
even function such that
\begin{equation}\begin{split}
\bar h(x)&= h_0(x), \quad \forall x\in [-\xi,\xi],\\
 \bar h(x)&<h_0(x), \quad \forall x\in [-2\xi, -\xi)\cup(\xi,2\xi]
\end{split}\end{equation}
where $h_0(\cdot)$ was defined in Section \ref{sec:45} to be the graph
of  $\partial\mathscr D\cap (\bbR\times \bbR^+)$ in the $(\be_1,\be_2)$ coordinate system. 
Once $\xi$ is fixed, we choose $\nu$ and $\bar h$ such that the point $(2\xi,\bar h(2\xi))$ lies in the interior of $U$.

Using equation \eqref{trucbien}, we can freeze the spins with
vertical coordinate $L \bar h(2\xi)$ and horizontal coordinate in
$(-2L\xi,2L\xi)$ (we assume for notational convenience that $2L\xi$
and $L\bar h(2\xi)$ are in $\bbZ^*$) to their initial value "$-$",
and w.h.p.\ , the dynamics we obtain is identical to the original one
up to time $\gep L^2$.

 \begin{figure}[hlt]

\leavevmode
\epsfxsize =10 cm
\psfragscanon 

\psfrag{A1}{$A_1$}
\psfrag{B1}{$B_2$}
\psfrag{U}{$L\,U$}
\psfrag{hbar}[c]{$L\bar{h}(\cdot/ L)$}
\psfrag{xi}{$L\xi$}
\psfrag{2xi}{$2L\xi$}
\psfrag{D}{$L\mathscr{D}$}
\epsfbox{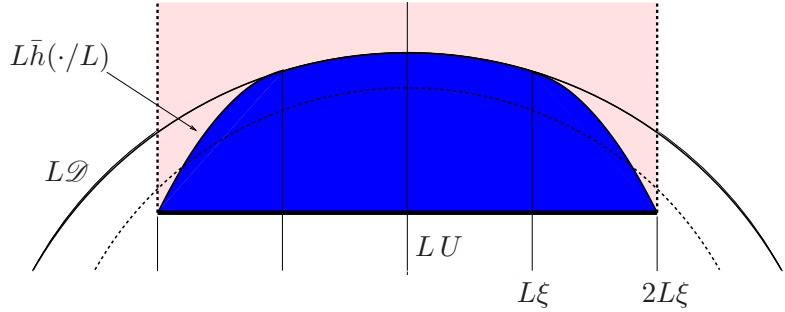}
\begin{center}
\caption{\label{fig:dessin9}
The sites in the dashed vertical lines are frozen to "$+$" and those
of the horizontal bold segment to "$-$" so that the dynamics in the
colored infinite rectangle is independent of the rest of the system. At time
$t=0$ the sites in the dark-colored region (whose upper boundary is
determined by $\bar h(\cdot)$ are "$-$" while those of the light-colored one are "$+$". The function $\bar h(\cdot)$ is such that the base of the 
dark-colored region is in $LU$.}
\end{center}
\end{figure}


Next we use a chain of monotonicities based on the graphical
construction of Section \ref{sub:gcm}.  Since we are after a
\textit{lower bound} on the set of minuses, we can freeze to "$+$" all
the spins with horizontal coordinate $\pm 2L\xi$ and vertical
coordinate larger than $L \bar h(2\xi)$.
Once this is done, we are reduced to considering the dynamics
restricted to the set
\begin{equation}
 Y_2:=[-2L\xi+1,2L\xi-1]\times [L\bar h(2\xi)+1,\infty)
,
\end{equation}
as spins on its boundary are fixed.  In principle, the initial
condition one should consider is such that $(x_1,x_2)\in Y_2$ has spin
"$-$" iff $x_2\in [L\bar h(2\xi)+1, L h_0(x_1/L)]$, but again by
monotonicity, we can add extra "$+$" spins: we stipulate that, at time
$t=0$, $(x_1,x_2)$ has spin "$-$" iff $x_2\in [L\bar h(2\xi)+1, L
\bar h(x_1/L)]$.  With some abuse of notation, the dynamics thus
modified is still called $(\sigma(t))_{t\ge0}$.

\medskip

As for the proof of \eqref{al:2steps2}, we need a final step to map
the dynamics onto the interface dynamics of Theorem \ref{spo}, the
problem being exactly the same as then: it is not true that the boundary of
$\mathcal A_L(t)$ stays connected for all $t$.  The solution adopted in
the previous section (leading to the dynamics $(\sigma^{(2)}(t))_t$,
see discussion before \eqref{eq:fields}) does not work here as we are
now looking for a \textit{lower bound}.

Let $(\sigma^{(3)}(t))_t$ be the dynamics that evolves like
$(\sigma(t))_t$ except that any spin that has three "$+$" neighbors is
turned instantaneously to "$+$" (see Figure \ref{fig:permismove}).  The coupling given by
graphical construction implies that
\begin{equation}\label{trucplusimp}
  \bigcup_{\{x: \sigma^{(3)}_x(t)=-\}}\mathcal C_x=: \mathcal A_L^{(3)}(t)\subset \mathcal A_L(t).
\end{equation}
Moreover our choice of initial condition guarantees that $\mathcal
A_L^{(3)}(t)$ stays connected for all time, since the set $\mathscr D$ is
convex.

 \begin{figure}[hlt]

\leavevmode
\epsfxsize =10 cm
\psfragscanon 
\psfrag{spinA}{\small $A$}
\psfrag{spinB}{\small $B$}
\epsfbox{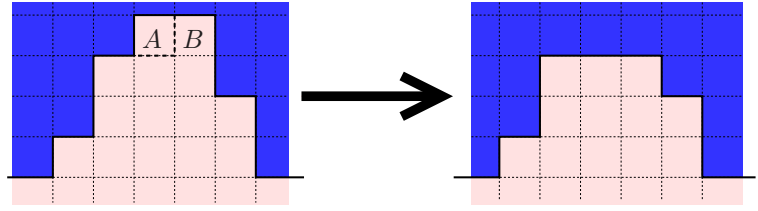}
\begin{center}
\caption{\label{fig:permismove} Light-colored (resp. dark-colored)
  squares denote ``$-$''  (resp. ``$+$'') spins. In our modified dynamics $\sigma^{(3)}$, when a spin has three ``$+$'' neighbors, it is instantaneously turned to 
``$+$''. On the figure, if spin at $A$ is updated and turns to
``$+$'', then the spin $B$ has three ``$+$'' neighbors and therefore also turns instantaneously to 
``$+$''.}
\end{center}
\end{figure}

We denote by $H_L(\cdot, t)$ the càdlàg function $[-2\xi,2\xi]\mapsto
\bbR$ whose graph  corresponds to the 
intersection between $\partial \mathcal A_L^{(3)}(t)$ and the vertical strip
$[-2L\xi+1/2,2L\xi-1/2]\times \bbR$.
Note that $H_L(\cdot, t)$ can be visualized as a collection of columns of
width $1$ and integer height.
With this notation and \eqref{trucplusimp}, equation
\eqref{eq:3steps3} is proved if one has w.h.p.
\begin{equation}\label{labba}
\frac{1}{L}H_L(x,L^2\gep)\ge h(x/L,(\alpha+\gd)\gep) \;\text{for
every\;} x\in (-\xi L,\xi L).
\end{equation}
Now we want to relate the dynamics of $H_L$ to that of Theorem
\ref{spo}. 
 The relation is almost identical to that discussed in Remark
\ref{rem:131}, except for a slight difference in the way particles of
types $A$ and $B$ annihilate in the zero-range process.  Given
$\bbZ\ni x=-2L\xi+1/2,\dots,2L\xi-1/2$, we say again that there are
$n>0$ particles of type A at time $t$ at site $x$ if $\lim_{y\to
  x^+}H_L(y,t) -\lim_{y\to x^-}H_L(x,t)=n$ and that there are $n>0$
particles of type B if the same difference equals $-n$. Then it is
easy to realize that, under the dynamics $(\sigma^{(3)}(t))_{t\geq 0}$, each
particle performs a symmetric simple random walk with jump rate
$1/(2n)$ both to right or left (with $n$ the occupation number of the
site where the particle is), and that particles of different type
annihilate immediately if they are \emph{at sites of distance $1$} (and not
\emph{on the same site}):
this is the effect of flipping instantaneously "$-$" spins with more
than two "$+$" neighbors. Note also that, due to convexity of $\bar
h(\cdot)$, particles of type A are always to the left of particles of
type B. Therefore, if we take $H_L(\cdot,t)$ and we eliminate one of
the columns of maximal height (see Figure \ref{fig:unecolonenmoin})
(note that there are always at least two), the modified height
function thus obtained follows exactly the evolution of Theorem
\ref{spo}. 

 \begin{figure}[hlt]

\leavevmode
\epsfxsize =10 cm
\psfragscanon 
\psfrag{suprcolon}{\small Supressed column}
\psfrag{0}{\small $0$}
\psfrag{N}{\small $N$}
\psfrag{N-1}{\small $N-1$}
\epsfbox{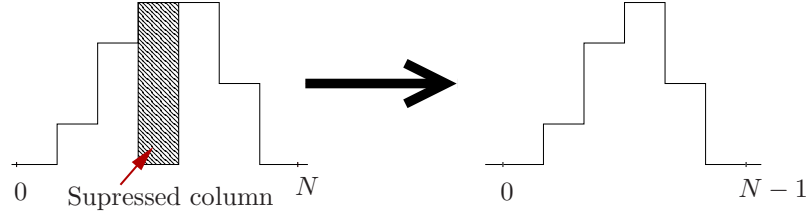}
\begin{center}
\caption{\label{fig:unecolonenmoin} 
Left: the height function associated to the ``$+/-$'' boundary for the dynamics
$\sigma^{(3)}(t)$. Right: the same height function, with one of the
highest columns removed; this follows the same evolution as in Theorem \ref{spo}. 
The fact that the new interface is step shorter makes no difference in
the macroscopic limit.}
\end{center}
\end{figure}

Of course, the erased column does not change the
scaling limit so that one can apply Theorem \ref{spo} and Corollary \ref{spo2} and get that
for any $t$ and $\eta>0$, w.h.p.
\begin{equation}
\frac{1}{L}H_L(x,L^2t)\ge  \bar\phi( x/L, t)-\eta\text{\;for
every\;} x\in\{-2\xi L,\dots,2\xi L\},
\end{equation}
where
\begin{equation}
\begin{cases}
\partial_t \bar\phi(x,t)&=\frac{1}{2}\partial^2_x \bar\phi(x,t)
\\
 \bar\phi( 2\xi, t)&=\bar\phi(-2\xi,t)=\bar h(2\xi)\\
 \bar \phi(x, 0)&=\bar h(x) \quad \text{for every }  x\in [-2\xi,2\xi].
\end{cases}
\end{equation}
Therefore, \eqref{labba} is proved if one can check that
\begin{equation}
\bar\phi( x, \gep)>h(x, (\alpha+\gd)\gep)\;\;\text{for every\;\;} x\in [-\xi ,\xi ].
\end{equation}
The above equation is proved is the same manner as \eqref{letrucpareildelautresection}: one just needs to choose  $\xi$  small enough.

\qed

\subsubsection{Proof of \eqref{eq:3steps2}}
\label{sec:po2}
First of all, one freezes to "$-$"
all the spins on the cross-shaped region
of sites in $L\,U$ (cf.\ \eqref{UAB}) such that at least one of their coordinates
is $\pm1/2$.
Equation \eqref{trucbien} guarantees that if $\gep$
is chosen small enough, w.h.p. the so-obtained dynamics coincides with the
original one up to time $\gep L^2$ if $\gep$ is small enough.

Then one defines $(\sigma^{(4)}(t))_{t\geq 0}$ as the dynamics obtained by
changing the initial condition in the following
manner: all spins $(x,y)\in L\DDD$ with either $|x|\geq L(1-\nu)$ or
$|y|\geq L(1-\nu)$ are changed from "$-$" to "$+$" (recall that $\nu$ is the constant
that enters the definition \eqref{UAB} of $U$) and therefore they stay ``$+$'' forever, since they have at least three ``$+$'' neighbors.
Note that, this way, the evolution in each quadrant of $(\bbZ^*)^2$ is 
independent. 
By monotonicity, we get that w.h.p, for every $t\le \gep L^2$,
\begin{equation}
\bigcup_{\{x : \sigma^{(4)}_x(t)=-\}} \mathcal C_x=:\mathcal A_L^{(4)}(t)\subset \mathcal A_L(t).
\end{equation}
and therefore \eqref{eq:3steps2} is proved if one can show that
\begin{equation}\label{finifini}
LA_1(\gep)\subset  \mathcal A_L^{(4)}(\gep L^2).
\end{equation}

Next, note that $\partial \mal^{(4)}(t)\cap \bbR_+^2$ 
in the coordinate system $({\bf f}_1,{\bf f}_2)$ is the graph of a
random piecewise affine function
\[
F_L:\left[-\frac{(1-\nu)L}{\sqrt{2}}, \frac{
    (1-\nu)L}{\sqrt{2}}\right]\mapsto \bbR
\]
which undergoes the corner-flip dynamics described in Theorem 
\ref{mickey} (apart from space  rescaling by a factor $\sqrt 2 $). For this reason on gets that, w.h.p.,
\begin{equation}
\label{inviu}
\lim_{L\to \8}\sup_{x\in[-\frac{1-\nu}{\sqrt{2}},\frac{1-\nu}{\sqrt{2}}]}\sup_{t\leq \gep}\left|\frac{1}{L}F_L(xL,tL^2)-g(x,t)\right|=0
\end{equation}
where 
\begin{equation}
 \begin{cases}
 \partial_t g(x,t)&=\frac{1}{4}\partial^2_x g(x,t)\\
 g(-\frac{1-\nu}{\sqrt{2}},t)&=g(\frac{1-\nu}{\sqrt{2}},t)=\frac{1-\nu}{\sqrt{2}}\\
 g(x,0)&=\bar f(x) \quad \text{for every } x\in \left[-\frac{1-\nu}{\sqrt{2}},\frac{1-\nu}{\sqrt{2}}\right]
\end{cases}
\end{equation}
and
$\bar f$ is the profile of the initial condition, i.e.
\begin{equation}
\bar f(x):= \min(f_0(x), (1-\nu)\sqrt{2}-|x|)).
\end{equation}
Let $P$ (resp. $P_1$) be the point on $\partial\mathscr D$ whose
coordinates $(x,y)$ (resp. $(x_1,y_1)$) in the coordinate system
$(\be_1,\be_2)$ satisfy $x>0,y=1-\nu$ (resp. 
$x_1>0,y_1=h_0(\xi)$).  Call $-d<0$ (resp. $-d_1<0$) the horizontal
coordinate of $P$ (resp. of $P_1$) in the coordinate system
$(\bff_1,\bff_2)$ .

In view of \eqref{inviu} and of Definition \eqref{UAB} of $A_1(\gep)$, equation \eqref{finifini} is satisfied if 
\begin{equation}
  g(x,\gep)>f(x,(\alpha+\gd)\gep) \quad\text{for every}\quad x\in (-d_1,d_1),
 \end{equation}
 The proof of this is very similar to that of \eqref{cczzcc} provided
 that $\bar f$ coincides with $f_0$ in a domain containing
 strictly $(-d_1,d_1)$ (this guarantees for instance that
 $\partial^2_x \bar f(\cdot)$ is uniformly continuous in a domain containing
 $(-d_1,d_1)$, so that the drift $\partial_t g$ is continuous in time,
 cf.\ Lemma \ref{lem:contilaplacien}).  For this to hold, it is
 enough to assume that $d>d_1$, i.e.\ that $\nu$ in
 \eqref{UAB} has been chosen sufficiently small as a function of $\xi$
 so that $1-\nu>h_0(\xi)$. \qed

\section{Proof of Theorem \ref{th:deterministico}: existence of
  anisotropic curve-shortening flow with convex initial condition }
\label{sec:convexbu}
Let us first recall some properties of the support function $h(\cdot)$
of a convex curve $\gamma$.  First of all, 
if $\gamma$ is contained in the convex set delimited by $ \gamma'$
then $h(\theta)\le h'(\theta)$ for every $\theta$.  Next, the support
function is related to the curvature and to the length $L(\gamma)$ of
$\gamma$ by (cf.\ \cite[Lemma 1.1]{GL2})
\begin{gather}
  \label{eq:1}
  \partial^2_\theta h(\theta)+h(\theta)=\frac1{k(\theta)}\\
L(\gamma)=\int_0^{2\pi}h(\theta)\dd\theta=\int_0^{2\pi}\frac1{k(\theta)}\dd\theta.
\label{eq:lungh}
\end{gather}
Also (cf.\ Lemma 4.1.1 in \cite{GH}, with the warning that what they call
$\theta$ is $\theta-\pi/2$ for us), the Cartesian coordinates
$(x(\theta),y(\theta))$ of the point of $\gamma$ where the outward
directed normal forms an anticlockwise angle $\theta$ with the
positive horizontal axis can be expressed as 
\begin{gather}
  \label{eq:3}
  x(\theta)=h(0)-\int_0^\theta \frac{\sin(s)}{k(s)}ds\\
  \label{eq:3bnis}
y(\theta)=h(\pi/2)+\int_{\pi/2}^\theta \frac{\cos(s)}{k(s)}ds.
\end{gather}
Under the flow \eqref{eq:meanc},
the time derivatives of area and length are  (cf.\ \cite[Lemma
2.1]{GL2})
\begin{gather}
  \label{eq:area}
  \frac d{dt} Area(\gamma(t))=-\int_{0}^{2\pi} a(\theta)\dd\theta\\
\label{eq:length}
  \frac d{dt} L(\gamma(t))=-\int_{0}^{2\pi} a(\theta)k(\theta,t)\dd\theta.
\end{gather}
For the moment these are formal statements since we do not know yet that
the flow exists.

\subsection{Proof of Theorem \ref{th:deterministico}}
Uniqueness of the flow is trivial, so we concentrate on existence.
First of all, we need to regularize the functions $a(\cdot)$ and $k(\cdot)$. 
Given $ 0<w< 1$ we define $a^{(w)}(\cdot)$ to be a
family of smooth approximations of the anisotropy function $a(\cdot)$.
More precisely:
\begin{assumption}\
  \begin{enumerate}
\item  $a^{(w)}(\cdot)$ is  
 $2\pi$-periodic and $C^\infty$;

\item $a^{(w)}(\theta)\stackrel{w\to0}\longrightarrow a(\theta)$
uniformly in $\theta$;

\item for fixed $\theta$, the function $w\mapsto a^{(w)}(\theta)$ is
  non-increasing;

\item the function $a^{(w)}(\cdot)$ is Lipschitz, uniformly in $w>0$
(this is possible because the function $a(\cdot)$ itself is $1-$Lipschitz);

\item the functions $w\mapsto \|\partial^2_\theta a^{(w)}\|_\infty:=
\max_\theta|\partial^2_\theta a^{(w)}(\theta)|
$ and $w\mapsto \|\partial^3_\theta a^{(w)}\|_\infty
$ are bounded, uniformly for $w$  in any  compact subset of $(0,1)$.
\end{enumerate}
\end{assumption}
\label{ass:a}
A possible choice is 
\[
a^{(w)}(\theta)=(a* g^{(w)})(\theta)+\gep_w
\]
where $g^{(w)}$ is a centered Gaussian of variance $w^2$. In the convolution
it is understood that $a(\cdot)$ is seen as a $2\pi$-periodic function on 
$\mathbb R$ and $\gep_w$ 
is chosen so that $a^{(w)}(\cdot)$ satisfies
the monotonicity with respect to  $w$. It is easy to check that one can choose
$\gep_w=-C w$ for some suitably large $C$. Indeed, monotonicity
in $w$ is guaranteed if for $w'<w$ one has
\[
\gep_{w'}-\gep_w\ge \|a* (g^{(w)}-g^{(w')})\|_\infty.
\]
On the other hand, since $a(\cdot)$ is Lipschitz, one sees easily that
$\|a* (g^{(w)}-g^{(w')})\|_\infty=O(w-w')$.

Also, we approximate $\gamma$ with a sequence of  convex
curves $(\gamma^{(w)})_{0<w<1}$ that satisfy the following properties:
\begin{assumption}\
\label{ass:h}
  \begin{enumerate}
\item $\gamma^{(w)}\supset \gamma^{(w')}\supset \gamma$ or equivalently $h^{(w)}(\cdot)\ge h^{(w')}(\cdot)\ge
  h(\cdot)$ if $0<w'<w$;
\item $\lim_{w\to 0} h^{(w)}(\cdot)=h(\cdot)$ uniformly in $\theta$,
  so that $\gamma$ is the limit of $\gamma^{(w)}$ in the topology of
  the Hausdorff distance;
\item the Lipschitz constant $\mathbb L(k^{(w)})$ of the curvature
  function $k^{(w)}(\cdot)$ is finite uniformly in $w$, $k^{(w)}(\cdot)\to
k(\cdot)$ uniformly  and
$\limsup_{w\to0}\mathbb L(k^{(w)})\le \mathbb L(k)$; 
\item The three first derivatives with respect to  $\theta$ of $k^{(w)}(\theta)$
  are bounded uniformly for $w$ in any compact subset of $(0,1)$.
\end{enumerate}
\end{assumption} 
(Like for the regularization of $a(\cdot)$ into $a^{(w)}(\cdot)$, a
possible
construction of $h^{(w)}(\cdot)$ is obtained convolving $h(\cdot)$
with a Gaussian of variance $w^2$ and adding a suitable constant
$\gep_w$).

For the regularized
mean curvature motion, it follows from \cite{GL2} that the equation
\begin{eqnarray}
  \label{eq:4bis}
  \begin{cases}
\partial_t h^{(w)}(\theta,t)=- a^{(w)}(\theta)k^{(w)}(\theta,t)  \\
h^{(w)}(\theta,0)=h^{(w)}(\theta)    
  \end{cases}
\end{eqnarray}
admits a solution corresponding to a flow of curves $(\gamma^{(w)}(t))_{t\ge0}$
which remain convex and shrink to a point in a finite time
\[
\tilde t_f:=t_f^{(w)}=Area(\gamma^{(w)}(0))/\int_{0}^{2\pi} a^{(w)}(\theta)\dd\theta
\]
(cf.\ \eqref{eq:area} with $a(\cdot)$ replaced by $a^{(w)}(\cdot)$).
For lightness of notation, we will often write $\tilde
h(\cdot,\cdot),\tilde \gamma(t),\tilde a(\cdot)$, etc. for the
regularized quantities 
 $h^{(w)}(\cdot,\cdot),\gamma^{(w)}(t),a^{(w)}(\cdot)$, etc.
Thanks to Assumption \ref{ass:a}, we have
that $\int_{0}^{2\pi} a^{(w)}(\theta)\dd\theta\to \int_{0}^{2\pi} a(\theta)\dd\theta=2$
as $w\to0$ and therefore
$t_f^{(w)}=t_f(1+o(1))$ when $w\to0$, with $t_f$ defined in Theorem \ref{th:deterministico}.

From \eqref{eq:1} and \eqref{eq:4bis} one can check that the curvature satisfies
the parabolic equation
\begin{eqnarray}
  \label{eq:eqcurv}
  \begin{cases}
 \partial_t \tilde k=\tilde k^2\partial^2_\theta(\tilde a \tilde k)+\tilde a\tilde k^3\\
\tilde k(\theta,0)=\tilde k(\theta).  
  \end{cases}  
\end{eqnarray}
Also, following \cite{GH} it is possible to see that the curvature function stays
$C^\infty$ until $\tilde t_f$ (since $\tilde a$ is
$C^\infty$). However, estimates on the regularity will \emph{not} be
necessarily uniform in the regularization parameter $w$ and we will
need to be very careful on this point.

\medskip

  For fixed $t$, set
\begin{eqnarray}
  \label{eq:gammadef}
\gamma(t):=\lim_{w\to 0} \gamma^{(w)}(t) 
\end{eqnarray}
where convergence is in the Hausdorff metric. 
\textit{A posteriori}, since we will see that $(\gamma(t))_t$ provides the (unique)
solution to our curve-shortening equation, it follows that the limit \eqref{eq:gammadef} does not depend on the choice of regularization.
Existence of the limit  (in the Hausdorff metric) along sub-sequences is
guaranteed by the Blaschke selection theorem
\cite[Th. 32]{Egg} which says that a family of convex subsets of a
bounded subset of $\bbR^n$ admits a sub-sequence converging to a
non-empty convex set. Uniqueness of the limit follows from the fact that $\gamma^{(w')}(t)\subset \gamma^{(w)}(t)$
if $w'<w$ and $t<t_f^{(w')}$ (because
$a^{(w)}(\theta)$ is decreasing in $w$ and that the curve is smooth at
all times).  One has to use Convergence in Hausdorff distance also holds for
the boundary curves.

Since the volume is continuous in the topology induced
by the Hausdorff metric \cite[Ch. 4]{Egg}
we also see that $Area(\gamma(t))=Area(\gamma)-t\int_0^{2\pi}a(\theta)
\dd\theta=Area(\gamma)-2t$; for $t\to t_f$ the curve $\gamma(t)$ shrinks
to a point (its diameter shrinks to zero).
We will prove 
\begin{theorem}
\label{th:detsmooth}
The flow of curves $(\gamma(t))_{t<t_f}$ defined in \eqref{eq:gammadef}
is a classical
solution of the anisotropic curve shortening flow \eqref{eq:meanc}
for $0\le t<t_f$.
\end{theorem}

\begin{definition}
  For $t<\tilde t_f$ let $\tilde k_{\rm max}(t)$ (resp. $\tilde k_{\rm
    min}(t)$) be the maximal (resp. minimal) curvature of $\tilde
  \gamma(t)$. We let $\tilde k_{\rm max}:=\tilde k_{\rm max}(0)$ and
  similarly for $\tilde k_{\rm min}$ and $\tilde a_{\rm max
    (min)}:=\max_\theta(\min_\theta) \tilde a(\theta)$.  Also,
  $k_{\rm min(max)}$ and $a_{\rm min(max)}$ are defined similarly to $\tilde
  k_{\rm min(max)},\tilde a_{\rm min(max)}$ but with $\tilde k(\cdot),\tilde
  a(\cdot)$ replaced by $k(\cdot),a(\cdot)$.
\end{definition}

 It is crucial that
$\tilde k_{\rm max}(t)$ stays bounded, uniformly for $w$ small, as long
as the disappearance time is not approached:
\begin{proposition}[Regularity estimate] Assume that the curvature
  function $k(\cdot)$ is Lipschitz. There exists $w_0>0$ such that,
\label{prop:curv}
for every  $b>0$, $t<t_f(1-b)$, $0<w\le w_0$  one has
\begin{eqnarray}
  \label{eq:maxcurv}
  \tilde k_{\rm max}(t)\le C_1
\end{eqnarray}
and
\begin{eqnarray}
  \label{eq:grad}
\max_\theta|\partial_\theta (\tilde a(\theta)\tilde k(\theta,t))|\le C_2(\mathbb L(k)+1)
\end{eqnarray}
where we recall that $\mathbb L(k)$ is the Lipschitz constant of the function $k(\cdot)$.
The constants $C_1$ and $C_2$ depend only on $b$ 
and on $k_{\rm max}$.
  \end{proposition}




\begin{proof}[Proof of Proposition \ref{prop:curv}]
The proof is based on ideas of  \cite{GL2}. However, 
it is important to make sure that estimates
are uniform in $w\le w_0$ (in \cite{GL2} the anisotropy function $a(\cdot)$ is
assumed to be $C^2$, so there was no need to regularize it).

  Fix $w> 0$.
  First we get a lower bound on $\tilde k_{\rm min}(t)$.
Note first of all that at time zero the minimal curvature is bounded
away from zero (uniformly in $w$): indeed, using \eqref{eq:lungh}
and the fact that the curvature function is $\mathbb L(k)$-Lipschitz,
\begin{gather}
\label{eq:kminimo}
L(\gamma(0))=
\int_0^{2\pi} \frac{1}{k(\theta)}\dd \theta\ge  2\int_0^{\pi} \frac{1}{k_{\min}+\mathbb L(k)\theta}\dd \theta
=\frac{2}{\mathbb L(k)}\log \frac{\mathbb L(k)\pi+ k_{\rm min}}
{k_{\min}}.
\end{gather}
Then, since the length of $\gamma(0)$ is finite, $k_{\rm min}$ must be positive.

Set for simplicity \[g=g(\theta,t)=\tilde a(\theta)\tilde k(\theta,t).\]
Formula \eqref{eq:eqcurv} gives
\begin{eqnarray}
  \label{eq:1GL}
\partial_t g=\frac{1}{\tilde a}
\left(g^2\partial^2_{\theta} g+g^3\right)=:g(\theta,t)\,u(\theta,t).
  \end{eqnarray}
 This, together with the fact that $\tilde a(\cdot)$ and $\tilde k(\cdot,t)$ are
smooth, 
 implies that
 \begin{eqnarray}
   \label{eq:kmin} 
 \frac d{dt}\min_{\theta}g\ge \frac{\min_\theta g^3}{ \tilde a_{\rm max}}
 \ge 0 
 \end{eqnarray}
 (at the minimum point the second derivative is positive) so that
 \begin{eqnarray}
   \label{eq:kmin2}
 \tilde k_{\rm min}(t)\ge \frac{\tilde a_{\rm min}}{ \tilde a_{\rm max}}\tilde k_{\rm min}\ge C k_{\rm min}>0   
 \end{eqnarray}
with  $C$ independent of $w$ (say for $w\le w_0$) thanks
 to the uniform convergence $a^{(w)}(\cdot)\to a(\cdot)$
 and $k^{(w)}(\cdot)\to k(\cdot)$. 

Next the real work: bounding  $\tilde k_{\rm max}(t)$ uniformly in $w$.
From \eqref{eq:1GL} one sees that, since $\tilde a(\cdot)$ and $\tilde k(\cdot,t)$ are smooth, 
\begin{eqnarray}
  \label{eq:2GL}
\frac d{dt}\max_\theta g\le \frac1{\tilde a_{\rm min}}\times (\max_\theta g)^3.
\end{eqnarray}
From this one immediately gets that $\tilde k_{\rm max}(t)$ is upper bounded
uniformly in $w\le w_0$, up to some time $t_1$ depending only on
$k_{\rm max}$.  However the solution of $\dot x=x^3$ explodes in finite
time, and certainly before the time $\tilde t_f$ when the curve shrinks to a
point, so we need to do better.

For this, 
we define $z(t)=\min_\theta  u(\theta,t)$ (cf.\ \eqref{eq:1GL}). Then,
taking the derivative of $u$ with respect to $t$ shows (cf.\ Lemma 4.2 of
\cite{GL2} for details) that \[\frac{\dd}{\dd t}z(t)\ge 2 z(t)^2\] so that if
$z(0)\ge0$ we get $z(t)\ge 0$, if $z(0)\le 0$ we get $z(t)\ge
-1(1/|z(0)|+2t)$.  Altogether, we get that 
\begin{equation}\label{eq:3GL}
u(\theta,t)\ge -\frac1{2t}
\end{equation}
uniformly in $\theta$ and $w\le w_0$.
Now we use this to get a uniform bound on $\|\partial_\theta g\|_\infty$ in
terms of $\tilde k_{\rm max}(t)$.
Without loss of generality suppose that there exists $\theta_1$ such that $\partial_\theta g(\theta_1,t)=\|\partial_\theta g\|_{\infty}$
(if this is not the case one can still find $\theta_1$ such that $\partial_\theta g(\theta_1,t)=-\|\partial_\theta g\|_{\infty}$ and apply the same method).
Let also $\theta_2>\theta_1$ be such that $\partial_\theta
g(\theta_2,t)=0$ (such an angle exists since $g$ is periodic). Then, from
the definition \eqref{eq:1GL} of $u$,
\begin{multline}\label{drgf}
\|\partial_\theta g\|_{\infty}
=-\int_{\theta_1}^{\theta_2}\partial^2_\theta g
\dd\theta=-\int_{\theta_1}^{\theta_2} \left(\frac{u(\theta,t)}{\tilde
    k(\theta,t)}-\tilde a(\theta,t)\tilde k(\theta,t)
\right) \dd \theta \\
\le 
\frac{1}{2t}\int_{\theta_1}^{\theta_2}\frac{\dd \theta}{\tilde k(\theta,t)}
+(\theta_2-\theta_1)\tilde a_{\max}\tilde k_{\max}(t)\le 
\frac{L(\gamma(0))}{t}+C_4\tilde k_{\rm max}(t).  
\end{multline}
In the last inequality we used \eqref{eq:lungh} and then
\eqref{eq:length} which says that $L(\tilde \gamma(t))\le L(\tilde
\gamma(0))\le 2 L(\gamma(0))$.
Since $g=\tilde a \tilde k$ and by assumption $\tilde a$ is $C^\infty$ and
Lipschitz
uniformly in $w$, one deduces that 
\begin{equation}
\label{eq:ifyou}
\|\partial_\theta \tilde k  \|_\infty
\le \frac{L(\gamma(0))}{t}+C_5\tilde k_{\rm max}(t)\le C_6(t)\tilde k_{\rm max}(t)
\end{equation}
and $C_6$ can be chosen to be decreasing in $t$.
From this it is
trivial to see that, if $\theta_0$ is such that
$\tilde k(\theta_0,t)=\tilde k_{\rm max}(t)$, one has
\begin{equation}
  \label{eq:astut}
\tilde k(\theta,t)\ge \tilde k_{\rm max}(t)/2\;
\text{whenever $|\theta-\theta_0|\le \alpha(t)$ }
\end{equation}
for some $\alpha(t)$ increasing in $t$ (it could vanish for $t\to 0$).
Next, one proves that for $t<(1-b)t_f$
\begin{eqnarray}
  \label{eq:entropia}
E(t):=\int_0^{2\pi}\tilde a(\theta)\log(g(\theta,t))\dd\theta \le C_7  
\end{eqnarray}
where $C_7$ depends only on ${a_{\rm max}}$ and on $b$ and on the
maximal curvature $k_{\rm max}$ of the initial curve $\gamma(0)$.
Indeed, \eqref{eq:entropia} is obvious for $t=0$, since the initial
curvature is bounded by assumption. To get the control for $t>0$, one
observes (cf.\ Propositions 5.3 and 5.4 of \cite{GL2}) that
\[
\frac{\dd }{\dd t}E(t)\le 2 \tilde a_{\rm max} \frac{L(\tilde\gamma(0))}
{Area(\tilde\gamma((1-b)t_f))}\left(
-\frac{\dd}{\dd t} L(\tilde \gamma(t))\right).
\]
The prefactor is bounded since $b>0$ and the time-integral of the time-derivative
of the length gives at most $L(\tilde \gamma(0))$.
At this point we are almost done: 
using \eqref{eq:astut}
\begin{gather}
  \label{eq:b2}
C_7\ge \int_0^{2\pi}\tilde a(\theta)\log(g(\theta,t))\dd\theta \\
\ge 2\alpha(t)\tilde a_{\rm min}\log (\tilde a_{\rm min}\tilde
k_{\rm max}(t)/2)+2\pi \tilde a_{\rm max}\log [\min(1,\tilde
a_{\rm min}\tilde k_{\rm min}(t))]
\end{gather}
and this (recall that $\tilde k_{\rm min}(t)\ge C k_{\rm min}>0$, cf.\ \eqref{eq:kmin2}) gives us an upper bound on $\tilde
k_{\max}(t)$ 
uniformly in $w\le w_0$ and $t<(1-b)t_f$: up to $t_1$ one uses the upper bound
which comes from \eqref{eq:2GL} and after $t_1$ the one from
\eqref{eq:b2}; Eq. \eqref{eq:maxcurv} is proven. 
 When $t$ approaches
the disappearance time $\tilde t_f$ (i.e.\ when $b$ approaches zero) the
upper bound diverges (because $C_7$ diverges), as it should.

Equation \eqref{eq:ifyou} says that the curvature function is Lipschitz
with a Lipschitz constant $C$ that depends on $t,b$ and $L(0)$ but not
on $w$.  This is not yet the desired \eqref{eq:grad} because 
the bound diverges for $t\to0$. To prove \eqref{eq:grad} 
  remark that,
 using \eqref{eq:1GL},
\begin{gather}
\partial_t\partial_{\theta} g=\partial_{\theta} \left( \frac{g^2}{\tilde a}
\partial^2_{\theta}g+\frac{g^3}{\tilde a}\right)
=-\frac{\partial_\theta \tilde a}{\tilde a^2}\left( g^2
\partial^2_{\theta}g+g^3\right)
\\+\frac{1}{\tilde a}\left(2g\partial_{\theta}g\partial^2_{\theta}g+g^2\partial^3_{\theta}g+3g^2\partial_{\theta} g\right).
\end{gather}
At the point where $\partial_{\theta}g$ is maximized, $\partial^2_{\theta}g$
 cancels and $\partial^3_{\theta}g$  is non-positive. This, together
with the boundedness of $g$ uniformly in $w\le w_0,\theta\in[0,2\pi]$ and $t<(1-b)t_f$, 
implies
\begin{equation}
 \partial_t\max_{\theta}\partial_{\theta} g(\theta,t)
\le C_8(1+\max_{\theta}\partial_\theta g(\theta,t)).
\end{equation}
where $C_8$ just depends on $k_{\rm max}$ and $b$.  Integrating
with respect to time, one gets
\[
\max_{\theta}\partial_{\theta} g(\theta,t)\le C_9\left[\max_{\theta}\partial_{\theta} (a^{(w)}(\theta)k^{(w)}(\theta))+1\right]
\]
with $C_9$ depending only on $C_8$. Also, observe that
\[\partial_{\theta} (a^{(w)}(\theta)k^{(w)}(\theta))\le (3/4)|\partial_\theta
k^{(w)}(\theta)|+C_{10}\le (3/4)\mathbb L(k^{(w)})+C_{10}\] with $C_{10}$ a constant depending on $k_{\rm max}$,
since for $w$ small $a^{(w)}_{\rm max}<(3/4)$ and $a^{(w)}$ is uniformly
Lipschitz.  
Finally, from Assumption \ref{ass:h}  (3), we can conclude
$\partial_{\theta} (a^{(w)}(\theta)k^{(w)}(\theta))\le C_{10}+\mathbb L(k) $
for $w$ small.
 An analogous lower bound can be found on
$ \partial_t\min_{\theta}\partial_{\theta} g(\theta,t)$ and this gives
\eqref{eq:grad}.

\end{proof}


Following \cite{GH} it is possible to prove that, once we have bounds
on the curvature and on $\|\partial_\theta g(\cdot,t)\|_\infty$, for every
$n\ge2$ and $t<t_f(1-b)$ the derivatives
$\partial^n_\theta g(\theta,t)$ are also bounded.  The bounds we get
are in general \emph{not uniform} in $w$ but this is not very
important for our purposes.  Indeed, we will need only:
\begin{proposition} \label{youhou} Fix
   $b>0$. There exists a  function $c(w)$, which is non-increasing with respect to
$w\in (0,w_0]$  such that for $t<(1-b)t_f$
  \begin{eqnarray}
    \label{eq:htt}
   \max_\theta| \partial^2_t \tilde h(\theta,t)|\le c(w).
  \end{eqnarray}

\end{proposition}
\begin{proof}
  Recall \eqref{eq:4bis} and \eqref{eq:1GL}:
\begin{equation}
 \partial^2_t \tilde h= -\frac{1}{\tilde a}(g^2\partial^2_{\theta}g+g^3).
\end{equation}
Thus we just have to bound $\partial^2_\theta g$, since we have
already proved that $g$ itself is bounded. 
For this, we adapt the method used by Gage and Hamilton in \cite{GH}
for the special case of the isotropic curve shortening flow where $a\equiv 1$.
What they observed \cite[Lemma 4.4.2]{GH}  is that, if the curvature and its $\theta$-derivative are bounded (which we proved in Proposition
\ref{prop:curv}), the $t$-derivative of $\Phi(t):=\int_0^{2\pi}[\partial^2_\theta g(\theta,t)]^4 \dd\theta$
can be upper bounded by a constant times $\Phi(t)$ itself and then one can
integrate the inequality with respect to $t$ to get a bound on $\Phi(t)$
in terms of $\Phi(0)$. In our case, with a
similar computation, we find that $(d/dt)\Phi(t)$ is upper bounded by
$\Phi(t)$ times a constant depending on $\|\partial_\theta
a^{(w)}\|_\infty$, which is finite uniformly for $w\le1$. 
Since $\Phi(0)$ is also bounded for $w$  in compact subsets of
$(0,1)$ (cf.\ Assumption \ref{ass:a} (5) and Assumption \ref{ass:h} (4)), we get that $\Phi(t)\le c_1(w)$
for $w\in(0,1)$ and $t<(1-b)t_f$ and we can choose $c_1$ to be
decreasing. In general, 
 $c_1$ will diverges when  $w$  approaches zero.

A similar computation (cf.\ \cite[Lemma 4.4.3]{GH} when $a(\theta)\equiv 1$) shows that 
\[\Psi(t):=\int_0^{2\pi}[\partial^3_\theta g(\theta,t)]^2 \dd\theta\le
c_2(w)\] with $c_2(\cdot)$ decreasing in  $w\in(0,1)$.
Then one uses the fact that for a smooth, $2\pi$-periodic function $f$ one has
(cf.\ \cite[Corollary 4.4.4]{GH})
\[
\|f\|_\infty^2\le  C\int_0^{2\pi}(f^2+(\partial_\theta f)^2)\dd\theta
\]
for some universal constant $C$, applied with $f(\cdot)=\partial^2_\theta g(\cdot,t)$,
to get that $\|\partial^2_\theta g\|_\infty\le c_3(w)$ as we wished.

\end{proof}

\begin{proof}[Proof of Theorem \ref{th:detsmooth}]
We are now ready to prove that $(\gamma(t))_t$ provides a classical solution
of \eqref{eq:meanc}.
This is based on the following easy consequence of the Arzel\`a-Ascoli
Theorem:
\begin{lemma}\label{Ascoli}
  Let $f^{(n)}$ be a sequence of periodic $C^1$ functions on $[0,2\pi]$,
such that both sequences $f^{(n)}$ and $\partial_x
  f^{(n)}$ are uniformly bounded and 
   equicontinuous. If $f^{(n)}\to f$ as $n\to\infty$, then $f$ is
$C^1$ and  
$\partial_x f=\lim_n \partial_x f^{(n)}$, where the convergence
is uniform and does not require sub-sequences.
\end{lemma}

First of all, we note that $\tilde h(\cdot,t)$ does converge (for
$ w\to0$) to $h(\cdot,t)$ for every fixed $t<t_f$. This just
follows from the fact that $\tilde \gamma(t)$ converges to $\gamma(t)$
in terms of Hausdorff distance. Furthermore, convergence is uniform in
$t<t_f(1-b)$ for every fixed $b$. 
This is true because the area difference between
$\tilde\gamma(t)$ and $\gamma(t)\subset \tilde\gamma(t)$ is small with $w$ (uniformly in $t$)
and the curvature is uniformly bounded: then, if $\tilde h(\theta,t)-h(\theta,t)$ 
where larger than some $\delta$ independent of $w$ for some
$(\theta,t)$, necessarily the area difference would be larger than
some $c(\delta)$ at that time.

Applying Lemma \ref{Ascoli} and recalling \eqref{eq:1}, we get that,
for $t$ fixed, $\partial_\theta \tilde h(\theta,t)$
and $\tilde k(\theta,t)$ converge to $\partial_\theta h(\theta,t)$ and
$k(t,\theta)$ respectively and that convergences are uniform in
$\theta$ (knowing that the curvature is Lipschitz is important here).
Note by the way that $k(\cdot,t)$ is Lipschitz, since
$\|\partial_\theta \tilde k(\cdot,t)\|_\infty$ is uniformly bounded.

Then applying dominated convergence (which is allowed in view of
Proposition \ref{prop:curv}), one gets that 
\begin{equation}\label{versionintegral}
 h(\theta,t)- h(\theta,s)=-\int_{s}^t a(\theta)k(\theta,u)\dd u,
\end{equation}
which is an integrated version of \eqref{eq:meanc}. 
To get the stronger statement \eqref{eq:meanc}, we need to prove that
$k(\theta,t)$ is continuous as a function of $t$.
\medskip

First of all, we prove that one can find a function $\gep:(0,1)\ni
w\mapsto\gep(w)\in\bbR_+ $, increasing and going to zero as $w\to0$
such that for all $\theta$, for all $t\le (1-b)t_f$,
\begin{equation}\label{trebien}
 |\tilde k(\theta,t)-k(\theta,t)|\le \gep(w).
\end{equation}
If this were not the case then, thanks to the fact that
$\tilde k(\cdot,t)$ and $k(\cdot,t)$ are uniformly Lipschitz, we would
have, say, for arbitrarily small $w$ and for some $\gep>0$, $t<t_f(1-b)$,
\[
 \tilde k(\theta,t)-k(\theta,t)\ge \gep
\]
for $\theta\in[\bar\theta,\bar\theta+\gep]$ for some $\bar\theta\in[0,2\pi]$.
But then, since
(cf.\ \eqref{eq:1}) 
\[
(\partial^2_\theta+1)(h(\theta,t)-\tilde h(\theta,t))=\frac1{k(\theta,t)}-\frac1{\tilde k(\theta,t)}, 
\]
this would contradict the uniform convergence of $\tilde
h(\cdot,\cdot)$ to $h(\cdot,\cdot)$.

On the other hand, from  Proposition \ref{youhou},
 for all $\theta$ and  for all
$t,s\le (1-b)t_f$
\begin{equation}
| \tilde k(\theta,t)-\tilde k(\theta,s)|\le c(w)|t-s|. 
\end{equation}
Together with \eqref{trebien} this implies that
\begin{equation}\label{unifcon}
 |k(\theta,t)-k(\theta,s)|\le \inf_{w} \left(2\gep(w)+c(w)|t-s|\right).
\end{equation}
The right-hand side clearly tends to zero with $|t-s|$ 
(choose a sequence $\{w_k\}$ tending to zero. If 
$c(w_k)$ does not diverge we are done. Otherwise, compute
the right-hand side for the $w=w_k$ with the largest value
of $k$ such that $c(w_k)\le |t-s|^{-1/2}$).
This shows that $t\mapsto k(\theta,t)$ is continuous away
from $t_f$ and the proof is complete.
\end{proof}

\section{Proof of Theorem \ref{th:convex}: evolution of a convex
  droplet }
\label{sec:thconv}

The proof is very similar to that of Theorem \ref{mainres} in the
scale-invariant case (Section \ref{infinitez}), and therefore it will be only sketched. We will also try to use as much as possible the
same notations as in Section \ref{infinitez}.

First we present two statements that are analogous to Propositions \ref{trucrelou} and \ref{mainprop}:
\begin{proposition}
\label{trucrelouc}
Let $\DD$ be convex with a bounded curvature function.
For every $\alpha>0$, w.h.p.
\begin{equation}\label{tipc}
\mathcal{A}_L(L^2 t) \subset L\, \mathcal{D}^{(\alpha)}\text{\;for every\;} t\ge0
\end{equation}
(recall definition \eqref{eq:cdelta}). Moreover, for every $\alpha>0$ there exists $\gep_1(\alpha,k_{\rm max})>0$ such that  w.h.p 
\begin{equation}\label{topc}
 \mathcal{A}_L(L^2 t) \supset L\,\mathcal{D}^{(-\alpha)} \text{\;for every\;} t\in[0,\gep_1].
\end{equation}
\end{proposition}
\begin{proof}
  The proof of \eqref{tipc} is essentially identical to that of \eqref{tip}, so we give no detail.  
As for \eqref{topc}, given $\alpha$ it is possible to give a finite collection
$\{\DDD_i\}_i$ such that:
\begin{itemize}
\item each $\DDD_i$ is an open convex subset of $\bbR^2$, obtained from (the interior
  of) the invariant shape $\DDD$ via a suitable translation and shrinking;

\item $\DDD_i\subset \DD$ for every $i$;

\item $\cup_i \DDD_i\supset \DD^{(-\alpha/2)}$.
\end{itemize}
Given $\eta>0$, thanks to Proposition \ref{trucrelou} there exists
$\gep>0$ such that, w.h.p., for every $t<\gep$ one has
\[
\mal(L^2 t)\supset \cup_i \left(L\,\DDD_i^{(-\eta)}\right).
\]
Here we use monotonicity (because $\DDD_i\subset \DD$) and the fact
that the union of a finite number of events which occur w.h.p. still
has probability tending to $1$. Note that the choice of $\gep$ is
depending on $\eta$ but also on the diameter of the
smallest set in the collection $\{\DDD_i\}_i$ and consequently on
$k_{\rm max}$.    Then, if
$\eta$ is small enough (depending on $\alpha$) it is clear that
$\cup_i \DDD_i^{(-\eta)}\supset \DD^{(-\alpha)}$ (recall that the
$\DDD_i$ are open sets, so that every $x\in \DD^{(-\alpha)}$
is contained in the interior of at least one $\DDD_i$).
\end{proof}
\begin{proposition}
\label{mainpropc}
Let $\DD$ be convex whose  curvature function is $\mathbb L(k)$-Lipschitz and
is bounded away from zero and infinity.
  For all $\delta>0$ there exists $\gep_0(\delta,k_{\rm min},k_{\rm
    max},\mathbb L(k))>0$ such that for all
  $0<\gep<\gep_0$, w.h.p.,
\begin{equation}\label{hokc}
\mathcal{A}_L(L^2 \gep) \subset L\mathcal{D}(\gep(1-\delta)),
\end{equation}
and 
\begin{equation}\label{hicc}
\mathcal{A}_L(L^2 \gep) \supset L\mathcal{D}(\gep(1+\delta))
\end{equation}
where we recall that $\DD(t)$ is the set enclosed by the curve $\gamma(t)$.
\end{proposition}

\begin{proof}[Proof of Theorem \ref{th:convex} assuming Propositions
  \ref{mainpropc} and \ref{trucrelouc}]
It is enough to prove
 \eqref{eq:scaling} for $t<(1-b)t_f$ for arbitrary $b>0$. 
 Then, the statement for $t\ge (1-b)t_f$ and also \eqref{eq:drif} follows
from the fact that the disappearence time of a droplet of diameter
$\ell$ is w.h.p. $O(\ell^2)$ (recall that $\gamma(t)$ shrinks to a
point when
$t\to t_f$ in the sense that its diameter converges to zero).
Define 
$k^*_{\rm min}>0$ (resp. $k_{\rm max}^*,\mathbb L^*_k<\infty$) to be the infimum
(resp. maximum) of $k_{\rm min}(s)$ (resp. $k_{\rm max}(s), \mathbb L(k(s))$) on
$[0,(1-b)t_f]$.  Fix $\delta'$ small and let
$\gep<\gep_0(\delta',k^*_{\rm min},k_{\rm max}^*,\mathbb L^*_k)$ and $\gep<\gep_1(\gd/2,k^*_{\rm max})$
with $\gep_0,\gep_1$ defined in Propositions \ref{trucrelouc} and \ref{mainpropc}. Using the Markov
property and the monotonicity of our process   
we get that, w.h.p., for any $k$ such that $\gep
k<(1-b)t_f$
\begin{equation}
\label{soles}
\mathcal{A}_L(L^2 k\gep) \subset L\mathcal{D}(k\gep(1-\delta')).
\end{equation}

From \eqref{soles} and 
Proposition \ref{trucrelouc} we get that  w.h.p., for every $ t\leq (1-b)t_f$,
\begin{equation}
\mathcal{A}_L(L^2 t) \subset L\left[\mathcal{D}\left(\left\lfloor \frac{t}{\gep} \right\rfloor\gep (1-\delta')\right)\right]^{(\delta/2)} 
 \subset L\left[\mathcal{D}((t-\gep)(1-\delta'))\right]^{(\delta/2)}.
\end{equation}
 Setting $\gep'=t_f\gd'+\gep$ this implies that w.h.p.
\begin{equation}
 \mathcal{A}_L(L^2 t) \subset L\left[\mathcal{D}(t-\gep')\right]^{(\delta/2)} \quad  \text{for every}\quad t\leq (1-b)t_f.
\end{equation}
Finally observe (this follows from \eqref{eq:meanc}) that the
Hausdorff distance between $\mathcal{D}(t-\gep')$ and $\mathcal{D}(t)$
is at most $\gep' k^*_{\rm max}\max_\theta|a(\theta)|$ so that if 
 $\gep'$ is chosen such that
\begin{equation}
\gep' k^*_{\rm max}\max_\theta|a(\theta)|< \delta/2
\end{equation}
we get \eqref{eq:scaling}.

The lower bound is proven similarly and this is where one has to use
the assumption $\gep<\gep_1(\gd/2,k^*_{\rm max})$.
\end{proof}

\subsection{Upper bound: Proof of \eqref{hokc}}

\begin{definition}
Define $(P_i(t))_{i=1}^4$ to be the four ``poles'' of $\mathcal D(t)$, where
the tangent vector is either horizontal or vertical (recall that $\DD(t)$
is strictly convex at all times under our assumptions, cf.\ discussion after
\eqref{eq:kminimo}, so that the four poles are distinct and uniquely
defined).  $P_1(t)$ denotes the ``north pole'' and the others are
numbered in the clockwise order.  Denote by $(x(P_i(t)),y(P_i(t)))$
(resp. $(u(P_i(t)),v(P_i(t)))$) the coordinates of $P_i(t)$ in the coordinate system
$(\bof_1,\bof_2)$ (resp. $(\mathbf{e}_1,\mathbf{e}_2)$). When $t=0$ we omit
the time coordinate.
  
\end{definition}

 \begin{figure}[hlt]

\leavevmode
\epsfxsize =14 cm
\psfragscanon 
\psfrag{suprcolon}{\small Supressed column}
\psfrag{l1}{\small $\ell_1$}
\psfrag{l2}{\small $\ell_2$}
\psfrag{f1}{\small ${\bf f}_1$}
\psfrag{f2}{\small ${\bf f}_2$}
\psfrag{e1}{\small ${\bf e}_1$}
\psfrag{e2}{\small ${\bf e}_2$}
\psfrag{h}{\small $h(y,t)$}
\psfrag{f}{\small $f(x,t)$}
\psfrag{2xi}{\small $2\xi$}
\psfrag{x}{\small $x$}
\psfrag{B}{\small $A$}
\psfrag{A}{\small $B$}
\psfrag{y}{\small $y$}
\psfrag{P1}{\small $P_1$}
\psfrag{P2}{\small $P_2$}
\psfrag{P3}{\small $P_3$}
\psfrag{P4}{\small $P_4$}
\psfrag{N}{\small $N_1(\gep,\xi)$}
\psfrag{M}{\small $M_1(\gep,\xi)$}
\psfrag{o}{\small $0$}
\epsfbox{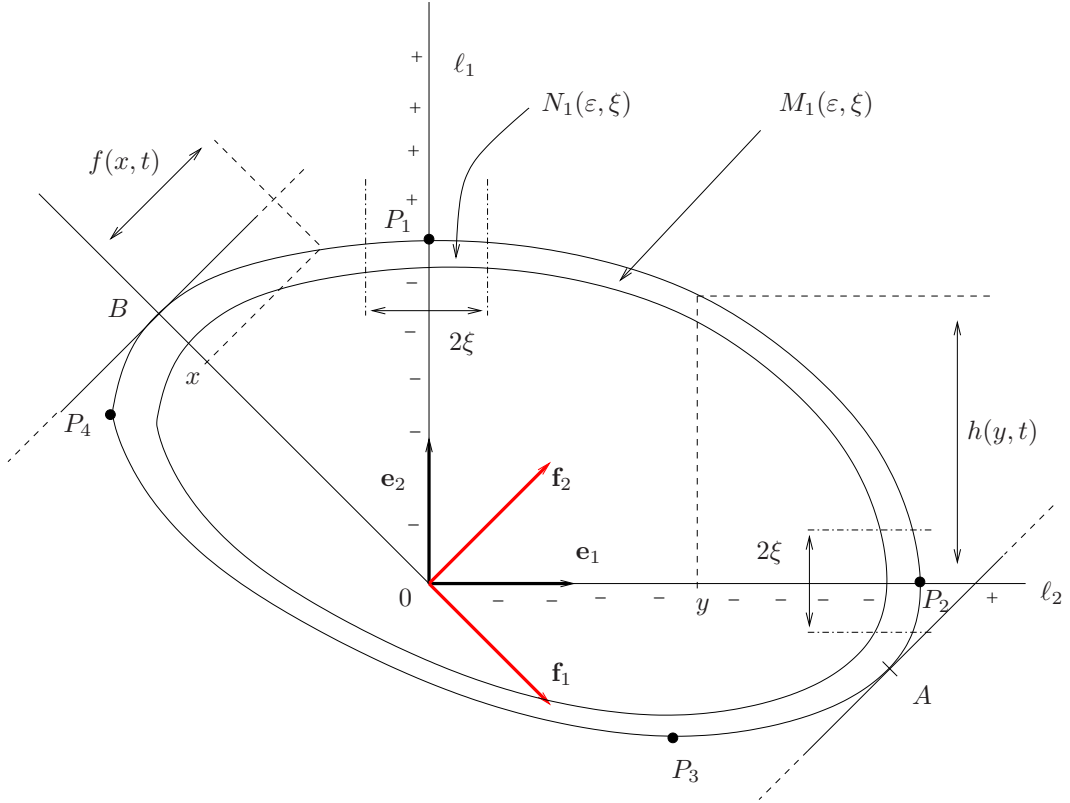}
\begin{center}
  \caption{\label{fig:novapatata} The larger convex set is $\DD$ and
    the smaller one is $\DD(\gep(1-\delta))$.  The poles $P_i$ of
    $\DD$ are marked with black dots (for convenience we have chosen $P_1$ one the vertical axis and $P_2$ on the horizontal one). The graph in $({\bf f}_1,{\bf
      f}_2)$ of the anti-clockwise portion of $\partial \DD$ between $A$ and
    $B$ is $f(\cdot,0)$ and the graph in $({\bf e}_1,{\bf e}_2)$ of
    the  portion of $\partial \DD$ between $P_4$ and $P_2$ is
    $h(\cdot,0)$.  For the proof of  \eqref{eq:Mc}, boundary spins to the left of $\ell_1$ are set to
    ``$-$'' below $P_1$ and ``$+$'' above; boundary spins below $\ell_2$ are set to
    ``$-$'' to the left of  $P_2$ and ``$+$'' to the right. }
\end{center}
\end{figure}

An equivalent formulation of \eqref{hokc} is: for all $\delta>0$ and $\gep$ small enough w.h.p.
\begin{equation}
\label{manierdevoirc}
 \sigma_x(\gep L^2)=+ \text{\;for every\;} x\in L\left[\mathcal{D}(\gep(1-\delta )) \right]^c.
\end{equation}
Given some small $\xi$ we divide $ \left[\mathcal{D}(\gep(1-\delta )) \right]^c$ in eight pieces $(M_i)_{i=1}^4$ and $(N_i)_{i=1}^4$ 
as follows (this is analogous
to the definition \eqref{MN} in the scale-invariant case, cf.\ Figure
\ref{label}):
 \begin{equation}
M_1(\gep,\xi):=  \big([u(P_1)+\xi,\infty)\times[v(P_2)+\xi,\infty)\big) \setminus \mathcal{D}(\gep(1-\delta))
\end{equation}
while $N_1(\gep,\xi)$ is 
the infinite component of  $([u(P_1)-\xi, u(P_1)+\xi]\times \bbR)\setminus \mathcal{D}(\gep(1-\delta))$ which contains $P_1$. The sets $M_i,N_i$ are defined
 analogously for $i=2,3,4$, so that
$\left[\mathcal{D}(\gep(1-\delta )) \right]^c=\bigcup_{i=1}^4(M_i\cup
N_i)$.  Equation \eqref{manierdevoirc} is proved if one can prove that
for every $i$, and $\varepsilon$ small enough, w.h.p.
\begin{align}\label{eq:Mc}
  \sigma_x(\gep L^2)&=+ \text{\;for every\;} x\in LM_i(\gep,\xi)\\
 \label{eq:Nc}  \sigma_x(\gep L^2)&=+ \text{\;for every\;} x\in LN_i(\gep,\xi).
\end{align}
Of course one can focus on  $i=1$, the other cases being obtained by a permutation of coordinates.

\subsubsection{Proof of \eqref{eq:Mc}}

We use the notation $f(\cdot,t)$ for the function whose graph in the
coordinate system $(\bof_1,\bof_2)$ is the portion of $\partial
\mathcal{D}(t)$ which goes in the anti-clockwise direction from point
$A$  where the tangent
forms an angle $\pi/4$ with the horizontal axis (cf.\ Figure \ref{fig:novapatata}) to the point $B$ 
where the angle is $(5/4)\pi$.
  The domain of
definition of $f(\cdot,t)$ decreases with time (because $\DD(t)$
shrinks) but for $t$ small enough it includes $[x(P_1),x(P_2)]$.  Let
$\mathcal D_1$ be the ``triangular-shaped'' region delimited by
$\partial\DD$, by the vertical line $\ell_1$ passing through $P_1$ and
by the horizontal line $\ell_2$ passing through $P_2$ (note that $
\mathcal D_1$ may not be included in $\mathcal D$).

We consider a modified dynamics in the north-east quadrant
$[Lu(P_1),\infty)\times[Lv(P_2),\infty)$ delimited by the lines
$L\ell_1,L\ell_2$.  All the spins are initially ``$-$'' in $L\mathcal
D_1$ and ``$+$'' otherwise. As for boundary spins, the spins at distance at most $1$ to the left of 
$L\ell_1$  are frozen to ``$-$'' if they are below $L P_1$
and to ``$+$'' if they are above. The spins  at distance at most $1$ below $L\ell_2$ are frozen to ``$-$'' if they are  to the left of $L P_2$
and to ``$+$'' otherwise, 
see Figure \ref{fig:novapatata}.  In the quadrant
under consideration, this dynamics dominates the original one (for the
inclusion order of the set of ``$-$'' spins).
Let $F_L(\cdot,t)$ denote the function whose graph in $(\bf f_1, \bf
f_2)$ is the interface between ``$-$'' and ``$+$'' spins for this
dynamics. Using exactly the same argument as in
\eqref{eq:limithydro1} we get that
\begin{equation}
\label{eq:limithydro1c}
\lim_{L\to \8}\sup_{x\in[x(P_1),x(P_2)]}\sup_{t\leq T}\left|\frac{1}{L}F_L(xL,tL^2)-g(x,t)\right|=0
\end{equation}
in probability,
where $g$ is the solution for $t\geq 0$ and $x\in(x(P_1),x(P_2))$ of
\begin{equation}
\begin{cases}
\partial_t g(x,t)=\frac{1}{4}\partial^2_x g(x,t) 
\\
 \label{eq:edp6} g(\cdot,t)=f(\cdot,0)
\\
 g(x(P_1),t)=y(P_1) \textrm{	and	}  g(x(P_2),t)=y(P_2).
\end{cases}
\end{equation}

We are thus reduced to prove that for every $\tilde{x}_1,\tilde{x}_2$ satisfying 
$x(P_1)<\tilde{x}_1<\tilde{x}_2<x(P_2)$ and every $x\in (\tilde{x_1},\tilde{x_2})$
\begin{equation}\label{cczzccC}
g(x,\gep)<f(x,(1-\gd)\gep).
\end{equation}
 Lemma \ref{lem:contilaplacien} (which is valid also in this case, 
 since the curvature is Lipschitz and therefore $\partial^2_x f(\cdot,0)$ is uniformly continuous) allows us to write that for any fixed $\eta$, for $\gep$ small enough,
\begin{equation}\label{tytu}
g(x,\gep) \leq f(x,0)+\frac{\gep}{4}\left(\partial^2_x f(x,0)+\eta\right).
\end{equation}

We are left to estimate the right-hand side of \eqref{cczzccC}.
For any $\theta\in(0,\pi/2)$ and $s>0$ define $x(\theta,s)$ to be the 
${\bf f}_1$ coordinate, in the $(\bof_1,\bof_2)$ coordinate system, 
of the point of $\gamma(s)$ where the outward normal vector forms an anticlockwise  angle $\theta$ 
with the horizontal vector $\mathbf{e}_1$.  
Note that for $s\geq 0$
$x(\cdot,s)$ defines a bijective function.
We denote $\theta(\cdot,s)$ its inverse.

It is more practical for the purposes of this section to rewrite the
curve-shortening flow in the $(\bof_1,\bof_2)$ coordinate system.
Using the explicit expression \eqref{eq:a} of $a(\theta)$,
some trigonometry and the expression $|f''(x)|/(1+(f'(x))^2)^{3/2}$ 
for the absolute value of the curvature at the point $(x,f(x))$ 
of the curve given by the graph of a function $x\mapsto f(x)$, 
one gets that for  $\theta\in (0,\pi/2)$
\begin{equation}
a(\theta)k(\theta,s)=-\frac{1}{4}\partial^2_xf(x(\theta,s),s)\cos(\theta-\frac{\pi}{4})
\end{equation}
and
\begin{equation}
\label{eq:fondamental}
\partial_tf(x, s)=-\frac{a(\theta(x,s))k(\theta(x,s),s)}{\cos(\theta(x,s)-\frac{\pi}{4})}=\frac{1}{4}\partial^2_xf(x,s),
\end{equation}
so that 
\begin{equation}
f(x,(1-\gd)\gep)=f(x,0)+\int_0^{(1-\gd)\gep}\frac{1}{4}\partial^2_xf(x,s) \dd s.
\end{equation}
We need therefore to prove time-regularity of $\partial^2_xf(\cdot,s)$:
\begin{lemma} One has 
\label{lem:contidrift}
\begin{equation}
\sup\{| \partial_t f(x,s)-\partial_t f(x,0)|, s\in[0,t]  \textrm{ and
} x\in[x(P_1),x(P_2)]  \}\le \Psi(t,k_{max},k_{min},\mathbb L(k))
\end{equation}
where $\Psi$ tends to zero with the first argument.
\end{lemma}
\begin{proof}[Proof of Lemma \ref{lem:contidrift}]

Recall from Section \ref{sec:convexbu} that the curvature function
$k(\theta,s)$ is  jointly continuous in $(\theta,s)$ 
and that its modulus of continuity depends only on
$k_{max},k_{min},\mathbb L(k)$.
Thus using equation \eqref{eq:fondamental} it is sufficient to prove that $\theta(x,s)$ is a continuous function in $s$ uniformly in $x$:
\begin{equation}
\sup\{| \theta(x,s)-\theta(x,0)|, s\in[0,t]  \textrm{ and }
x\in[x(P_1),x(P_2)]  \}\le
\Psi_2(t,k_{max},k_{min},\mathbb L(k))
\end{equation}
where again $\Psi_2$ tends to zero as $t\to0$.
This comes from the continuity of $x(\theta,\cdot)$:
\begin{equation}
\sup\{| x(\theta,s)-x(\theta,0)|, s\in[0,t]  \textrm{ and } \theta
\in[0,\frac{\pi}{2}] \}\le \Psi_3(t,k_{max},k_{min},\mathbb L(k))
\end{equation}
and from the fact that $x(\cdot,s)$ is strictly monotone:
 for $t\geq 0$,
\begin{equation}
\label{eq:derivec}
\inf\{|\partial_\theta x(\theta,s)|,s\leq t, \theta \in [0,\frac{\pi}{2}] \}>c(k_{min})>0.
\end{equation}
Both properties are a consequence of 
\begin{equation}
x(\theta,t)=x(\pi/4,t)-\int^\theta_{\pi/4}\frac{\cos(\theta'-\pi/4)\dd \theta'}{k(\theta',t)}
\end{equation}
which is easily derived from \eqref{eq:3}-\eqref{eq:3bnis}.
\end{proof}

We finally get that for $x\in(x(P_1),x(P_2))$ and $\gep$ small enough
(as a function of $k_{min},k_{max},\mathbb L(k)$),
\begin{equation}
  f(x,(1-\delta)\gep)\ge f(x,0)+(1-\delta)\frac{\gep}{4} (\partial^2_x f(x,0)-\eta).
\end{equation}
Thus, combining this with \eqref{tytu}, \eqref{cczzccC} is proved if one has
\begin{equation}
\partial^2_xf(x,0)+\eta< (1-\delta)\left(\partial^2_xf(x,0)-\eta\right)
\end{equation}
i.e.
\begin{equation}
2\eta+{\delta}\partial^2_xf(x,0)\leq 0.
\end{equation}
For this it is sufficient to have $\eta$ small enough, since
(cf.\ \eqref{eq:fondamental}) $\sup\{\partial^2_xf(x,0),
x\in[x(P_1),x(P_2)]\}$ can be upper bounded by a negative constant times the minimal curvature
$k_{\rm min}$, which is strictly positive.
\qed

\subsubsection{Proof of \eqref{eq:Nc}}

Set $h(\cdot,t)$ to be the continuous concave
function whose graph in the $(\mathbf e_1,\mathbf e_2)$ coordinate
system is the portion of  $\gamma(t)$ which goes from $P_2(t)$ to $P_4(t)$
with the anti-clockwise orientation.
 Given a small $\eta$ choose $\xi$ small
enough so that $\sup\{|\partial_x h(x,0)|, u(P_1)-\xi \leq x \leq u(P_1)+\xi
\} \le \eta.$

Consider $\bar{h}(\cdot)$ the $C^1$ function equal to $h(\cdot,0)$ on
$[u(P_1)-2\xi,u(P_1)+2\xi]$ and affine outside.  
Assume for definiteness that $\bar h(u(P_1)-4\xi)\le \bar h(u(P_1)+4\xi)$.
Define $\xi^-=u(P_1)-4\xi$ and
$\xi^+=\inf\{x>u(P_1),\bar{h}(x)=\bar{h}(\xi^-)\}$.  We consider the
restriction of $\bar{h}$ to $[\xi^-,\xi^+]$ and still call it
$\bar{h}$.  Define
\begin{equation}
J^1:=[\xi^+,\infty)\times [\bar h(\xi^+),\infty),
\quad\quad J^2:=(-\8,\xi^-]\times[\bar h(\xi^+),\8).
\end{equation}
We consider the same chain of monotonicities as in the scale-invariant 
case (Section \ref{sec:45}) and we end up with a dynamics in the half-strip
$
[L\xi^-,L\xi^+]\times [L \bar h(\xi^+),\infty)$ 
with boundary spins frozen to ``$+$''
in $L(J^1\cup J^2)$ and to  ``$-$'' in $\bbZ^*
\times (-\infty,L \bar h(\xi^+)]$ and an initial condition with
``$-$'' spins under the graph of $L \bar h(\cdot/L)$. Also, the dynamics thus obtained does not allow moves that make
the interface non-connected. Calling $(\sigma_2(t))_{t\geq 0}$ this dynamics,
\eqref{eq:fields} is satisfied.

Define $H_L:[L\xi^-,L\xi^+]\to \bbZ$ to be the function whose graph in $(\be_1, \be_2)$ is the interface between ``$+$'' and ``$-$'' spins. We have to prove
\begin{equation}\label{cocomeroc}
\frac1L H_L(Lx,\gep L^2)\le  h(x,(1-\gd)\gep)\text{\;for every\;} x\in (u(P_1)-\xi ,u(P_1)+\xi ).
\end{equation}
Following the same steps as in \eqref{atrusk1} to \eqref{patatra} 
(recall that $\partial^2_x h(\cdot,0)$ is uniformly continuous by the
Lipschitz curvature assumption)
one
finds that the left-hand side of \eqref{cocomeroc} is upper bounded w.h.p.\ by
\begin{equation}\label{letrucpareildelautresectionc}
h(x,0)+\frac{\gep}{2}(1+\eta)^{-2}(\partial^2_x
h(u(P_1),0)+r(x,\mathbb L(k)))+o(\gep),
\end{equation}
where $r(x,\mathbb L(k))$ tends to $0$ when $x\to u(P_1)$.

To estimate the r.h.s of \eqref{cocomeroc}, one  remarks that,
in analogy with \eqref{eq:fondamental},
\begin{equation}
\partial_s h (x,s)=-a(\theta(x,s))k(\theta(x,s),s)/\sin(\theta(x,s))
\end{equation}
so that $\partial_ t h(x,t)$ is continuous in $x$ and $t$ (since $\theta$ is around $\pi/2$, $\sin(\theta(x,s))$ is bounded away from zero). Moreover
\begin{equation}
\partial_t h (u(P_1),0)=\frac{1}{2}\partial^2_x h(u(P_1),0),
\end{equation}
which can be obtained directly from $a(0)=1/2$ and from the fact that
the curvature of $\DD$ at the north pole $P_1$ equals minus the second
derivative of $h(x,0)$ computed at $x=u(P_1)$.  Thus for every $x\in (u(P_1)-\xi,u(P_1)+\xi)$
\begin{equation}\label{bo}
  h(x,(1-\delta)\gep)\geq h(x,0)+(1-\delta)\frac{\gep}{2}(1+\eta)\partial^2_x h(u(P_1),0),
\end{equation}
and \eqref{cocomeroc} is proven (combining \eqref{letrucpareildelautresectionc} and \eqref{bo}) choosing $\eta$ and $\xi$ small enough.

\subsection{Lower bound: Proof of \eqref{hicc}}

We are confident that the reader is by now convinced that the proof of
Theorem \ref{th:convex} is essentially identical to that in the
scale-invariant case, modulo the fact that the definitions of the
various subsets of $\bbR^2$ needed to define the regions where spins
are frozen to ``$-$'' or ``$+$'' ($U,J^1,J^2$, etc) have to be
adapted in the obvious way due to the lack of discrete-rotation
symmetry of the general initial droplet $\DD$.
We will therefore skip altogether the proof of \eqref{hicc} and we
limit ourselves to indicating the only point where some (minor) care
has to be taken.  

The definition \eqref{UAB} of the set $U$ is replaced by
$U:=\DD^{(-\nu)}$, cf.\ \eqref{eq:cdelta}.  Let $s_1$ be the vertical
segment obtained moving downwards from the ``north pole of $U$'' until
the point $c$ where $s_1$ meets $s_2$, the horizontal segment obtained
moving to the left from the ``east pole'' of $U$ until $c$ is reached.
To prove the analog of \eqref{eq:3steps2}, mimicking the proof given
in Section \ref{sec:po2}, one would like to apply \eqref{topc} in
order to freeze to ``$-$'' all the spins along the two rescaled
segments $Ls_1,Ls_2$. This is however not allowed in general, because
nothing guarantees that they are entirely contained in $LU$, i.e.,
that $c\in U$ (this problem does not occur for the invariant shape
$\DDD$, where $c$ is the origin).  The solution however is simple
(cf.\ Figure \ref{fig:Dessin13}): one
just freezes to ``$-$'' all the spins along the portions of $L s_1,L
s_2$ which are inside $L U$, and along the shorter portion of
$L\partial U$ which connects them (call $\Gamma$ this portion).  The
point is that in this situation the $+/-$ interface between north and
east poles follows again the corner dynamics and Theorem \ref{mickey}
is applicable. The freezing of ``$-$'' spins along $\Gamma$ is
equivalent to putting a hard-wall constraint in the corner dynamics
(the interface is not allowed to cross a zig-zag path which
approximates $\Gamma$) but this is irrelevant: since $\Gamma$ is at
distance of order $L$ away from the linear profile the corner dynamics
approaches for long times, the probability that the interface even
feels the hard-wall constraint within the diffusive times of order
$L^2$ we are interested in goes to zero with $L$ (this again can be
seen via Theorem \ref{mickey}).  Other than that, the proof of
\eqref{hicc} is identical to that in the $\DD=\DDD$ case.

\begin{figure}[hlt]

\leavevmode
\epsfxsize =8 cm
\psfragscanon

\psfrag{ls1}[c]{$Ls_1$}
\psfrag{ls2}[c]{$Ls_2$}
\psfrag{c}[c]{$c$}
\psfrag{U}[c]{$LU=L\DD^{(-\nu)}$}
\psfrag{ls2}[c]{$Ls_2$}
\psfrag{gamma}{$\Gamma$}
\psfrag{D}[c]{$L\DD$}

\epsfbox{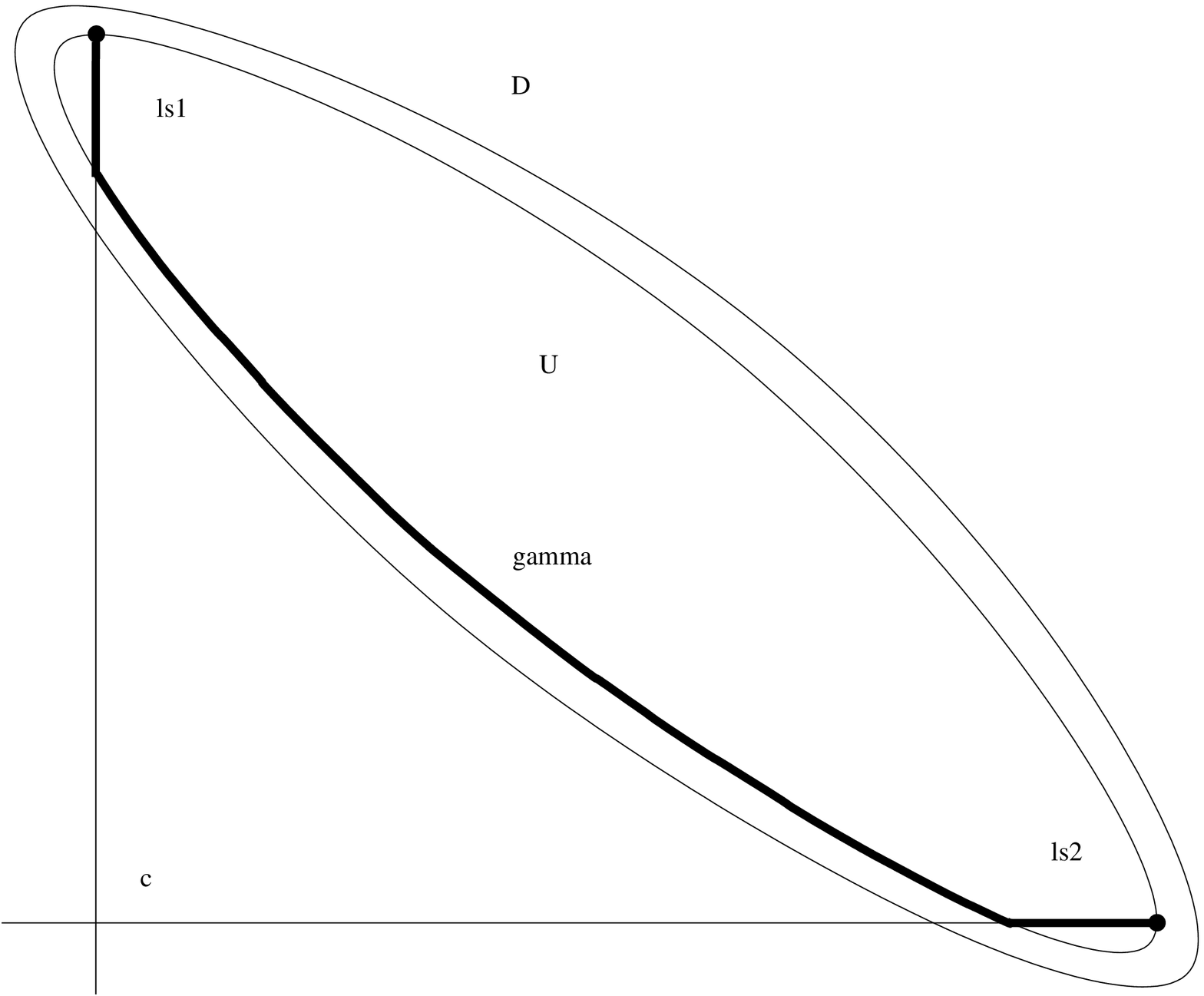}
\begin{center}
 \caption{
 }
\label{fig:Dessin13}
\end{center}
\end{figure}

\section{Proof of Theorem \ref{mickey}: scaling limit for SSEP}
\label{sec:potm}

The first step is to discretize \eqref{eq:edpcont}, so that instead of working with $\phi(\cdot,\cdot)$ we get 
$\Phi(\cdot,\cdot)$ solution of the analogous discrete Cauchy problem:
\begin{equation}
\begin{cases}
\partial_t \Phi( x, t)=\frac12\Delta \Phi (x, t) \\
\Phi^L( 0, t)=h^0_0=0\\  \Phi^L(L, t)=h^0_L \\
\Phi^L(x,0)=h^0_x 
\end{cases}
\end{equation}
for every $t\ge0$ and $x\in \{1,\dots, L-1\}$.
Here $\gD$ is the discrete Laplacian operator:
\begin{equation}
( \gD f)(x):= f(x+1)+f(x-1)-2f(x) \quad \forall x\in \{1,\dots,L-1\}.
\end{equation}
Note that $\Phi(x,t)=\bbE\left[h_x(t)\right]$, with $(h(t))_{t\ge0}$
the process with generator \eqref{defL1}, and that
$\Phi(0,t)-h_0(t)=\Phi(L,t)- h_L(t)=0$. It is a standard result that
$\Phi$, solution of the discrete space heat-equation, converges to
$\phi$ in all reasonable norms when $L \to \infty$ in the diffusive
limit. We record this result here:
 \begin{lemma}
 \label{lem:comparaison}
  \begin{equation}
 \lim_{L\to \infty} \max_{t\in [0,T]} \max_{x\in [0,1]}\frac{1}{L} | \Phi(\lfloor xL \rfloor,tL^2)  -L\phi(x,t)| =0.
 \end{equation}
\end{lemma}
Using Lemma \ref{lem:comparaison}, we are reduced to prove
\begin{equation}\label{convlinf}
\lim_{L\to \infty} \bbP\left[ \max_{t\in [0,TL^2]} \max_{x\in \{1,\dots,L-1\}} | h_{ x} (t)  -\Phi(x,t)|<\gep L \right]=1.
\end{equation}
Both $h_{\cdot}(t)$ and $\Phi(\cdot,t)$ are $1$-Lipschitz functions
(for all $t$) so that $|h_{\cdot}(t)-\Phi(\cdot,t)|$ is $2$-Lipschitz
and 
\begin{equation}
\left\{\max_{x\in \{1,\dots,L-1\}} | h_{ x} (t)  -\Phi(x,t)|\geq a \right\}\Longrightarrow
\left\{ \sum_{x=1}^{L-1} \left[h_x(t)-\Phi(x,t)\right]^2\geq a^3/3\right\}.
\end{equation}
As a consequence, \eqref{convlinf} is equivalent to prove the following $\bbL_2$ convergence statement:
\begin{proposition}
\label{prop:convLdeux}
The following convergence in probability holds:
\begin{equation}
\label{eq:convLdeux}
\lim_{L\to \infty}  \sup_{t\in[0,L^2T]} \frac{1}{L^3}\sum_{x=1}^{L-1} \left[h_x(t)-\Phi(x,t)\right]^2=0.
\end{equation}
\end{proposition}

\begin{proof}[Proof of Proposition \ref{prop:convLdeux}]
The restriction of the operator $\gD$ to \[\gL_L
=\{ g:\{0,\dots,L\}\mapsto \bbR, g(0)=g(L)=0\}
\]
  is self-adjoint (for the canonical scalar product on $\bbR^{L-1}$ denoted in the sequel by 
$\langle \cdot, \cdot \rangle$) 
and the family of functions 
\begin{equation} \begin{split}
f_k:\{0,\dots,L\}\ni x \mapsto \sqrt{\frac2L}\sin\left(\frac{k\pi x}{L}\right),\quad k=1,\dots,L-1
\end{split}\end{equation}
forms an orthonormal basis of $\gL_L$ of $\gD$-eigenfunctions, with respective eigenvalues 
\begin{equation}
-\gl_k:= 2\cos\left(\frac{\pi k}{L}\right)-2<0.
\end{equation}
As the function $x\mapsto h_x(t)-\Phi(x,t)$ is in $\gL_L$ it can be decomposed on this basis.
 We use the notation $H_t^k$ for its $k-$th coordinate (multiplied by $\sqrt{L/2}$ for convenience):
\begin{equation}
H_t^k:=\sum_{x=0}^L \left[h_x(t)-\Phi(x,t)\right]\sin\left(\frac{k\pi x}{L}\right).
\end{equation}
The quantity one wants to estimate in \eqref{eq:convLdeux} is equal to
\begin{equation}
\label{eq:inter}
\sup_{t\in[0,L^2 T]}  \frac{2}{L^4}\sum_{k=1}^{L-1} (H_t^k)^2
\le \frac{2}{L^4}\sum_{k=1}^{L-1} \sup_{t\in[0,L^2 T]} (H_t^k)^2.
\end{equation}
 We control the right-hand side\ by controlling each $H_t^k$ separately.
\begin{lemma}
\label{lem:controlH}
For every $L$, $k\in\{1,\dots,L-1\}$ and $t>0$ one has deterministically
\begin{equation}
\label{eq:controlH1}
|H_t^k|\le 4 L^2/k.
\end{equation}
Moreover for any given $T$, w.h.p.\
\begin{equation}
\label{eq:controlH2}
 |H_t^k|\le L^{7/4}\quad \mbox{for every}\quad k\le (\log L)^{1/3}\quad\mbox{and
every}\quad t \le L^2 T.
\end{equation}
\end{lemma}

\begin{proof}[Proof of Lemma \ref{lem:controlH}]
The first point is easy. Using summation by parts
\begin{eqnarray}
 H_t^k=
 \sum_{x=1}^{L}\left(\left[h_x(t)-\Phi(x,t)\right]-\left[h_{x-1}(t)-\Phi(x-1,t)\right]\right)\sum_{y=x}^L \sin\left(\frac{k\pi y}{L}\right).
\end{eqnarray}
Then one can check that for every $x$ and $k$
\begin{gather}
\left|\left[h_x(t)-\Phi(x,t)\right]-\left[h_{x-1}(t)-\Phi(x-1,t)\right]
\right|\le 2,\\
\left|\sum_{y=x}^L \sin\left(\frac{k\pi y}{L}\right)
\right|
\le 
\frac{2L}k 
\end{gather}
so that \eqref{eq:controlH1} follows.

For the second point, first, one notices that for all $k\in \{1,\dots,L-1\}$, the functions
\begin{equation}\label{eq:fpro}
\begin{split}
F_k: \gO_{M_L,N_L}\ni
h\mapsto \sum_{x=0}^L \sin\left(\frac{\pi k x}{L}\right)  \left[h_x-\frac{h_L-h_0}{L}x\right]
\end{split}
\end{equation}
are eigenfunctions of $\mathcal L$ with respective eigenvalues $-\gl_k/2$.
Indeed $F_k$ is just a linear combination of the coordinate function $A_x: h\mapsto h_x$ (plus a constant), and it can be seen from the 
very definition \eqref{defL1} of the generator $\mathcal L$, that
\begin{equation}
 \mathcal L (A_x)(h)=\frac12(\gD h)(x)=\frac12(\gD \tilde h)(x)
\end{equation}
using the notation $\tilde h_x= h_x-\frac{h_L-h_0}{L}x$.  Hence
(note that $\tilde h \in \gL_L$)
\begin{multline}
2 \mathcal L F_k(h)= 
\sqrt{L/2}\langle f_k, \gD \tilde h\rangle=\sqrt{L/2}\langle \gD f_k, \tilde h\rangle= 
-\gl_k\sqrt{L/2}\langle f_k, \tilde h\rangle= -\gl_k F_k(h).
\end{multline}
As a consequence one can rewrite 
\begin{equation}
 H_t^k=\sum_{x=0}^L \sin\left(\frac{k\pi x}{L}\right)\tilde h_x(t)-e^{-\gl_kt/2} 
\sum_{x=0}^L \sin\left(\frac{k\pi x}{L}\right)\tilde h_x(0)
\end{equation}
and notice that
$ M_t^k:= e^{\gl_kt/2}H_t^k$
is a martingale. Therefore one can get the result by computing the second moment of
$ M_t^k$ and using Doob's inequality.

It is not  difficult to bound the quadratic variation of $M^k$. Notice that
\begin{equation}
 \bbE\left[ (M_t^k)^2\right]=\bbE \left[\int_{0}^t \dd \langle M^k\rangle_s\right]
\end{equation}
and that
\begin{equation}
 \dd \langle M^k\rangle_s=e^{\gl_k s}\dd \langle H^k\rangle_s
=e^{\gl_k s}\sum_{k=1}^{L-1} \sin^2\left(\frac{k\pi x}{L}\right)\frac{(\gD(h(t))(x))^2}4 \dd s\le Le^{\gl_k s}\dd s
\end{equation}
so that 
$  \bbE\left[ (M_t^k)^2\right]\le L\int_{0}^t e^{\gl_ks}\dd s.$
Therefore (using $\lambda_k=\pi^2 k^2/L^2(1+o(1))$ uniformly for all $k\leq  (\log L)^{1/3}$),
\begin{equation}
\bbP\left[\sup_{t\in[0,L^2 T]} |H_t^k|\ge  a \right]\le  
\bbP\left[\sup_{t\in[0,L^2 T]} |M_t^k|\ge  a \right]\le C \frac{L^3e^{\gl_k L^2 T}}{a^2k^2}.
\end{equation}
Using this inequality for $a:=L^{7/4}$ and all $k \leq (\log L)^{1/3}$ one gets that
\begin{equation}
\bbP\left[\exists t\in[0,L^2T], \ \exists k\leq (\log L)^{1/3}, \  |H_t^k|\ge L^{7/4}\right]\le
\sum_{k\leq (\log L)^{1/3}}\frac{C}{k^2\sqrt L}e^{k^2\pi^2 T}. 
\end{equation}
One can check that the right-hand side\ above tends to zero when $L$ goes to infinity, which finishes the proof of Lemma \ref{lem:controlH}.
\end{proof}

We now turn  to \eqref{eq:inter}:
 \begin{equation}
\frac{2}{L^4}\sum_{k=1}^{L-1} \sup_{t\in[0,L^2 T]} (H_t^k)^2  \le \frac{2}{L^4}\sum_{k\leq (\log L)^{1/3}} \sup_{t\in[0,L^2 T]} (H_t^k)^2+ 32\sum_{k=\lceil(\log L)^{1/3}\rceil}^L k^{-2}.
\end{equation}
The second term tends to zero (it is roughly $(\log L)^{-1/3}$). The
first one is w.h.p. less than
\begin{equation}
\frac{2}{L^4}\sum_{k\leq (\log L)^{1/3}} L^{7/2}\leq \frac{\log L}{\sqrt{L}}.
\end{equation}
This achieves the proof of Proposition \ref{prop:convLdeux} and thus also the one of Theorem \ref{mickey}.
\end{proof}


\appendix
\renewcommand{\thesubsection}{\Alph{section}.\arabic{subsection}}

\section{Proof of Theorem \ref{spo}: scaling limit for the 
zero-range process}
\label{sec:pots}

This section follows quite closely computations in Appendix A of \cite{cf:Spohn}.
\subsection{Particle system and monotonicity}
For $x\in\{- L,\ldots, L\}$ we denote $\eta_x:=h_{x+1}-h_x$ the
discrete gradient of $h$ in $x$.  A configuration $h\in \gO_L$ can be
alternatively given by $\eta\in\Theta_{L}:=\{\eta:\{- L,\ldots, L\}\to
\bbZ\}$.  It turns out that the {\emph zero-range process} description
of the dynamics (cf.\ Section \ref{sub:dap}) is easier to work with.

For a more formal description of the dynamics we write explicitly its generator.
For $\eta\in\Theta_{L}$ and $x\in \{- L,\ldots, L-1\}$, we define the configuration $\stackrel{\rightarrow}{\eta}^{(x)}$
as 
\begin{equation}\begin{split}
{\stackrel{\rightarrow}{\eta}}^{(x)}(x)&:=\eta_{x}-\sign(\eta_x),\\
{\stackrel{\rightarrow}{\eta}}^{(x)}(x+1)&:=\eta_{x+1}+\sign(\eta_x),\\
{\stackrel{\rightarrow}{\eta}}^{(x)}(y)&:=\eta_y, \quad \forall y\notin \{x,x+1\}.
\end{split}\end{equation}
We define $\stackrel{\leftarrow}{\eta}^{(x)}$ analogously for $x \in \{-L+1,\ldots, L\}$ replacing $x+1$ in the second and third lines by $x-1$.
The sign function $\sign$ is given by
\begin{equation}
\sign(a):=
\begin{cases} 
1 \text{ if } a >0,\\
 -1 \text{ if } a <0,\\
0 \text{ if } a=0.
\end{cases}
\end{equation}

The generator of the chain seen in the state-space  $\Theta_{L}$ is given by
\begin{equation}\label{dnaspo2}
\mathcal L f:= \frac12\sum_{x=-L}^{ L-1} \left[f ({\stackrel{\rightarrow}{\eta}}^{(x)})+ f ({\stackrel{\leftarrow}{\eta}}^{(x+1)})-2 f(\eta)\right].
\end{equation}
Note that the dynamics conserves the sum of the $\eta$'s, i.e.\ the value
of $h_{L+1}$.

Before going to the core of the proof, we need to change slightly the
initial condition. In order to compare with the original one, one needs
the following monotonicity statement:
\begin{proposition}\label{fuzz}[Coupling]
\begin{itemize}
\item[(i)] There is a canonical way of constructing simultaneously the
  dynamics with generator \eqref{dnaspo} from all possible initial
  configurations $h^0$.  It satisfies the following monotonicity
  property: given $h^0$ and $\bar h^0$ with $h^0_x\ge \bar h^0_x$ for
  all $x$, the dynamics $h$ and $\bar h$ starting from $h^0$ and $\bar
  h^0$ respectively satisfy
$ \quad h_x(t)\ge\bar h_x(t) $ for every $t$ and $x$.
Moreover, the dynamics started from $h^0+a$,
$a\in \bbZ$, (a vertically translated version of $h^0$, including the boundary
conditions $h_0$ and $h_{L+1}$), is simply $(h(t)+a)_{t\ge 0}$.
\item[(ii)] 
There is a canonical  way of constructing  the dynamics with generator \eqref{dnaspo2}   from all possible initial configurations 
$\eta^0$. 
It satisfies the following monotonicity property: given $\eta^0$ and $\bar \eta^0$ with $\eta^0_x\ge \bar \eta^0_x$ for all $x$, the dynamics $\eta$ and $\bar \eta $ starting 
from $\eta^0$ and $\bar \eta^0$ respectively satisfy
 $  \eta_x(t)\ge\bar \eta_x(t)$ for every $t$ and $x$.

\end{itemize}
\end{proposition}

\begin{proof}
The idea of the proof is using a canonical construction of the process, similarly to what is done in Section \ref{sub:gcm}.
It is quite classic but we perfom it here for the sake of completeness.
\begin{itemize}
\item For $x\in \{-L+1,L\}$ we define $(\tau_{n,x})_{n\ge0}$ and
  $(\tau'_{n,x})_{n\ge0}$ to be two IID clock processes, with
  $\tau_{0,x}=0$ and $\tau_{n+1,x}-\tau_{n,x}$ IID 
  exponential variables of mean $2$.
 
\item The process $h(\cdot)$ is  c\`adl\`ag and constant in time except at the of the
  ringing times of the clock processes.  At time $\tau_{n,x}$ only
  $h_x$ is modified, as follows:
  $h_x(\tau_{n,x})=h_x(\tau_{n,x}^-)+\sign(h_{x-1}(\tau_{n,x}^-)-h_x(\tau_{n,x}^-))$, the other
  coordinates being left unchanged.  At time $\tau'_{n,x}$ only $h_x$
  is modified, as follows:
  $h_x(\tau'_{n,x})=h_x((\tau'_{n,x})^-)+\sign(h_{x+1}((\tau'_{n,x})^-)-h_x((\tau'_{n,x})^-))$, the other
  coordinates being left unchanged.
\end{itemize}
The reader can check that this allows to couple the dynamics from all possible initial conditions and that our coupling has the desired properties.
This coupling induces a coupling on $\eta$ that also has the right properties.

\end{proof}

\subsection{Changing the initial condition}

We prove \eqref{atrusk} working with an initial condition which is not
the one, $h^{0}$, described in \eqref{defho}, which is random and for
which the number of particle at a site is given by a geometric
variable.  The reason for this change of initial condition will appear
in the proof of $(iii)$ in Lemma \ref{reloumec}.  We explain in this
section why this implies the result starting from $h^{0}$.

Given a continuous function $\phi^0: [-1,1]\to \bbR$ with
$\phi^0(\pm1)=0$ and with a finite number of changes of monotonicity,
set $(\hat\eta_{x})_{x\in \{- L,\ldots, L\}}$ to be a family of
independent variables with the following distribution: if
$\phi^0((x+1)/L)-\phi^0(x/L)\ge 0$ then $\hat\eta_x$ is a geometric
variable of mean $L(\phi^0((x+1)/L)-\phi^0(x/L))$ and if
$(\phi^0((x+1)/L)-\phi^0(x/L))<0$ then $-\hat\eta_x$ is a geometric
variable of mean $L(\phi^0(x/L)-\phi^0((x+1)/L))$ (with the convention
that $\phi^0(1+1/L)=0$).  One sets
\begin{equation}\label{geomcond}
\hat h_x^0=\sum_{y=-L}^{x-1} \hat\eta_y.
\end{equation}
Note that for every $\gep>0$, w.h.p, 
\begin{equation}\label{youth}
 \hat h^0_x-L^{1/2+\gep}\le  h^0_x\le \hat h^0_x+L^{1/2+\gep}
\quad  \text{for every\;\;
}x \in \{-L,\ldots, L+1\}.
\end{equation}
 Let
$(h(t))_{t\ge0}$, $(\hat h(t))_{t\ge0}$ be the dynamics with generator
\eqref{dnaspo} started with initial condition $h^{0}$, $\hat h^0$
respectively, constructed using the canonical way of Proposition
\ref{fuzz} (i). Then with high probability, for every $t>0$ and 
$ x \in \{-L,\ldots, L\}$
\begin{equation}
 \hat h_x(t)-L^{1/2+\gep} \le  h_x(t)\le \hat h_x(t)+L^{1/2+\gep}.
\end{equation}
Therefore in order to prove \eqref{atrusk} for $h(\cdot)$, it is
sufficient to prove it for $\hat h (\cdot)$. We let $\hat\eta_x(t)=\hat h_{x+1}(t)-\hat h_x(t)$ 
denote the gradient of $\hat h$.

\subsection{Proof of an $\bbL_2$ statement}

For $(\hat h(t))_{t\ge0}$ defined above one has
\begin{proposition}
\label{prop:propspohn}
For any $t\ge 0$
\begin{equation}
\lim_{L\to \infty}\bbE \left[\frac{1}{L^3}\sum_{x=- L}^{ L+1} (\Phi(x,L^2 t)- \hat h_x(L^2 t))^2\right]=0
\end{equation}
\end{proposition}
This result does not directly imply \eqref{atrusk} ($\hat h$ may have a priori
unbounded gradients), but it is not to difficult conclude from 
Proposition \ref{prop:propspohn},
see Section \ref{sec:concluding}.
In the rest of the section, for lightness of notation we write
$h,\eta$ instead of $\hat h,\hat \eta$.

Before starting the proof we need some technical statements. First note, recalling the definition of the generator \eqref{dnaspo}, that for every $x\in\{-L+1,\dots,  L\}$
\begin{equation}\begin{split}\label{trucdement}
2\partial_t \bbE \left[h_x(t)\right]&=\bbE\left[\sign(\eta_x(t))-\sign(\eta_{x-1}(t))\right],\\
2\partial_t \bbE \left[h^2_x(t)\right]&=\bbE\left[2  h_x(t)(\sign(\eta_x(t))-\sign(\eta_{x-1}(t)))+(|\sign(\eta_x(t))|+|\sign(\eta_{x-1}(t))|)\right].
\end{split}\end{equation}
Now some remarks:
\begin{lemma}\label{reloumec}
The following properties hold (recall notations in \eqref{qs}):
\begin{itemize}
\item[(i)] $\max_{x} |q_x(t)|$ is a non-increasing function of $t$.
As a consequence 
\begin{equation}
\forall t>0,\  \forall x\in\{- L,\ldots, L\}, \ |q_x(t)|\le \|\partial_x\phi^0\|_\infty.
\end{equation}
\item[(ii)] $\max_{x} |\sigma (q_{x+1}(t))-\sigma(q_{x}(t))|$ is a
  non-increasing function of $t$ (recall $\sigma(u)=u/(1+|u|)$).
Then, using also (i), for some
  $C(\phi^0)=C(\|\partial_x\phi^0\|_\infty,\|\partial^2_x\phi^0\|_\infty)<\infty$
  one has
\begin{equation}
\label{sigmaq}
\forall t>0,\  \forall x\in\{- L,\ldots, L\}, \ |q_{x+1}(t)-q_{x}(t)|\le C(\phi^0)/L.
\end{equation}
\item[(iii)] For any $t$,  the random vectors $(\eta_x(t))_{x\in \{- L,{\ldots,} L\}}$ and  $(-\eta_x(t))_{x\in \{- L,{\ldots,} L\}}$ are stochastically dominated
by $2 L+1$ IID geometric variables with mean $\|\partial_x\phi^0\|_\infty$.
\end{itemize}
\end{lemma}

\begin{proof}
  For $(i)$ it is sufficient to show that $Q(t)=\max_x q_x(t)$ is
  non-increasing (by a similar argument one shows that $ \min q_x(t)$ is
  non-decreasing).  As the maximum over finitely many differentiable
  functions, $\max_x q_x(t)$ possesses a right and a left-derivative
  everywhere and the right-derivative is equal to
\begin{equation}
\partial^+_t Q(t)=\max_{x\in \argmax q_{\cdot}(t)} \partial_t q_x(t).
\end{equation}
For any $x$ in $\max_{x\in \argmax q_{\cdot}(t)}$, one has
\begin{equation}
2\partial_t q_x(t)=\sigma( q_{x+1}(t))+\sigma( q_{x-1}(t))-2\sigma( q_{x}(t))\le 0,
\end{equation}
(as $\sigma( q_{x}(t))$ is maximal), and therefore $Q(t)$ is decreasing.

For $(ii)$: Using the same argument as for the point $(i)$, we have to note that for any fixed time $T$ and $x_0$ where $\max_x [\sigma(q_{x+1})-\sigma(q_x)](T)$ is attained one has
\begin{multline}
2\left[\partial_t\{\sigma(q_{x_0+1})-\sigma(q_{x_0})\}\right](T)
=\sigma' (q_{x_0+1}(T))\left[\sigma(q_{x_0+2}(T))-\sigma(q_{x_0+1}(T))\right]
\\+\sigma' (q_{x_0}(T))\left[\sigma(q_{x_0}(T))-\sigma(q_{x_0-1}(T))\right]
\\-\left[\sigma' (q_{x_0+1}(T))+\sigma' (q_{x_0}(T))\right]\left[\sigma(q_{x_0+1}(T))-\sigma(q_{x_0}(T))\right]\le 0.
\end{multline}
Therefore, one has that 
\[
|\sigma(q_{x+1}(t))-\sigma(q_{x}(t))|\le \frac{C(\phi^0)}L.
\]
In order to deduce \eqref{sigmaq}, write
\[
\sigma(q_{x+1}(t))-\sigma(q_{x}(t))=\sigma'(y)\left[
q_{x+1}(t)-q_{x}(t)\right]
\]
for some $q_{x+1}(t)\le y\le q_{x}(t)$. Since the $q_x$ are bounded
(point $(i)$) and $\sigma(\cdot)$ has uniformly positive derivative on
bounded intervals, \eqref{sigmaq} follows.

For $(iii)$: One has that
$L\left(\phi^0(\frac{x+1}{L})-\phi^0(\frac{x}{L})\right)\le
\| \partial_x \phi^0 \|_\infty$ so that the initial configuration $\eta^0$ is
stochastically dominated by $\tilde \eta^0$ the configuration given by
$2 L+1$ IID geometric variable{s} with mean
$\|\partial_x\phi^0\|_\infty$.  According to Proposition \ref{fuzz}
(ii), one can couple the two dynamics $\eta$ and $\tilde \eta$
starting from $\eta^0$ and $\tilde \eta^0$ so that $\eta(t)\le
\tilde\eta(t)$ for all $t\ge 0$. For fixed $t$ the law of $\tilde
\eta(t)$ is the same as the one of $\tilde \eta^0$ as this
distribution is stationary for the dynamics.  The other domination is
proved in the same way.
\end{proof}

\begin{proof}[Proof of Proposition \ref{prop:propspohn}]
 We estimate the difference between $\bbE \left[\frac{1}{L^3}\sum_{x=- L}^{ L+1} (\Phi(x,L^2 t)- h_x(L^2 t))^2  \right]$ and the same quantity at time zero,
by considering it as the integral of its time-derivative.

\begin{multline}\label{develo}
\bbE \left[\frac{1}{L^3}\sum_{x=- L}^{ L+1} (\Phi(x,L^2 t)- h_x(L^2 t))^2  \right]-
\bbE \left[\frac{1}{L^3}\sum_{x=- L}^{  L+1} (\Phi(x,0)- h_x(0))^2  \right]\\
=\frac{1}{L^3}\int_0^{L^2t}  \sum_{x=- L+1}^{ L}  \partial_s\bbE\left[ (\Phi(x,s)- h_x(s))^2\right]\dd s\\
=\frac {1} {L^3} \sum_{x=- L+1}^{ L}  \int_0^{L^2t}\bbE\left\{ (\Phi(x,s)-h_x(s))\left(\sigma(q_x(s))-\sigma(q_{x-1}(s))\right)\right.\\
-\Phi(x,s)(\sign(\eta_x(s))-\sign(\eta_{x-1}(s)))\\
\left.+h_x(s)(\sign(\eta_x(s))-\sign(\eta_{x-1}(s)))+\frac12\left(|\sign(\eta_x(s))|+|\sign(\eta_{x-1}(s)|\right)\right\}\dd s\\
=\frac {1} {L^3}  \sum_{x=- L}^{ L} \int_0^{L^2t} \bbE\Big[-q_x(s)\sigma(q_x(s))+\eta_x(s)\sigma(q_x(s))+q_x(s)\sign(\eta_x(s))\\
-(|\eta_x(s)|-|\sign(\eta_x(s)|)\Big]\dd s\\\
-  \frac  1{L^3}  \int_0^{L^2t} \bbE\left[h_{ L+1}(s)(\sign(\eta_{ L}(s))-\sigma(q_{ L}(s)))+\frac{1}{2}(|\sign(\eta_{- L}(s)|+ |\sign(\eta_{ L}(s)|)\right]\dd s.
\end{multline}
The second equality is obtained by expanding the product and  using \eqref{trucdement} and \eqref{eq:edpdiscret} 
to estimate all the derivated terms.
  The third equality is obtained via summation by
parts, it gives a term that is due to boundary effect (the second one) which can be bounded as follows
\begin{multline}
L^{-3}\left| \int_0^{L^2t} \bbE \left[h_{ L+1}(s)(\sign(\eta_{ L}(s))-\sigma(q_{ L}(s)))+\frac{1}{2}(|\sign(\eta_{- L}(s))|+ |\sign(\eta_{ L}(s))|)\right]\dd s\right|
\\
\le C L^{-1}(1+\bbE  |h_{ L+1}|)=O(L^{-1/2}).
\end{multline}
Indeed $h_{ L+1}(t)$ is constant through time and is the sum of $2L+1$ independent variables.
 The mean of this sum is $0$ and the variance of each term is bounded as we supposed $\phi^0$ to be smooth. The variance of $h_{ L+1}$ is thus $O(L)$.
We can also neglect the second term in the first line as
\begin{equation}
\bbE \left[\frac{1}{L^3}\sum_{x=- L}^{  L+1} (\Phi(x,0)- h_x(0))^2 \right]=\frac{1}{L^3}\sum_{x=- L}^{  L+1} \var( h_x(0)) =O(L^{-1})
\end{equation}
where the last equality is easy to obtain once noticed that $h_x$ is the sum of $(L+x)$ independent geometric variables with bounded variance.


\medskip
Set 
\begin{equation}
\begin{split}
A(x,s)&:=-q_x(s)\sigma(q_x(s))+\eta_x(s)\sigma(q_x(s))+q_x(s)\sign(\eta_x(s))-(|\eta_x(s)|-|\sign(\eta_x(s)|).
\end{split}
\end{equation}
From the previous equations one gets that 
\begin{equation}\label{reza}
\bbE \left[\frac{1}{L^3}\sum_{x=- L+1}^{ L} (\Phi(x,L^2 t)- h_x(L^2 t))^2  \right]=\frac 1 {L^3}  \int_0^{L^2t}\sum_{x=- L+1}^{ L}  \bbE[ A(s,x)]\dd s+o(1).
\end{equation}

To understand better the rest of the proof, the reader should notice
that if  $(\eta_x(s))_{x\in\{- L,\ldots, L\}}$ were distributed
like geometric variables it would be possible to factorize
$\bbE\left[A(x,s)\right]$ in a product of negative sign and from
equation \eqref{reza} the proof would be over.  Indeed, for $q>0$ and
$\eta$ distributed like a geometric variable of mean $u>0$ (or  
$-\eta$ is distributed like a geometric variable of mean $-u>0$),
\begin{equation}
\bbE\left[-q\sigma(q)+\eta\sigma(q)+q\sign(\eta)-(|\eta|-|\sign(\eta)|)\right]=-(q-u)(\sigma(q)-\sigma(u))\le 0,
\end{equation}
(recall that $\sigma(\cdot)$ is an increasing function).  It is not
true in general that $\eta_x(s)$ are geometrically distributed for
$s>0$ but this is reasonable to think that their distribution is close
to geometric: as the system mixes locally in finite time, what one
should observe on finite but large windows is close to an equilibrium
measure, and from \cite{cf:Andjel} it is known that the only
(infinite-volume translation invariant) equilibrium measures for the
zero-range process are convex combinations of products of geometric
variables.  Most of our efforts will therefore be focused on proving
convergence to the infinite volume measure for a space-time averaged
version of the probability distribution of the $\eta_x(s)$ (using this
space-time average is somehow crucial for the proof to work).

As the limiting object is an infinite volume measure, it is somehow
more convenient to consider $\eta(s)$ as an element of $\bbZ^\bbZ$ by
periodizing it: for the system of size $(2L+1)$ 
 one sets
$ \eta_{x+k(2L+1)}=\eta_x$ for every
$ k \in \bbZ, x\in \{-L,\dots,L\}.$
For $y\in\bbZ$ one defines $\theta_y$ to be the shift operator $\eta\mapsto \theta_x\eta$ defined by
\begin{equation}\label{periodiz}
\forall x\in\bbZ, \quad ( \theta_y \eta)_x:=\eta_{x+y}.
\end{equation}

 We define for each $L>0$ the measure $\mu_t^L$ on $\bbZ^\bbZ$ our space-time averaged measures by its action on  
local functions
(for $K\in \bbN$ we call $f(\eta)$ a $K$-local function if $f$ is bounded and 
can be written as a function of $\eta_{|[-K,K]}$;
$f$ is a local function if there exists a $K$ such that it is $K$-local):
\begin{equation}
\label{eq:defmu}
\mu^L_t(f):=\bbE\left[\frac 1 {tL^2} \frac{1}{2L+1}\int_0^{L^2t} \sum^{ L}_{y=- L}f(\theta_y (\eta(s)))\dd s\right],
\end{equation}
We want to prove that any limit point (when $L\to\infty$) of $\mu^L_t$ is an equilibrium and use this information to bound 
the right-hand side of \eqref{reza}.

We introduce some notation to describe the limiting measure.
For $u\in \bbR$ define $\rho^u$ to be a measure on $\eta=(\eta_x)_{x\in \bbZ}$ such that
 the $\eta_x$ are IID geometric variables of mean $u$ if $u\ge0$ while
the $-\eta_x$ are IID geometric variables of mean $-u$ if $u<0$.
If $\nu$ is a probability measure on $\bbR$ define
\begin{equation}
\rho^\nu:= \int \rho^u \nu(\dd u).
\end{equation}

\begin{proposition}\label{wconv}
  Fix $t>0$.  For any subsequence of $(\mu^{L_n}_t)_{n\ge 0}$, it is
  possible to find a sub-subsequence $(\mu^{L'_n}_t)_{n\ge 0}$ that converges locally to $\rho^\nu$ with 
  $\nu$ a probability measure on $\bbR$ with support included in
  $[-\| \partial_x \phi^0\|_{\infty},\| \partial_x \phi^0\|_{\infty}]$,
   in the sense that for any
  local function $f$
\begin{equation}\label{weakconv}
\lim_{n\to\infty} \mu^{L'_n}_t(f)=\rho^{\nu}(f).
\end{equation}
As a consequence for any local function $f$
\begin{equation}
\limsup_{L\to\infty} \mu^{L}_t(f)\le \max_{u\in [-\| \partial_x \phi^0\|_{\infty},\| \partial_x \phi^0\|_{\infty}]}  \rho^u (f).
\end{equation}
\end{proposition}

\begin{rem}\rm
Note that the convergence does not hold in the topology induced by the total variation distance: 
indeed $\mu^L_t$ give mass one to $L$-periodic $\eta$ whereas these configurations have mass zero
for the limiting measure.
\end{rem}

\begin{proof}[Proof of Proposition \ref{wconv}]
  For any fixed $K>0$, the sequence of laws of $(\eta_x)_{x\in
    [-K,K]}$ under $\mu^{L_n}_t$ is tight by Lemma \ref{reloumec}
  $(iii)$ and hence we can extract a converging subsequence.  By
  diagonal extraction it is possible to extract a subsequence $L'_n$
  of $L_n$ and a family of measures $(\mu_K)_{K\geq 0}$ on
  $\bbZ^{[-K,K]}$ such that the law of $(\eta_x)_{x\in [-K,K]}$ under
  $\mu^{L'_n}_t$ converges to $\mu^K$ for all $K$.  By construction for
  $H\geq K$, $\mu_H$ projected on $\bbZ^{[-K,K]}$ is equal to $\mu_K$
  and by Kolmogorov extension theorem, there exists a measure $\mu$ on
  $\Z^\Z$ such that $\mu$ projected on $\bbZ^{[-K,K]}$ equals $\mu_K$ for all
  $K$.  One has therefore for all local function $f$
\begin{equation}
\lim_{n\to\infty} \mu^{L'_n}_t(f)=\mu(f).
\end{equation}
We have to show that $\mu$ can be written as $\rho^\nu$.  First one
remarks that $\mu^{L_n}_t$ is translation invariant, so that $\mu$ is
too.  A second point to make is that $\mu$ -almost surely all the
$\eta_x$ (that are not equal to zero) have the same sign.  Indeed
\begin{multline}
\mu\left(\exists x, x'\in \bbZ, \eta_x\eta_x'<0\right)=\lim_{K\to\infty} \mu\left(\exists x, x'\in [-K,K], \eta_x\eta_x'<0\right)\\
=\lim_{K\to\infty}\lim_{n\to\infty}\mu_t^{L'_n}\left(\exists x, x'\in [-K,K], \eta_x\eta_x'<0\right)
\end{multline}
and 
\begin{multline}
\mu_t^{L}\left(\exists x, x'\in [-K,K], \eta_x\eta_x'<0\right)
\\
=\frac{1}{tL^2(2L+1)}\int_0^{L^2t}\sum_{y=-L}^L\bbP\left[\exists x, x'\in [-K+y,K+y], \eta_x(s)\eta_{x'}(s)<0\right]\dd s.
\end{multline}
One realizes easily that
\begin{equation}
 \sum_{y=-L}^L\ind_{\{\exists x, x'\in [-K+y,K+y], \eta_x\eta_{x'}<0\}}
\end{equation}
is upper bounded by $(2K+1)$ times the number of 
changes of sign 
in $(\eta_x)_{x\in[-L,L+1]}$.  From the definition of the dynamics, a
transition can only lower the number of changes of sign.
 Its initial value is smaller than the
number of changes of monotonicity of $\phi^0$ (which is assumed to be finite) plus one (the ``plus one'' can  come 
from periodizing).  Therefore
\begin{equation}
\sum_{y=-L}^L\bbP\left[\exists x, x'\in [-K+y,K+y], \eta_x(s)\eta_{x'}(s)<0\right]\le 2K C(\phi^0).
\end{equation}

A third point is to show is that $\mu$ is an invariant measure for the
infinite volume dynamics (the infinite volume version of
\eqref{dnaspo2}, call its generator $\mathcal L^\infty$).  For $f$ a
$K$-local function one has (for $L\ge K$ large enough)
\begin{equation}
\mu_t^L (\mathcal L^{\infty} f)=\frac 1 {tL^2}  \frac{1}{2L+1}\int_0^{L^2 t} \sum^{L}_{y=- L}\bbE\left(\mathcal L ^\infty(f\circ \theta_y)(\eta(s))\right)\dd s.
\end{equation}
For $y\in [-L+K,L-K]$ the infinite volume generator applied to $f$ has the same effect as the finite volume generator so that 
\begin{equation}
\int_0^t \bbE \left[\mathcal L^\infty (f\circ \theta_y)(\eta(s))\right]\dd s = \int_0^t \partial_s \bbE \left[ (f\circ \theta_y)(\eta(s))\right]\dd s 
=  \bbE \left[ (f\circ \theta_y)(\eta(t))-(f\circ \theta_y)(\eta(0))\right].
\end{equation}
Therefore
\begin{multline}
  \mu_t^L (\mathcal L^\infty f) =\frac 1 {tL^2} \frac{1}{2L+1}
  \sum^{L-K}_{y=- L+K} \bbE \left[ (f\circ \theta_y)(\eta(tL^2)) -(f\circ
    \theta_y)(\eta(0))\right]
  \\
  + \frac 1 {tL^2} \frac{1}{2L+1}\int_0^{L^2 t}
  \left(\sum^{-L+K-1}_{y=- L}+\sum^{L}_{y=L-K+1}\right)\bbE \left(\mathcal L^\infty (f\circ
  \theta_y)(\eta(s))\right)\dd s =O(1/L).
\end{multline}
As a consequence,  for any local function
\begin{equation}
\mu(\mathcal L^\infty f)=\lim_{n\to\infty} \mu_t^{L'_n} (\mathcal L^\infty f)=0.
\end{equation}
Restricted on the event $\eta_x$ have all the same sign, $\mathcal L^\infty$ is the generator of the zero-range process with one type of particle
and therefore $\mu$ is a translation invariant measure for the zero-range process. From \cite[Theorem 1.9]{cf:Andjel}
one can write $\mu=\rho^\nu$ for some $\nu$.
By Lemma \ref{reloumec} $(iii)$, under $\mu$, at time zero $\eta$ is dominated by a IID family of geometric variables of mean $\| \partial_x\phi^0\|_\infty$
and so is  $-\eta$. This implies the claim on the support of $\nu$.

\medskip

The second point of  Proposition \ref{wconv} is standard; we include its proof
for completeness.  Given a local $f$ one can extract a subsequence
$L_n$ such that
\begin{equation}
\lim_{n\to \infty} \mu_t^{L_n}(f)=\limsup_{L\to \infty} \mu_t^{L}(f).
\end{equation}
From $L_n$ one can extract a subsequence $L'_n$ such that
$\mu_t^{L'_n}$ converges to $\rho^\nu$ so that
\begin{equation}
\lim_{n\to \infty} \mu_t^{L_n}(f)=\lim_{n\to \infty} \mu_t^{L'_n}(f)=\int \rho^u (f)\nu(\dd u),
\end{equation}
which ends the proof.
\end{proof}

Fix $l$ to be a large fixed integer.
Set for $y\in \bbZ$
\begin{equation}
B_y:=\{1+y,\dots, l+y\}.
\end{equation}
For notational convenience, similarly to $\eta$ with \eqref{periodiz}, one considers now periodized version $(q_x(s))_{s\in \bbZ}$ of $q(s)$ and $(A(x,s))_{x\in \bbZ}$ of $A(\cdot,s)$.

Now, one uses Proposition \ref{wconv} to control each term in $\bbE
\sum A(s,x)$.
\begin{lemma}\label{djspoo}
\begin{equation}
\label{eq:djs}
 \lim_{l\to\infty} \limsup_{L\to\infty} \frac{1}{L^3}\bbE \int_0^{L^2t}  \sum_{y=- L}^{ L}  \left|\left( \frac{1}{l}
 \sum_{x\in B_y}q_x(s)\sign(\eta_x(s))\right)-q_{y}(s)\sigma
 \left(\frac{1}{l}\sum_{x\in B_y} \eta_x(s)\right) \dd s\right|=0.
 \end{equation}
\end{lemma}

\begin{proof}[Proof of Lemma \ref{djspoo}] 
Fix $l>0$.  For $L$ large enough, any all $y\in\{- L,\ldots, L-l\}$, 
\begin{multline}
\left|\left(\frac{1}{l}\sum_{x\in B_y}q_x(s)\sign(\eta_x(s))\right)-q_{y}(s)\sigma\left(\frac{1}{l}\sum_{x\in B_y} \eta_x(s)\right)\right|\\
 \le |q_y(s)| \left|\left(\frac{1}{l}\sum_{x\in B_y}\sign(\eta_x(s))\right)-\sigma\left(\frac{1}{l}\sum_{x\in B_y} \eta_x(s)\right)\right|
+\max_{x\in B_y} |q_x(s)-q_y(s)|.
 \end{multline}
 Moreover, uniformly in $y\in\{- L,\ldots, L-l\}$, as a consequence of Lemma \ref{reloumec} $(ii)$ 
 \begin{equation}
 \max_{y\in\{- L,\ldots, L-l\},x\in B_y, s\ge 0} |q_x(s)-q_y(s)|=O(l/L).
 \end{equation}
  The contribution of $y\in \{L-l+1,L\}$  to the sum
under the integral in \eqref{eq:djs} is $O(l)$.
Therefore summing over $y\in\{- L,\ldots, L\}$, integrating over $s$
and taking expectation one gets
\begin{multline}
\left| \int_0^{tL^2} \bbE \left[\sum_{y=- L}^{ L}\left( \frac{1}{l}
 \sum_{x\in B_y}q_x(s)\sign(\eta_x(s))\right)-q_{y}(s)\sigma
 \left(\frac{1}{l}\sum_{x\in B_y} \eta_x(s)\right)\right]\dd s\right|\\
 \le (\max_{y,s} |q_y(s)|) \left| \int_0^{tL^2} \bbE \left[\sum_{y=- L}^{ L}\left( \frac{1}{l}
 \sum_{x\in B_y}\sign(\eta_x(s))\right)-\sigma
 \left(\frac{1}{l}\sum_{x\in B_y} \eta_x(s)\right)\right]\dd s\right|
 +O(lL^2)\\
 =(\max_{y,s} |q_y(s)|)
 tL^2(2L+1)\mu_t^L\left( \left|\left(\frac{1}{l}\sum_{x\in B_0}\sign(\eta_x)\right)-\sigma\left(\frac{1}{l}\sum_{x\in B_0} \eta_x\right)\right|\right)+O(lL^2)
 \end{multline}
where $\mu_t^L$ is defined in \eqref{eq:defmu}.
Therefore, the proof of our statement is finished provided one  proves 
\begin{equation}
\lim_{l\to \infty} \limsup_{L\to \infty} \mu_t^L\left( \left|\left(\frac{1}{l}\sum_{x\in B_0}\sign(\eta_x)\right)-\sigma\left(\frac{1}{l}\sum_{x\in B_0} \eta_x\right)\right|\right)=0.
\end{equation}
From Proposition \ref{wconv} one has
\begin{multline}\label{erce}
\limsup_{L\to \infty} \mu_t^L\left( \left|\left(\frac{1}{l}\sum_{x\in B_0}\sign(\eta_x)\right)-\sigma\left(\frac{1}{l}\sum_{x\in B_0} \eta_x\right)\right|\right)\\
\le \sup_{0\le u\le \| \partial_x \phi^0\|_\infty} \rho^u 
\left( \left|\left(\frac{1}{l}\sum_{x\in B_0}\sign(\eta_x)\right)-\sigma\left(\frac{1}{l}\sum_{x\in B_0} \eta_x\right)\right|\right)
\end{multline}
and one can check that the right-hand side term tends to zero when $l$
tends to infinity: we note that under $\rho^u$, for every $x$ one has
$\rho^u(\sign(\eta_x))=\sigma(u)$, and the law of large numbers tell
us that the two terms $\frac{1}{l}\sum_{x\in B_y}\sign(\eta_x)$ and
$\sigma\left(\frac{1}{l}\sum_{x\in B_y} \eta_x\right)$ have the same
limit when $l$ tends to infinity.  However, because of the $\sup$ over
$u$ one needs more quantitative estimates than the law of large
numbers to conclude. For instance we can get them by the use of second
moment method; we leave the details to the reader.
\end{proof}

Similarly to Lemma \ref{djspoo}
one shows that
\begin{lemma}\label{leptilem}
\begin{equation}
\lim_{l\to\infty} \limsup_{L\to\infty} \frac{1}{L^3} \int_0^{L^2t} \sum_{y=- L}^{ L} \bbE
(G(\eta(s)))= \lim_{l\to\infty} \limsup_{L\to\infty}
\frac{t(2L+1)}L\mu^L_t(G(\eta))=0
\end{equation}
where
\begin{eqnarray}
G(\eta)= \left|\left(\frac{1}{l}\sum_{x\in B_y} 
 |\eta_x(s)|-|\sign(\eta_x(s)| \right)- 
  \left(\frac{1}{l}\sum_{x\in B_y}   \eta_x(s)\right) \sigma\left( \frac{1}{l}\sum_{x\in B_y}   \eta_x(s) \right)\right|.
\end{eqnarray} 
\end{lemma}

\begin{proof}
The proof is very similar to that of Lemma \ref{djspoo}, the
only additional technical point being that
the function $G(\eta)$
is not bounded so that one cannot use directly Proposition \ref{wconv}. However stochastic domination given by Lemma \ref{reloumec} $(iii)$
allows us to get the same conclusion by considering the function 
$\eta\mapsto \min (G(\eta),K)$, and letting  $K$ tend to infinity afterwards.
Altogether one gets
\begin{eqnarray}
\limsup_{L\to \infty} \mu_t^L(G)
 \le \sup_{0\le u\le \| \partial_x \phi^0\|_\infty} \rho^u(G).
\end{eqnarray}
We end the proof in the same way that for the previous Lemma remarking that 
\[\rho^u ( |\eta_x|-|\sign(\eta_x)|)=u\sigma(u).\]
\end{proof}

Now we are ready to conclude:
\begin{multline}
\label{eq:factorisation}
\sum_{y=- L}^{ L} A(y,s)\\
=\sum_{y=- L}^{ L}\left\{ -q_y(s)\sigma(q_y(s))+\frac{1}{l}\left(\sum_{x\in B_y} \eta_x(s)\sigma(q_x(s))+
q_x(s)\sign(\eta_x(s))-(|\eta_x(s)|-|\sign(\eta_x(s)|)\right)\right\}\\ \le \sum_{y=- L}^{ L}
-q_y(s)\sigma(q_y(s)+\left( \frac{1}{l}\sum_{x\in B_y}\eta_x(s)\right)\sigma(q_y(s))
+ q_y(s)\sigma\left(\frac{1}{l}\sum_{x\in B_y}\eta_x(s)\right)\\
-\left( \frac{1}{l}\sum_{x\in B_y}\eta_x(s)\right)\sigma\left(\frac{1}{l}\sum_{x\in B_y}\eta_x(s)\right)
+R(s,l,L)\\
=R(s,l,L)-\sum_{y=- L}^{ L}
\left[q_y(s)-\left( \frac{1}{l}\sum_{x\in B_y}\eta_x(s)\right)\right]\left[\sigma(q_y(s))-\sigma\left(\frac{1}{l}\sum_{x\in B_y}\eta_x(s)\right)\right]
\end{multline}
where 
\begin{multline}
R(s,l,L)= -\sum_{y=- L}^{ L}\frac{1}{l}\left(\sum_{x\in B_y} \eta_x(s)(\sigma(q_y(s))-\sigma(q_x(s)))\right)\\
+  \left| \left( \frac{1}{l}
 \sum_{x\in B_y}q_x(s)\sign(\eta_x(s))\right)-q_{y}(s)\sigma
 \left(\frac{1}{l}\sum_{x\in B_y} \eta_x(s)\right) \right|\\
+ \left|\left(\frac{1}{l}\sum_{x\in B_y} 
 |\eta_x(s)|-|\sign(\eta_x(s)| \right)- 
\left(\frac{1}{l}\sum_{x\in B_y}   \eta_x(s)\right) \sigma\left( \frac{1}{l}\sum_{x\in B_y}   \eta_x(s) \right)\right|
\end{multline}
and the second term is non positive ($a-b$ and $\sigma(a)-\sigma(b)$ have the same sign).

According to $(ii)-(iii)$ in Lemma \ref{reloumec} (to control the first term) and Lemmata \ref{djspoo} and \ref{leptilem}
\begin{equation}
\lim_{l\to\infty} \limsup_{L\to 0}  \frac{1}{L^3}\int_0^{L^2t} \bbE \ R(s,l,L)\dd s =0.
\end{equation}
This implies
\begin{equation}
\limsup_{L\to\infty} \frac{1}{L^3}\int_0^{L^2t}\sum_{y=- L}^{ L} \bbE A(x,s)\dd s\le 0
\end{equation}
and therefore the result that we want to prove,  from \eqref{reza}.

\end{proof}

\subsection{Concluding the proof of Theorem \ref{spo}}
\label{sec:concluding}
It is not hard to transform the $\bbL_2$ statement of Proposition
\ref{prop:propspohn} into the desired ``almost sure'' statement:
\begin{proposition}
For any $\gep>0$, $t\ge 0$, w.h.p
\begin{equation}
\max_{x\in \{-L,\dots,L+1\}}\frac{1}{L}|\Phi(x,L^2 t)- \hat h_x(L^2 t)|\le \gep.
\end{equation}
\end{proposition}

\begin{proof}
Also here we write $h$ for $\hat h$.
Note that from Lemma \ref{reloumec} $(iii)$ the random vector
$(|\eta_x(t)|)_{x\in\{ -L,\dots,L\}}$ is stochastically dominated for
every $t$ by a
vector of IID time-independent geometric variables.
This implies that there exists a constant $C$ such that for any $t\ge
0$, w.h.p.
\begin{equation}\label{artichaud}\begin{split}
 |h_x(t)-h_y(t)|\le C |x-y| \quad \text{for every\;\;} x, y \in \{ -L,\dots,L\}, |x-y|\ge \log L
\end{split}
\end{equation}
(this can be proved 
by using large deviation estimates and a union bound on  $ x, y \in \{
-L,\dots,L\}$). Moreover 
Lemma \ref{reloumec} $(i)$ ensures that  
$\Phi(\cdot,t)$ is always Lipschitz so that \eqref{artichaud} holds also for $\Phi(\cdot, t)-h_\cdot(t)$.

With \eqref{artichaud} and $L$ large enough, one has
\begin{gather}
\left\{\max_{x\in \{-L,\dots,L+1\}}|\Phi(x,L^2 t)- h_x(L^2 t)|\ge \gep L\right\}\subset
 \left\{\sum_{x\in \{-L,\dots,L+1\}}|\Phi(x,L^2 t)- h_x(L^2 t)|^2\ge \frac{\gep^3 L^3}{10 C}\right\}
\end{gather}
so that the left-hand side event has  small probability when $L$ is large, otherwise
Proposition \ref{prop:propspohn} would be false.
\end{proof}

\subsection{Laplacian bounds}
\label{sub:lb}
Recall that $\Phi(x,t)$ is the solution of the Cauchy problem
\eqref{eq:edpdiscret}. We want to bound  $\Phi(x,t)$ above and below
with the solution of a suitable heat equation. For
this, we will suppose that the function $\phi^0$, through which the
initial condition $\Phi_0$ for $\Phi(x,t)$ is defined, is concave on
$[-1,1]$ (in addition to the assumptions required for Theorem \ref{spo}).
One defines the evolution $\Phi_1(x,t)$ as the solution of  
\begin{equation}
\begin{cases}
 \partial_t \Phi_1(x,t)&=  \frac12\gD \Phi_1(x,t)\\
 \Phi_1( -L, t)&=\Phi( L+1,t)=0 \\ 
  \Phi_1(x, 0)&=\Phi_0(x)
\end{cases}
\end{equation}
for $t\ge0,  x\in\{- L+1,  L\}$. Also we define $\Phi_2(x,t)$ as the
solution of the analogous equation (with the same boundary values)
where the discrete Laplacian is multiplied by 
$(1/2)\sigma'(\| \partial_x \phi^0\|_\infty)=1/(1+\| \partial_x \phi^0\|_\infty)^2$.
\begin{proposition}\label{proplaplb}
For every $t\ge 0$ every $x\in \{- L,\dots,  L +1\}$ one has
\begin{equation}\label{trare}
 \Phi_1(x, t)\le  \Phi(x, t)\le  \Phi_2(x, t).
\end{equation}
\end{proposition}

\begin{proof}
We prove the upper bound, the lower one being very similar.
Suppose that the result does not hold and set
\begin{equation}
T:=\max\{t \ | \    \Phi(x, t)\le  \Phi_2(x, t)
\mbox{\;for every\;} t\le T,x\in \{- L,\dots,  L +1\}\}.
\end{equation}
Note that by property of the heat-equation, $\Phi_2(x, t)$ is a
strictly concave function of $x$ for all positive $t$ (except in the case where
one starts from the flat initial condition but in that case the
statement is trivial).
Let $x_0$ be such that
\begin{equation}
\Phi(x_0, T)= \Phi_2(x_0, T).
\end{equation}
Then one remarks that 
$q_{x_0}(T)-q_{x_0-1}(T)<0$
(by strict concavity of  $\Phi_2(\cdot, T)$) and that by Lemma
\ref{reloumec} $ \max_x|q_x(t)|\le \| \partial_x \phi^0\|_\infty$  so that 
\begin{equation}
\sigma(q_{x_0}(T))-\sigma(q_{x_0-1}(T))<(q_{x_0}(T)-q_{x_0-1}(T))\sigma'(  \| \partial_x \phi_0\|_\infty)
\end{equation}
(since $\sigma'(\cdot)$ is decreasing on $\bbR^+$) and hence
\begin{multline}
2\partial_t [\Phi_2-\Phi](x_0, T)=
\sigma'(\| \partial_x \phi_0\|_\infty)\gD\Phi_2(x,t)-\sigma(q_{x_0}(T))+\sigma(q_{x_0-1}(T))\\
>\sigma'(\| \partial_x \phi_0\|_\infty)[(\Phi_2(x+1,t)+\Phi_2(x-1,t))-(\Phi(x+1,t)+\Phi(x-1,t))].
\end{multline}
Since the last expression is non-negative, one has $\Phi(x, t)<\Phi_2
(x, t)$ on an interval $[T,T+\gep(x)]$ for some $\gep(x)>0$, for every
$ x\in \{- L,\dots, L +1\}$ and that concludes the proof since the
only possibility is that $T=\infty$.
\end{proof}

\section*{Acknowledgments}

The authors would like to thank P.\ Cardaliaguet for enlightening
discussions and valuable comments on various analytical aspects of the
present work, and the referee for a careful reading of the manuscript. During part of the writing, H.L.\ was 
hosted by Instituto Nacional de Matemática Pura e Aplicada; he acknowledges kind hospitality and support.

\end{document}